%% file: main2.tex
\documentclass[12pt]{amsart}
\usepackage{amssymb,color}
\usepackage{amsmath} 
\usepackage{amsthm} 
\usepackage{calc} 
\usepackage{verbatim} 
\usepackage{bbm}
\usepackage{enumitem}
\usepackage{hyperref}

\input{shortcuts.tex}

\input{macros.tex}

\title{Energy, equilibrium measure and entropy for toric surface maps}
\author{Jeffrey Diller}
\address{Department of Mathematics\\
  University of Notre Dame\\
  Notre Dame, IN 46556\\
  USA}
\email{diller.1@nd.edu}

\author{Roland Roeder}

\address{Department of Mathematics\\ Indiana University Indianapolis \\
         402 North Blackford Street room LD270\\
         Indianapolis, Indiana 46202-3267 \\
         USA}
\email{roederr@iu.edu}


\begin{document}
\begin{abstract}
We consider the ergodic theory of plane rational maps that preserve the natural holomorphic volume form on the algebraic torus.  Specifically we construct natural invariant probability measures for a large class of such maps by intersecting the equilibrium currents we constructed in our previous work \cite{DiRo24}.  We show further that these measures are mixing and that each admits an underlying geometric product structure.  The main result of \cite{DDG11b} then implies that the topological entropy of each map covered by our results is the log of its first dynamical degree.  In light of examples presented in \cite{BDJ20}, this implies in particular that the entropy of a rational map can equal the log of a transcendental number.   
\end{abstract}

\maketitle
\markboth{\today}{\today}

\section{Introduction}
\label{sec:intro}
\input{intro.tex}
\section{Inverse limits of toric surfaces, divisors and currents}
\label{sec:backgd}
\input{rzt2.tex}

\section{Toric maps}
\label{sec:maps}
\input{maps4.tex}

\section{Pushing forward and pulling back}
\label{sec:currents}
\input{eqcurrents4.tex}

\section{A continuity result}
\label{sec:currents_II}
\input{eqcurrents_continued2.tex}

\section{Wedge products and energy}
\label{sec:products}
\input{products7.tex}

\section{Equilibrium measure: construction}
\label{sec:measure}
\input{measure4.tex}

\section{Equilibrium measure: dynamics and geometry}
\label{sec:mixing}
\input{moremeasure3.tex}

%

\nocite{}

\end{document}

%% file: shortcuts.tex
\newtheorem{thm}{Theorem}[section]
\newtheorem{lem}[thm]{Lemma}
\newtheorem{cor}[thm]{Corollary}
\newtheorem{prop}[thm]{Proposition}

\theoremstyle{definition}
\newtheorem{defn}[thm]{Definition}

\theoremstyle{remark}

\newtheorem{rem}[thm]{Remark}

\newcounter{probno}

\renewcommand{\qed}{\hfill$\Box$ \par \medskip}

\newcommand{\dist}{\mathrm{dist}}

\newcommand{\tto}{\dashrightarrow}

\newcommand{\norm}[2][{}]{\left\|#2\right\|_{#1}}

\newcommand{\pair}[2]{\left\langle #1,#2 \right\rangle}

\newcommand{\R}{\mathbf{R}}

\newcommand{\C}{\mathbf{C}}
\newcommand{\cp}{\mathbf{P}}
\newcommand{\Z}{\mathbf{Z}}
\newcommand{\N}{\mathbf{N}}

\newcommand{\supp}{\mathrm{supp}\,}



%% file: macros.tex
\newtheorem*{ackn}{Acknowledgment}

\newcommand{\bP}{\mathbf{P}}
\newcommand{\torus}{\mathbb{T}}
\newcommand{\nvec}{v}

\renewcommand{\div}{\operatorname{Div}}

\newcommand{\ch}[2][{}]{[#2]_{#1}}        
\newcommand{\pcc}{\mathcal{D}_{1,1}}       

\newcommand{\sfn}{\psi}

\newcommand{\ind}{\mathrm{Ind}}
\newcommand{\exc}{\mathrm{Exc}}
\newcommand{\dtop}{\lambda_2}          
\newcommand{\ddeg}{\lambda_1} 
\newcommand{\tdeg}{\lambda} 
\newcommand{\hoo}{H^{1,1}_\R} 
\newcommand{\eltwo}{\mathcal{L}^2}

\newcommand{\Log}{\operatorname{Log}}

\newcommand{\ram}{\operatorname{Ram}}

\newcommand{\rzt}{\hat\torus}
\newcommand{\rztO}{\hat\torus^\circ}

\newcommand{\trop}[1]{\mathrm{A}_{#1}}
\newcommand{\tropf}{\trop{f}}
\newcommand{\isect}[2]{\left({#1}\cdot{#2}\right)}
\renewcommand{\N}{\mathbf{Z}_{\geq 0}}
\newcommand{\oone}{\mathbbm{1}}
\newcommand{\ti}{\tilde}
\newcommand{\nrg}[1]{\langle #1 \rangle}
\newcommand{\dnorm}[2][]{\norm[#1]{#2}'}
\newcommand{\rpot}{\varphi}

\DeclareMathOperator{\dpsh}{DPSH}
\DeclareMathOperator{\psh}{PSH}
\DeclareMathOperator{\qpsh}{QPSH}
\DeclareMathOperator{\err}{Err}
\DeclareMathOperator{\eclass}{\mathcal{E}^1}
\DeclareMathOperator{\weclass}{\mathcal{E}^1_{wk}}

\DeclareMathOperator{\htop}{h_{top}}
\DeclareMathOperator{\hmu}{h_{\mu}}

%% file: intro.tex
%

Let $\torus \cong (\C^*)^2 \subset \bP^2$ denote the two dimensional complex algebraic torus and $\eta := \frac{dx_1\wedge dx_2}{x_1x_2}$ denote the natural $\torus$-invariant holomorphic two form on $\torus$.  Here we call a rational self-map $f:\bP^2\tto\bP^2$ \emph{toric} (short for `toric volume preserving') if it satisfies $f^*\eta = \rho\eta$ for some constant $\rho = \rho(f)\in\C^*$.  Our goal in this article and its predecessor \cite{DiRo24} is to shed light on the ergodic theory of such maps, especially in cases where $f$ exhibits pathological degree growth.

Let $\deg(f) := \deg f^* L$ be the degree of the divisor obtained by pulling back a line $L\subset\bP^2$.  The \emph{(first) dynamical degree} of $f$ is then the asymptotic growth rate
$$
\ddeg(f) := \lim_{n\to\infty} \deg(f^n)
$$
of $\deg(f^n)$.  
Though the limit defining $\ddeg(f)$ always exists, it was observed in \cite{DiLi16} (see Theorem F and Example 8.5) that the degree sequence $(\deg(f^n))_{n\geq 0}$ can behave erratically.  This point was amplified in \cite{BDJ20}, which gave explicit examples of toric maps $f$ for which $\ddeg(f)$ is a transcendental number.  

The first dynamical degree $\ddeg(f)$ should be compared with the \emph{topological degree}  $\dtop(f)$ of $f$, i.e. the number of preimages of a general point in $\bP^2$.  We will focus on the case when $f$ has \emph{small topological degree}, i.e. $\dtop(f) < \ddeg(f)$.  A central motivation for this paper, and an easily stated consequence of our main result is the following.

\begin{thm}
\label{thm:punchline}
Suppose that $f:\bP^2\tto\bP^2$ in \cite[Main Theorem]{BDJ20} has small topological degree.  Then the topological entropy of $f$ is given by $\htop(f) = \log\ddeg(f)$.  In particular, $e^{\htop(f)}$ can be transcendental.  
\end{thm}

The foundational result \cite[Th\'eor\`eme 1]{DiSi05} gives $\htop(f)\leq \log\ddeg(f)$ for \emph{any} plane rational map with small topological degree.  Here we construct natural invariant measures $\mu$ of maximal metric entropy $\hmu(f) = \log\ddeg(f)$ for a large class of toric maps including the ones in Theorem \ref{thm:punchline}.  The opposite inequality $\htop(f) \geq \hmu(f) =\log\ddeg(f)$ then follows as in the proof of the classical variational principle \cite[Theorem 4.5.3]{KaHa95}.  

The series of papers \cite{DDG10, DDG11, DDG11b} are an important precedent for us.  They construct and analyze measures of maximal entropy for a broad class of (not necessarily toric) rational self-maps of $\bP^2$ with small topological degree.  For instance, they establish the analogue of Theorem \ref{thm:punchline} for polynomial maps $f:\C^2\to\C^2$.  However, these articles limit attention to maps that are `algebraically stable', or at least those that become so after suitably modifying the domain $\bP^2$.  The toric maps that we consider here all fail the stabilizability hypothesis, and this leads us to lean much more heavily on the toric context to accomplish our aims.  In the remainder of this introduction, we present and provide context for our main result, describing among other things the particular class of maps to which it applies and discussing the main ingredients of the proof.

It is important for our arguments to consider a toric map $f$ not only as a self-map of $\bP^2$ but also a rational map $f:X\tto Y$ between any two toric surfaces $X,Y\to\bP^2$ obtained by $\torus$-invariant blowups of $\bP^2$.  A lift $f_X:X\tto X$ of $f$ to some particular surface $X$ is called \emph{algebraically stable} if the induced pullback operator $f^*:\hoo(X)\to\hoo(X)$ satisfies $(f^*)^n = (f^n)^*$ for all $n \geq 0$.  To cope with unstabilizable maps, we introduce in \S\ref{sec:backgd} the inverse limit $\rzt$ of all toric blowups of $\bP^2$.  The \emph{toric limit space} $\rzt$ is compact and Hausdorff, and it decomposes into a closed nowhere dense subset of $\torus$-invariant points and the \emph{toric limit surface} $\rztO$, a (non-compact, non-algebraic) complex manifold obtained from $\torus$ by adjoining countably many mutually disjoint \emph{poles}, i.e. one dimensional complex tori $C_\tau \cong\C^*$ indexed by the rational rays $\tau\subset\R^2$.  All other properly embedded complex curves $C\subset \rztO$ are compact with $C\setminus\torus$ finite, so we call them \emph{internal}.

For many purposes $\rztO$ is as natural a domain for a toric map as any particular $\torus$-invariant blowup $X\to \bP^2$.  Specifically, a toric map $f$ induces a meromorphic self-map of $f:\rztO\tto\rztO$ whose important properties include the following.
\begin{itemize}
 \item The indeterminacy set $\ind(f)\subset \rzt^\circ$ of $f$ is finite, and the image $f(p)$ of each $p\in\ind(f)$ is a finite union of internal curves;
 \item The exceptional set $\exc(f)\subset\rzt^\circ$ of $f$ consists of finitely many internal curves $C$, each contracted to a point $f(C)\in\rztO\setminus\torus$;
 \item There is a positively homogeneous, continuous and piecewise linear map $\tropf:\R^2~\to~\R^2$ of $f$ such that $f(\Z^2) \subset \Z^2$ and $f(C_\tau) = C_{\tropf(\tau)}$ for every pole $C_\tau\subset\rztO$. 
\end{itemize}
We call $\tropf$ the \emph{tropicalization} of $f$.  In \S\ref{sec:maps} we explain (see Theorem \ref{thm:tropapprox}) that beyond merely encoding the action of $f$ on poles, $\tropf$ serves as a good approximation of $f|_\torus$; i.e. away from the exceptional set $\exc(f)$ the natural map $\Log:\torus\to\R^2$ approximately semiconjugates $f|_\torus$ to $A_f$.  This fact is central to our analysis.

All examples of toric maps that we know exhibit two further properties, which we will adopt as assumptions below: $\ind(f)\cap\torus = \emptyset$; and more importantly, $\tropf$ is a homeomorphism.  Homogeneity implies that $\tropf$ descends to a self-map of the unit circle; hence $\tropf$ has a well-defined rotation number whenever it is a homeomorphism.  It was shown in \cite{DiLi16} that the rotation number of $\tropf$ is rational if and only if there is a smooth rational surface $X$ with a birational morphism $X\to \bP^2$ that lifts $f$ to an \emph{algebraically stable} self-map of $X$.

In contrast, even when $\tropf$ is a homeomorphism with irrational rotation number, $f$ often satisfies a weaker condition that we call `internal stability.'  Let $\pcc^+(\rzt)$ denote the cone of positive closed $(1,1)$-currents on $\rztO$ and $\pcc(\rzt) = \pcc^+(\rzt) - \pcc^+(\rzt)$ denote its linear span.  Let $\hoo(\rzt)$ denote the (infinite dimensional) quotient obtained by identifying $S,T\in\pcc(\rzt)$ when $S-T$ is $dd^c$-exact.  As with rational maps on smooth compact surfaces, a toric map $f$ induces linear pushforward/pullback operators $f_*$ and $f^*$ on $\pcc(\rzt)$ and $\hoo(\rzt)$.  We call $f$ \emph{internally stable} if, again, $(f^*)^n = (f^n)^*$ for all $n\in\N$.  Equivalently, $f^n(\exc(f)) \cap\ind(f)~=~\emptyset$ for all $n\geq 0$.

We can state our main result as follows.

\begin{thm}
\label{thm:mainthm}
Let $f$ be an internally stable toric map with small topological degree.  Assume further that $\tropf$ is a homeomorphism with irrational rotation number and that $\ind(f)\cap\torus = \emptyset$.  Then there is a Borel probability measure $\mu$ on $\rztO$ with the following properties.
\begin{enumerate}
 \item $\mu(C) = 0$ for every properly embedded curve $C\subset \rztO$. 
 \item $\mu$ is $f$-invariant and mixing.
 \item The metric entropy of $f$ relative to $\mu$ is given by $\hmu(f) = \log\ddeg(f)$. 
\end{enumerate}
Consequently, the topological entropy $f$ is also given by $\htop(f) = \log\ddeg(f)$.
\end{thm}

\noindent 
Conclusion (1) and the fact that $\rztO\setminus\torus$ is a countable union of poles imply that the domain of the map $f$ in Theorem \ref{thm:mainthm} is flexible; i.e. the theorem remains valid if we consider $f$ as a self-map of the algebraic torus $\torus$, of $\cp^2$, of $\rztO$ or of some other compact toric surface.  

\subsection{Proof of the main result}
To prove Theorem \ref{thm:mainthm}, we follow a general approach \cite{BeSm91} created by Bedford and Smillie to understand the complex dynamics of plane polynomial automorphisms.  By Theorems 10.1 and 10.6 in \cite{DiRo24}, a toric map satisfying the hypotheses of Theorem \ref{thm:mainthm} admits a pair of \emph{equilibrium currents} $T^*,T_*\in\pcc^+(\rzt)$.  These are distinguished by invariance $f^*T^* = \ddeg T^*, f_*T_* = \ddeg T_*$ and the fact (Corollary \ref{cor:nidinfinity}) that neither dominates the integration current associated to any curve $C\subset\rztO$.  In particular, both are \emph{internal}, with no mass outside $\torus$.  We normalize so that the classes $\ch{T^*},\ch{T_*}\in\hoo(\rzt)$, have intersection numbers $\ch{T^*}^2 = \ch{T^*}\cdot\ch{T_*}=1$.  Then we realize $\mu$ as an `intersection' $T^*\wedge T_*$ of $T^*$ and $T_*$.  However, $T^*$ and $T_*$ are distributional rather than pointwise objects, so we need to take a great deal of care when defining and employing this intersection.  There are three main steps in our approach.

\subsubsection{Regularity of potentials} The first step is to prove the following, a consequence of Theorem \ref{THM:CONTINUITY} below.

\begin{thm}
\label{thm:ctyresult}
Let $f$ be as in Theorem \ref{thm:mainthm} and $p\in\rztO$ be a point whose forward orbit $\{f^n(p):n\geq 0\}$ does not meet $\ind(f)$.  Then there is a neighborhood $U\ni p$ and a continuous function $u:U\to\R$ such that $T = dd^c u$ near $p$.
\end{thm}

Theorem \ref{thm:ctyresult} guarantees that $uT_*$ is a well-defined current on
$U$ and allows us to use the classical idea of Bedford and Taylor
\cite{BeTa76}, setting $T^*\wedge T_*|_U$ equal to the weak derivative
$dd^c(uT_*)$.  The hypotheses of Theorem \ref{thm:mainthm} ensure that the
countable set of points $\ind(f^\infty) := \bigcup f^{-n}(\ind(f))$ excluded in
Theorem \ref{thm:ctyresult} is closed and discrete in $\rztO$.  Hence an
integrability result (see Proposition \ref{prop:boundedish}) for potentials
with isolated singularities implies that this definition of $T^*\wedge T_*$
applies even near points in $\ind(f^\infty)$.  

The proof of Theorem \ref{thm:ctyresult} occupies \S\ref{sec:currents_II} and relies heavily on some key preliminary results.  The first (see Theorem \ref{thm:homogenization} and the preceding discussion) is that the class $\ch{T}\in\hoo(\rztO)$ of any internal current $T\in\pcc(\rzt)$ includes a canonical `homogeneous' representative $\bar T\in\pcc(\rzt)$.  By definition of $\hoo(\rzt)$, we have $T-\bar T = dd^c\rpot_T$ for some function $\rpot_T\in L^1_{loc}(\rztO)$.  The potential $\rpot_T$ is unique up to additive constants, and homogeneous currents have continuous local potentials about any point in $\rztO$ (Proposition \ref{prop:bartonpole}).  Hence $T$ has continuous local potentials on an open set $U\subset\rztO$ if and only if $\rpot_T|_U$ is continuous.  In the case of interest $T=T^*$, we constructed $T^*$ in \cite{DiRo24} by proving $L^1_{loc}$-convergence
$$
\rpot_{T^*} = \lim_{n\to\infty} \rpot_{\ddeg^{-n} f^{n*} \bar T^*}.
$$
Moreover, the potentials $\rpot_{\ddeg^{-n} f^{n*}\bar T^*}$ on the right are all continuous away from $\ind(f^\infty)$.  To prove Theorem \ref{thm:ctyresult}, we use our tropical approximation result Theorem \ref{thm:tropapprox} to go further and demonstrate that the limit is achieved uniformly on compact subsets.  The starting point is Theorem \ref{thm:pullback} which implies that though $\rztO$ is not compact, the potentials $\rpot_{\ddeg^{-n}f^{n*}\bar T^*}$ are decreasing in $n$ and uniformly bounded off any neighborhood of $\ind(f^\infty)$.

\subsubsection{Mass and energy}
The second step in constructing the measure $\mu = T^*\wedge T_*$ in Theorem~\ref{thm:mainthm} is to show the result is a probability measure, i.e. that in accord with the normalization $\isect{\ch{T^*}}{\ch{T_*}}=1$, we get total mass $\mu(\rztO) = 1$.  The difficulty here is that $\rztO$ is not compact, so that the relationship between mass and intersection number is not (or at least not obviously) as close as in the compact setting.  On the other hand, the currents $T^*,T_*$ restrict to well-defined positive closed $(1,1)$ currents on any given toric surface $X$.  The intersection number $\isect{\ch[X]{T^*}}{\ch[X]{T_*}}$ and mass of the product $T^*\wedge_X T_*$ of the restrictions on $X$ are easily seen to agree.  Moreover, the intersection numbers $\isect{\ch[X]{T^*}}{\ch[X]{T_*}}$ converge by definition to $\isect{\ch{T^*}}{\ch{T_*}}$ as $X$ converges to $\rzt$.  So the problem becomes to establish convergence of measures $T^*\wedge_X T_* \to T^*\wedge T_*$.  More precisely, we have a natural inclusion $X^\circ\hookrightarrow\rztO$ of the open set $X^\circ\subset X$ obtained by removing $\torus$-invariant points, and the measures $T^*\wedge_X T_* |_{X^\circ}$ increase to $T^*\wedge T_*$ as 
$X^\circ$ increases to $\rztO$.  However, $T^*\wedge_X T_*$ assigns non-zero mass to each of the $\torus$-invariant points of $X$, so we need to argue that the sum of those masses diminishes to zero as $X$ increases.

It is fairly straightforward to do this if we replace $T^*$ and $T_*$ with their homogeneous counterparts $\bar T^*$ and $\bar T_*$ (see Conclusion (3) in Theorem \ref{thm:wedgewhomogeneous}).  In that case, the  limiting measure $\bar T^*\wedge \bar T_*$ on $\rztO$ is just Haar measure on the maximal real subtorus $\torus_\R\subset \torus$.  The next step (Conclusion (4) in \ref{thm:wedgewhomogeneous}) is to show that convergence succeeds for $\bar T^*\wedge T_*$, i.e. if we homogenize only $T^*$ in the product.  Here we employ Theorem \ref{thm:eqnearhom} which says that $T_*$ is not in some sense too far from its homogenization (see Definition \ref{defn:nearlyhomogeneous}).  This allows us to relate point masses of $T^*\wedge_X \bar T_*$ at $\torus$-invariant points of $X$ directly to Lelong numbers of $\bar T^*$ and use that the Lelong numbers disappear in the limit (Proposition \ref{prop:smalllelong}).  The most difficult step is to argue from $\bar T^*\wedge T_*$ to $T^*\wedge T_*$.  The main ingredient for this is the following.

\begin{thm}
\label{thm:weakenergy1}
The potential $\rpot_{T^*}$ for $T^*-\bar T^*$ has weakly finite $T_*$-energy on $\rztO$.
\end{thm}

Definitions \ref{defn:WEAKLY_FINITE_ENERGY} and \ref{DEF:WEAKLY_FINITE_ENERGY_ON_RZTO} specify the meaning of weakly finite energy on a compact surface $X$ and on $\rztO$, respectively.  This notion builds on the idea of Dirichlet energy relative to a positive closed current (see \eqref{EQN:DEF_SMOOTH_PAIRING} and \eqref{eqn:epairing} below) that was introduced to complex dynamics in \cite{BeDi05b} and enlarged upon in \cite{DDG11}.  The main novelty here is that the notion of Dirichlet energies defined in those papers applies only to functions that are not far from plurisubharmonic, whereas our more delicate notion applies to \emph{differences} $u-v$ of plurisubharmonic functions even when neither $u$ nor $v$ separately have finite Dirichlet energy.  In any case, we use Theorem \ref{thm:weakenergy1} to tie point masses of $(T^*-\bar T^*) \wedge_X T_*$ to Lelong numbers of $T_*$ at $\torus$-invariant points which can, again, be seen to disappear as $X$ increases to $\rzt$.

\subsubsection{Geometry and dynamics of $\mu$}
Beyond its use in establishing the mass of $T^*\wedge T_*$, Theorem \ref{thm:weakenergy1} also plays an important role in \S\ref{sec:mixing} where we prove the three enumerated conclusions of Theorem \ref{thm:mainthm}.  Our arguments for invariance and mixing follow precedents in \cite{BeSm91} and \cite{DDG11}.  The computation of $h_\mu(f)$ depends on Conclusions (1) and (2) together with a geometric structure theorem for $\mu$ modeled on \cite[Theorem 5.2]{Duj06}.  To state the theorem, we recall from \cite[Theorem 1.4]{DiRo24} that $T^*$ can be approximated arbitrarily well by so-called `uniformly laminar' currents $T^*_\epsilon \leq T^*$, obtained by averaging currents of integration over a family of mutually disjoint analytic disks.  Likewise we have arbitrarily good approximations $T_{*,\epsilon}\leq T_*$ by `uniformly woven' currents, obtained by averaging a family of disks that are allowed to intersect each other.  The intersections $T^*_\epsilon \wedge T_{*,\epsilon} \leq T^*\wedge T_*$ therefore have a natural product structure based on intersections between disks in the two families.  This is discussed at greater length in \S\ref{ss:geometric} where we prove

\begin{thm}
\label{thm:geometric}
The wedge product $T^*\wedge T_*$ is `geometric', equal to an increasing limit of the measures $T^*_\epsilon\wedge T_{*,\epsilon}$ obtained by intersecting the uniformly laminar and uniformly woven approximants of $T^*$ and $T_*$, respectively.
\end{thm}

Given this and Conclusions (1) and (2) of Theorem \ref{thm:mainthm}, the entropy bound in Conclusion~(3) then follows immediately from the main result Theorem B of \cite{DDG11b}.  We stress that unlike Theorem A in the same paper, Theorem B requires only that $f$ be a rational map with small topological degree and that $\mu$ be an $f$-invariant Borel probability measure that is $f$-invariant and mixing and has the geometric structure guaranteed by Theorem \ref{thm:geometric}. 

\subsection{Context and related work}

Dynamical systems that are stable under perturbation and/or whose dynamics can be described well, e.g. Axiom A diffeomorphisms, often have topological entropy equal to the log of an algebraic number.  On the other hand, in a family of maps where entropy is non-constant and varies continuously with parameters, e.g. $d$-modal piecewise linear interval maps, the Intermediate Value Theorem guarantees that most maps will have entropy equal to the log of a transcendental number.  The dynamics of rational maps $f:\bP^2~\tto~\bP^2$ with given degree $\deg(f) = d$ is very poorly understood at present, but the quantity $\htop(f)$ seems nevertheless quite rigid.  There is substantial evidence supporting the conjecture (see \cite[Conjecture 3.2]{GUEDJ_ENTROPY}) that when $\ddeg(f) \neq \dtop(f)$, one has
\begin{align}\label{EQN:CONJ_FMLA}
\htop(f) = \log\max\{\ddeg(f),\dtop(f)\}.
\end{align}
Since there are only countably many possible values for $\ddeg(f)$ and $\dtop(f)$ (see \cite{BF00}), it is a priori less obvious that there should exist rational maps $f:\bP^2~\tto~\bP^2$ with $e^{\htop(f)}$ transcendental.  Theorem \ref{thm:mainthm} confirms \eqref{EQN:CONJ_FMLA} for a new class of rational maps, and in doing so identifies specific rational maps for which $e^{\htop(f)}$ is transcendental.  For instance, it applies to the map from $f = g\circ h$ given by
$$
g(x_1,x_2) = \left(-x_1\frac{1-x_1+x_2}{1-x_1-x_2},-x_2\frac{1+x_1-x_2}{1-x_1-x_2}\right)\quad\text{and}\quad h(x_1,x_2) = (x_1 x_2^2,x_1^{-2} x_2),
$$
which is shown in \cite{BDJ20} to have transcendental first dynamical degree $\ddeg(f) > \dtop(f)$.

Spaces created by performing infinitely many blowups on a compact complex surface have previously been used in several ways to study dynamics of rational maps.  We mention in particular the work of Hubbard and Papadappol \cite{HPmemoir} on Newton's method for polynomials in two variables; of Cantat \cite{Can11} on groups of plane birational maps; and of Boucksom, Favre and and Jonsson \cite{BFJ08}, Favre and Jonsson \cite{FaJo11}, and Blanc and Cantat \cite{BlCa16} aimed at better understanding dynamical degrees of various
classes of plane rational maps; and of \cite{DETHELIN} which proposes the space obtained by resolving all points of indeterminacy of all iterates of a rational map $f$ as an appropriate setting for understanding the entropy of $f$.  As far as we know, this article and its predecessor \cite{DiRo24} are the first to effectively use an infinitely blown up space to construct and investigate a measure of maximal entropy for a rational map.  What makes this possible, and what distinguishes the toric limit surface $\rztO$ from the spaces used in earlier works is that $\rztO$ is a complex manifold in its own right.  This allows us to work directly with $\rztO$ much of the time rather than only with, say, classes in $\hoo(\rzt)$ or toric surfaces that approximate $\rztO$.

\subsection{Organization of the paper}
Following is a quick summary of the rest of this article.  
\begin{itemize}
\item \S\ref{sec:backgd} provides background on toric surfaces, on the non-compact surface $\rztO$ and on the associated spaces $\pcc(\rzt)$ and $\hoo(\rzt)$ of currents and classes.
\item \S\ref{sec:maps} reviews facts about toric maps and proves the tropical approximation result Theorem \ref{thm:tropapprox}.
\item \S\ref{sec:currents} discusses pullbacks and pushforwards of toric currents, including several facts about the equilibrium currents and their classes which were not observed in \cite{DiRo24}.
\item \S\ref{sec:currents_II} establishes the continuity result Theorem \ref{THM:CONTINUITY} for potentials of $T^*$.
\item \S\ref{sec:products} reviews the Bedford-Taylor approach to defining wedge products and energy pairings for positive closed $(1,1)$ currents on compact K\"ahler surfaces.  While this material is well-known in some quarters, we hope that it will be useful to readers less acquainted with pluripotential theory.  In \S\ref{ss:toricwedge}, we adapt everything to the non-compact surface $\rztO$.
\item \S\ref{sec:measure} constructs the wedge product $\mu = T^*\wedge T_*$ of interest.   
\item \S\ref{sec:mixing} establishes the dynamical and geometric properties of $\mu$, completing the proof of Theorem \ref{thm:mainthm}.
\end{itemize}

\medskip
\begin{ackn}
We thank Eric Bedford for his interesting comments.  Both authors gratefully acknowledge support from the NSF, the
first by grant DMS-2246893 and the second by grant DMS-2154414.
\end{ackn}

%% file: rzt2.tex
In this section, we review toric surfaces and their inverse limits.  For substantially more detail we refer readers to \cite{CoSc11} and \S3 in our previous paper \cite{DiRo24} on this subject.

Let $\torus\cong(\C^*)^2$ denote the two dimensional complex algebraic torus and $\eta$ be the canonical invariant holomorphic $2$-form on $\torus$.  As usual, we let $N\cong \Z^2$ denote the dual of the character lattice $M$ of $\torus$, $N_\R = N\otimes_\Z \R \cong\R^2$, and $\Log:\torus\to N_\R$ (the \emph{logarithm map}) be the surjective group homomorphism that assigns to a point $p\in\torus$ the linear functional 
$$
\Log(p):m\in M_\R \mapsto -\log|m(p)|.
$$
The real (i.e. maximal compact) subtorus of $\torus$ is then the set $\torus_\R := \Log^{-1}(0).$

A ray $\tau\in N_\R$ is \emph{rational} if $\tau\cap N$ is non-empty, and the unique primitive vector $v\in\tau\cap N$ is its \emph{generator}.  A \emph{sector} for us will be a closed, strictly convex two dimensional cone $\sigma\subset N_\R$.  We say that $\sigma$ is \emph{rational} if its bounding rays are rational, and then call $\sigma$ \emph{regular} if the generators for these rays form a basis for $N$.  Identifying the generators of a regular rational cone $\sigma$ with the standard basis for $\Z^2$, we obtain a unique isomorphism $x_\sigma = (x_1,x_2):\torus \mapsto (\C^*)^2$ satisfying
$$
\Log(p) = (-\log|x_1(p)|,- \log|x_2(p)|).
$$
We call $(x_1,x_2)$ the \emph{$\sigma$-coordinates} for $\torus$, observing that in these coordinates
\begin{itemize}
\item $\Log^{-1}(\sigma) = \{0<|x_1|,|x_2|\leq 1\}$ is the intersection of $(\C^*)^2$ with the closed unit polydisk;
\item $\eta = \frac1{4\pi^2}\frac{dx_1\wedge dx_2}{x_1x_2}$.
\end{itemize}
If $\tilde\sigma\in N_\R$ is another regular rational sector, then we have the change of coordinate formula $x_{\tilde\sigma} = h_A\circ x_\sigma$, where $A = (a_{ij}) \in\mathop{SL}(2,\Z)$ is the change of coordinate matrix from the basis for $N_\R$ determined by $\sigma$ to the basis determined by $\tilde\sigma$ and $h_A:(x_1,x_2) \mapsto (x_1^{a_{11}}x_2^{a_{12}},x_1^{a_{21}} x_2^{a_{22}})$ is monomial map determined by $A$.

For our purposes, a \emph{toric surface} will be a smooth complex manifold $X$ that compactifies $\torus$ and to which the action of $\torus$ on itself extends to a holomorphic action of $\torus$ on $X$.  Any such $X$ is determined by its \emph{fan} $\Sigma(X) = \{0\} \cup\Sigma_1\cup \Sigma_2$, a collection of closed and strictly convex $0$, $1$ and $2$ dimensional cones whose relative interiors partition $N_\R$.  Concretely, $X = \bigcup_{\sigma\in\Sigma_2} U_\sigma$, where $x_\sigma:U_\sigma \to \C^2$ are coordinate charts that extend the above $\sigma$-coordinates on $\torus$.  Letting $\tau_1,\tau_2\in\Sigma_1(X)$ denote the boundary rays of $\sigma$, we have that the coordinate axes in $\C^2$ are the images by $x_\sigma$ of $\torus$-invariant curves (which we call \emph{poles} for short) $C_{\tau_2}, C_{\tau_1}\subset X\setminus\torus$ that meet in a single $\torus$-invariant point $p_\sigma\in X\setminus\torus$.  If distinct sectors $\sigma,\tilde\sigma\in\Sigma_2$ share a boundary ray $\tau$ then $U_\sigma\cap U_{\tilde\sigma} = \torus \cup C_\tau^\circ$ where $C_\tau^\circ := C_\tau\setminus\{p_\sigma,p_{\tilde\sigma}\}$ denotes the complement of the torus invariant points in $C_\tau$.  We similarly write $X^\circ = X\setminus\{p_\sigma:\sigma\in\Sigma_2\}$.

If $X$ and $Y$ are both toric surfaces, we let $\pi_{XY}:X\tto Y$ denote the canonical \emph{transition}---i.e. the birational extension of the identity map on $\torus$.  The map $\pi_{XY}$ is a morphism if and only if $\Sigma(X)$ refines $\Sigma(Y)$, in which case we write $X \succ Y$.  In any case $\pi_{XY}$ contracts precisely those poles $C_\tau\subset X$ for which $\tau$ is an interior ray of $\sigma\in\Sigma(Y)$.  For any two toric surfaces $X_1,X_2$ there is another $X \succ X_j$ that dominates both.  We therefore define the \emph{toric limit space} $\rzt$ to be the inverse limit of all toric surfaces with respect to transitions.  The following summarizes the discussion in \S3.1 of \cite{DiRo24}.

\begin{thm}
The toric limit space $\rzt$ is a Hausdorff topological compactification of $\torus$ on which $\torus$ acts by homeomorphisms compatible with the natural map $\rzt\to X$ onto any toric surface $X$.  The complement $\rzt\setminus\torus$ consists of
\begin{itemize}
\item a $\torus$-invariant \emph{irrational point} $p_\tau$ for every irrational ray $\tau\subset N_\R$,
\item two $\torus$-invariant points and a \emph{pole} $C_\tau \cong \C^*$ for each rational ray $\tau\subset N_\R$. 
\end{itemize}
Distinct poles and $\torus$-invariant points are disjoint from each other, and every pole $C_\tau$ is compactified to $\bP^1$ by the two \emph{rational} $\torus$-invariant points indexed by the same ray $\tau$.  

For any toric surface $X$, the natural projection $\rzt\to X$ may be (partially) inverted to give an inclusion $X^\circ\hookrightarrow \rzt$.  The complement of the $\torus$-invariant points of $\rzt$ is a (non-compact, non-algebraic) complex manifold $\rztO$ equal to the union of the images of all such inclusions.
\end{thm}

We call $\rztO$ the \emph{toric limit surface}.  For simplicity, we will write $X^\circ\subset\rztO$ for any toric surface $X$, implicitly identifying $X^\circ$ with its image in $\rztO$. Given $R>0$ and a finite collection of rational rays $\tau_j\subset N_\R$, $1\leq j\leq J$ we define the \emph{star} of width $R$ along $\tau_1,\dots,\tau_J$ to be the $\torus_\R$-invariant set
$$
Q = Q(\tau_1,\dots,\tau_J,R) := \overline{\{p\in \torus: \dist(\Log(p),\tau_j) \leq R \text{ for some } 1\leq j \leq J\}}.
$$
The closure here is understood to take place in $\rzt$.

\begin{cor}
\label{cor:star}
Every star is a compact subset of $\rztO$, and every compact subset of $\rzt^\circ$ is contained in a star.
\end{cor}

\begin{proof}
If $X$ is a toric surface whose fan includes the rays $\tau_1,\dots,\tau_J$, then the closure $\bar Q_X$ of $Q\cap\torus$ in $X$ is a $\torus_\R$ invariant compact subset of $X^\circ\subset\rztO$.   If $K\subset\rztO$ is some other compact set, then in fact $K\subset X^\circ$ for some toric surface $X$, since the sets $X^\circ$ form an open exhaustion of $\rztO$.  The proof is therefore completed by observing that any given $X^\circ$ is exhausted by the interiors of stars of increasing width $R$ directed by the rays in $\Sigma_1(X)$.
\end{proof}

\subsection{Currents and divisors}
The following mostly just summarizes \S's 2, 6 and 7 in \cite{DiRo24}, with a shift in viewpoint that places greater emphasis on the complex manifold $\rztO$.  We refer readers to that article for more details and arguments.

Given a complex surface $X$, we let $\pcc^+(X)$ denote the set of positive closed $(1,1)$ currents\footnote{When $X$ is not compact, we allow currents with non-compact support; hence currents act on \emph{compactly supported} test forms.} on $X$ and $\pcc(X)$ denote the set of differences $S-T$ of elements $S,T\in\pcc^+(X)$.  There are of course many closed $(1,1)$ currents that cannot be expressed this way, so our definition of $\pcc(X)$ is a little non-standard.  We regard the set $\div(X)$ of $\R$-divisors on $X$ as a subset of $\pcc(X)$ by identifying each divisor $D$ with the current obtained by integrating over $D$. 

If $X$ is projective (in particular, compact and K\"ahler), then each $T\in\pcc(X)$ represents a class $\ch[X]{T} \in \hoo(X)\subset H^2(X,\R)$.  The class $\ch[X]{T}$ is trivial precisely when $T = dd^c \rpot$ (in the distributional sense) for some $\rpot\in L^1(X)$; locally $\rpot = u - v$ is a difference between two psh functions $u,v$.  The class $\ch{T}$ is \emph{effective} if it is represented by a positive current) and \emph{K\"ahler} if it is represented by a K\"ahler form.  A limit $\ch{T}$ of K\"ahler classes is \emph{nef}, characterized by the fact that its intersection number $\isect{T}{D}_X$ with every effective class is non-negative.

Abusing notation slightly we write $\pcc(\rzt)$ in place of $\pcc(\rztO)$ and call its elements \emph{toric currents}.  If $X$ is a toric surface and $T$ a toric current, then standard extension theorems (see e.g. \cite{Sib85}) for positive closed currents imply that the restriction $T|_{X^\circ}$ extends uniquely (by zero) to a current $T_X\in \pcc(X)$, which we continue to call the \emph{restriction} of $T$ to $X$.  One checks easily that for toric surfaces $X\succ Y$, we have $\pi_{XY*} T_X = T_Y$.  Though $\rztO$ is not compact, or even an open subset of a compact complex manifold, we continue to declare toric currents $S$ and $T$ to be cohomologous in $\rztO$ if $S-T = dd^c \rpot$ for some real-valued $\rpot\in L^1_{loc}(\rztO)$.  We let $\hoo(\rzt)$ denote the resulting quotient of $\pcc(\rzt)$, writing $\ch{T}$ for the class of $T$ in $\hoo(\rzt)$.

\begin{prop} $S,T\in\pcc(\rzt)$ are cohomologous if and only if $S_X$ and $T_X$ are cohomologous in every toric surface $X$.  In either case, relative potentials $\rpot\in L^1_{loc}(\rztO)$ for $S-T$ and $\rpot_X\in L^1(X)$ for $S_X-T_X$ are unique up to additive constants and so can be chosen so that
$$
\rpot_X = \rpot|_{X^\circ}.
$$
\end{prop}
 
\begin{proof}
It was shown in \cite[Proposition 7.4]{DiRo24} that if $S_X$ and $T_X$ are cohomologous in every $X$, then $S$ and $T$ are cohomologous in $\rztO$.  For the other direction, we suppose without loss of generality that $S$ and $T$ are positive, and suppose that $S-T= dd^c\varphi$ for some locally integrable $\varphi:\rztO\to \R$.  Fix a toric surface $X$ and let $S_X,T_X\in \pcc^+(X)$ be the restrictions of $S$ and $T$ to $X$.  Let $U\subset X$ be a union of pairwise disjoint coordinate balls centered at the finitely many $\torus$-invariant points of $X$.  Then we have psh functions $u,v:U\to \R$ such that $dd^c u = S_X$ and $dd^c v = T_X$.  Let $U^\circ = X^\circ \cap U$.  Then restriction gives on $U^\circ$ that 
$$
dd^c (u-v) = (S_X-T_X) = S-T = dd^c \varphi.
$$
Hence $h = \varphi-u-v$ is pluriharmonic on $U^\circ$.  Hartog's extension gives that $h$ and therefore $\varphi = h+ v-u$ extends to a potential for $S_X-T_X$ on all of $U$.  Since $\varphi$ is already a potential for $S_X-T_X$ on $X-U \subset X^\circ$, it follows that $S_X-T_X = dd^c\varphi$ on $X$. 
\end{proof}

From this discussion, it follows that $\pcc(\rzt)$ and $\hoo(\rzt)$ are the inverse limits of the corresponding objects $\pcc(X)$ and $\hoo(X)$ on toric surfaces.  We will therefore say that $\alpha\in\hoo(\rzt)$ is \emph{effective} (or \emph{nef}, or \emph{K\"ahler}) if the same property adheres to its restriction $\alpha_X\in \hoo(X)$ in each toric surface $X$.  This is non-standard since e.g. we are \emph{not} saying that the K\"ahler classes in $\hoo(\rzt)$ correspond to K\"ahler forms on $\rztO$.

We adopt the convention that, unless otherwise noted, a \emph{curve} in a complex surface $X$ is a reduced and irreducible, properly embedded, one dimensional analytic subvariety $C\subset X$.  With this convention, every curve corresponds to a current of integration on $X$, so the above discussion implies that every curve $C\subset\rztO$ is either
\begin{itemize}
 \item \emph{internal}, i.e. compact and given by $C = \overline{C\cap\torus}$; or
 \item the complement $C_\tau^\circ := C_\tau\cap\rzt^\circ\cong \C^*$ of the $\torus$-invariant points in the pole associated to some rational ray $\tau\subset N$.
\end{itemize}
A \emph{toric divisor} $D\in\div(\rzt)$ is then given by $D = D_{int}+D_{ext}$, with \emph{internal} part $D_{int}$ supported on finitely many internal curves, and \emph{external} part $D_{ext} = \sum_{\tau} c_\tau C_\tau^\circ$, where the sum is (countably) infinite, over \emph{all} rational rays $\tau\subset N_\R$.  

Similarly, any current $T\in\pcc(\rzt)$ decomposes as $T = T_{int}+T_{ext}$ where $T_{ext}$ is an external divisor and $T_{int}\in\pcc(\rzt)$ is \emph{internal}, equal to the trivial extension to $\rztO$ of $T|_{\torus}$.  The cohomology class of a positive internal current $T = T_{int}$ is always nef \cite[Corollary 6.10]{DiRo24}.

Every internal current $T\in\pcc(\rzt)$ can be averaged by the action of $\torus_\R$ to produce a $\torus_\R$-invariant internal current $T_{ave}$ representing $\ch{T}$.  

\begin{prop}
\label{prop:average}
For any internal current $T\in\pcc(\rzt)$, there is a continuous function $\sfn_T:N_\R\to\R$, unique up to the addition of an affine function, such that
$$
T_{ave}|_{\torus} = dd^c(\sfn_T\circ\Log).
$$
\end{prop}

We call $\sfn_T$ a \emph{support function} for $T$.  Where possible, we will choose support functions to be partially normalized, satisfying $\sfn_T(0) = 0$.  We have the following characterization of support functions for positive internal currents.

\begin{thm}[See \cite{DiRo24}, Theorem 6.9]
A function $\sfn:N_\R\to \R$ is the support function of a positive internal current $T\in\pcc(\rzt)$ if and only if it is convex and has `linear growth', i.e.
$$
\frac{\norm{\sfn_T(v)}}{\norm{v}+1}
$$
is uniformly bounded on $N_\R$.
\end{thm}

We call an internal current $T\in\pcc(\rzt)$ \emph{homogeneous} if it is $\torus_\R$-invariant and its support function is \emph{positively homogeneous}, satisfying $\sfn_T(tv) = t\sfn(v)$ for all $v\in N_\R$ and $t>0$.  We will usually signify that a current is homogeneous by writing $\bar T$ instead of $T$.  Homogeneous currents serve as canonical proxies for the cohomology classes that they represent.

\begin{thm}[See \cite{DiRo24}, Theorem 1.3]
\label{thm:homogenization}
Any internal current $T\in\pcc(\rzt)$ is cohomologous to a unique homogeneous current $\bar T \in\pcc(\rzt)$.  If $T$ is positive, then so is $\bar T$.  In any case, if $\sfn_T$ is a support function for $T$, then the limit
\begin{equation}
\label{eqn:homogenization}
\sfn_{\bar T}(v) := \lim_{t\to\infty} \frac{\sfn_T(tv)}{t}
\end{equation}
converges uniformly on compact subsets of $N_\R$ to a support function for $\bar T$.  If $T$ is positive and $\sfn_T(0) = 0$, then 
\begin{itemize}
 \item $\sfn_T \leq \sfn_{\bar T}$ on $N_\R$;
 \item $\sfn_T-\sfn_{\bar T}$ is non-increasing along each ray $\tau\subset N_\R$, i.e. if $v\in\tau$ is non-zero, then 
 $t\mapsto (\sfn_T-\sfn_{\bar T})(tv)$ is non-increasing for $t>0$.
\end{itemize}
 \end{thm}

It follows from this theorem that any internal current $T\in\pcc(\rzt)$ can be written
$$
T - \bar T = dd^c \rpot_T,
$$
where the potential $\rpot_T:\rztO\to\R$ is unique up to additive constants and varies continuously with $T$ if we normalize by e.g. $\rpot_T(0) = 0$.

As a global function on $\rztO$, the composition $\sfn_{\bar T}\circ\Log$ is locally integrable even about points in $\rztO\setminus\torus$.  More precisely $\sfn_{\bar T}$ has a logarithmic singularity along each pole $C_\tau^\circ\subset\rztO$, and if we regard $\sfn_{\bar T}\circ \Log$ as a locally integrable function on the larger domain $\rztO$,  we have the following enhancement of the formula in Proposition \ref{prop:average}.
$$
dd^c(\sfn_{\bar T}\circ\Log) = \bar T - D,
$$
where the external divisor $D$ is given by
$$
D = \sum_{\tau} \sfn_{\bar T}(v_\tau) C_\tau.
$$
The sum is, as usual, over rational rays $\tau\subset N_\R$ and $v_\tau$ is the primitive vector in $\tau\cap N$.  For the same divisor $D$, we also have
$$
dd^c(\sfn_T\circ\Log) = T_{ave} - D.
$$
It follows that $\rpot_{T_{ave}} := (\sfn_T-\sfn_{\bar T})\circ\Log\in L^1_{loc}(\rzt^\circ)$ is a potential for $T_{ave}-\bar T$.

In what follows we will take advantage of a distinction that was not remarked on in \cite{DiRo24}.

\begin{defn}
\label{defn:nearlyhomogeneous}
We say that a support function $\sfn_T$ for a positive internal current $T\in\pcc^+(\rzt)$ is \emph{nearly homogeneous} if $|\sfn_T-\sfn_{\bar T}|$ is uniformly bounded on $N_\R$.  Here $\sfn_{\bar T}$ is assumed to be the support function for $\bar T$ derived from $\sfn_T$ via \eqref{eqn:homogenization}.
\end{defn}

Not all positive internal currents $T\in\pcc^+(\rzt)$ have nearly homogeneous support functions.  But since the difference $\sfn_T - \sfn_{\bar T}$ is unique up to additive constants, one support function for $T$ is (nearly) homogeneous if and only if all of them are.  Since $\sfn_T-\sfn_{\bar T}\leq 0$ is continuous on $N_\R$ and decreasing along every ray, one can verify that $\sfn_T$ is nearly homogeneous by checking it on a `large enough' subset of $N_\R$.

\begin{prop}
\label{prop:nearlyhomogenough} The support function $\sfn_T$ for $T\in\pcc^+(X)$ is nearly homogeneous if there exists a constant $C>0$ and a dense set $\mathcal{S}$ of rays $\tau\subset N_\R$ such that for each ray $\tau\in S$, we have $|\sfn_T-\sfn_{\bar T}| \leq C$ outside a compact subset of $\tau$.
\end{prop}

\begin{prop}
\label{prop:changes}
If $T\in\pcc^+(\rzt)$ represents a nef class and $Y\succ X$ are toric surfaces, then we have
$$
\pi_{YX}^* T_X = T_Y + D
$$
where $D$ is an effective external divisor with $\supp D\subset \exc(\pi_{YX})$.  If $\ch{T}$ is K\"ahler then $\supp D = \exc(\pi_{YX})$, and in particular $\pi_{XY}^* T_X > T_Y$ when $Y\neq X$.
\end{prop}

\begin{proof}
It suffices to consider the case where $\pi_{YX}$ is the blowup of a single $\torus$-invariant point $p_\sigma\in X$.  Let $E = \pi_{XY}^{-1}(p_\sigma)$ denote the contracted curve.  Then
$T_Y = \pi_{XY}^* T_X + cE$ for some $c\in\R$, and since $\ch{T_Y}$ is nef, we have
$$
0\leq \isect{T_Y}{E} = \isect{\pi_{XY}^*T_X}{E} - c\isect{E}{E} = \isect{T_X}{\pi_{XY*} E} + c = c.
$$
If $\ch{T_X}$ is K\"ahler, the inequality is strict.
\end{proof}
As discussed in \cite{DiRo24}, we have a well-defined intersection product on a subspace $\eltwo(\rzt)\subset\hoo(\rzt)$ that includes all nef (hence all internal) classes.  Indeed if $\alpha,\beta\in \hoo(\rzt)$ are nef classes, then the intersection $\isect{\alpha}{\beta}_X$ of $\alpha_X,\beta_X\in\hoo(X)$ decreases as the toric surface $X$ increases, and we have
$$
\isect{\alpha}{\beta} := \inf_X \isect{\alpha}{\beta}_X.
$$

Given a positive closed current $T$ on a surface $X$ and a point $p\in X$, let $u$ be a local potential for $T$ near $p$.  The \emph{Lelong number} of $T$ at $p$ is the quantity
$$
\nu(T,p) := \max\{s:u(\cdot)-s\log\dist(\cdot,p) \text{ is bounded above near }p\}\geq 0,
$$
which is finite and does not depend on the choice of $u$.  A deep theorem of Siu \cite{Siu74} asserts that $p\mapsto \nu(p,T)$ is upper semicontinuous in the Zariski topology on $X$.

\begin{prop}
\label{prop:smalllelong}
Let $T\in\pcc^+(\rzt)$ be an internal current.  Given $\epsilon>0$ we have
$$
\sum_{\sigma \in\Sigma_2(X)} \nu(T_X,p_\sigma) < \epsilon
$$
for sufficiently dominant $X$.  If $\ch{T} \in \hoo(\rzt)$ is K\"ahler, then we also have $\nu(T_X,p_\sigma)>0$ for all $\torus$-invariant $p_\sigma\in X$.
\end{prop}

\begin{proof}
Fix an initial toric surface $X$ and let $-K_X := \sum_{\tau\in\Sigma_1(X)} C_\tau$ denote the anticanonical divisor on $X$.  Since $\ch{T}$ is nef, we have $\isect{-K_X}{T_X} \geq 0$.  Now let $\pi:Y\to X$ denote the blowup of $X$ at some $\torus$-invariant point $p_\sigma\in X$ and $C_\tau\subset Y$ be the pole contracted by $\pi$.  Then $K_Y = \pi^*K_X + C_\tau$, and since the restriction of an internal current to a toric surface does not charge poles,
$$
\pi^* T_X = T_Y + \nu(T_X,p_\sigma) C_\tau.
$$
If in particular, $\ch{T}$ is K\"ahler, then $\pi^* T_X > T_Y$ by Proposition \ref{prop:changes}, so $\nu(T_X,p_\sigma) >0$ in that case.  Regardless, we infer
\begin{align}\label{EQN:DECREASING_INTERSECTION}
0 \leq \isect{-K_Y}{T_Y} = \isect{-K_X}{T_X} - \nu(T_X,p_\sigma).
\end{align}
So $\isect{-K_X}{T_X}$ is uniformly bounded below and non-increasing in $X$.  If $\isect{-K_X}{T_X}$ is within $\epsilon$ of ${\rm inf}_X(-[K_X] \cdot [T_X])$, then (\ref{EQN:DECREASING_INTERSECTION}) shows that $\sum_{\sigma\in\Sigma_2(X)} \nu(T_X,p_\sigma)<\epsilon$.
\end{proof}

The following restates the second conclusion of \cite[Theorem 6.7]{DiRo24}.

\begin{prop} 
\label{prop:bartonpole}
A homogeneous current $\bar T\in\pcc^+(\rzt)$ has continuous local potentials about all points in $\rztO$.   
\end{prop}

%% file: maps4.tex
As noted in the introduction, we will call a rational map $f:\bP^2\tto\bP^2$ \emph{toric} if it preserves the canonical two-form up to scale, i.e. if $f^*\eta = \rho(f)\eta$ for some $\rho\in\C^*$.  Such a map extends naturally to a meromorphic self-map $f_{XY}:X\tto Y$ between any two toric surfaces $X$ and $Y$.  Passing to the limit, we obtain that $f$ extends to a meromorphic self-map (which we will continue to denote by $f$) of the complex manifold $\rztO$.  Even though $\rztO$ is not compact, this extension behaves much like a rational self-map of a compact projective surface.   Let $\ind(f)\subset \rztO$ denote the indeterminacy set of $f:\rztO\tto\rztO$ and $\exc(f)$ denote the exceptional set, i.e. the union of curves contracted by $f$ to points.

\begin{thm}[\cite{DiRo24} Theorem 4.3 through Corollary 4.7]
\label{thm:tmapbasics}
If $f:\bP^2\tto\bP^2$ is toric, then the induced map $f:\rztO\tto\rztO$ has the following properties
\begin{enumerate}
 \item $\ind(f)$ is a finite subset of $\exc(f)$; and $f(\ind(f))$ is a finite union of internal curves.
 \item $\exc(f)$ consists of finitely many internal curves $C$, each with image
 $f(C):=\overline{f(C\setminus\ind(f))}$ equal to a point in $\rztO\setminus\torus$.
 \item If $p\in\torus\setminus\exc(f)$, then $f(p)\in\torus$ and $Df(p)$ is non-singular.
 \item If $C\not\subset\exc(f)$ is an internal curve, then so is $f(C)$.
 \item If $C_\tau\subset\rztO\setminus\torus$ is a pole, then so is $f(C_\tau)$.
 \item \label{item:proper} $f$ extends continuously to a self-map $\hat f:\rzt\setminus\ind(f)\to \rzt$ for which $\rzt\setminus\rztO$ is totally invariant.
\end{enumerate}
\end{thm}

\begin{cor}\label{COR:PROPER_AND_COVER} If $f:\rztO\tto\rztO$ is a toric rational map, then 
\begin{enumerate}
\item $f:\rztO\tto \rztO$ is \emph{proper} in the following sense: for any compact $K \subset \rztO$ the set $\{p \in \rztO \, : \, f(p) \cap K \neq \emptyset\}$ is compact.   
\item $f:\rztO \setminus f^{-1}(f(\ind(f))) \to \rztO\setminus f(\ind(f))$ is a finite holomorphic covering with branch locus contained in $\rztO\setminus \torus$.
\end{enumerate}
\end{cor}

\begin{defn} 
\label{defn:intstable}
A toric map $f:\rztO\tto\rztO$ is \emph{internally stable} if $f^n(\exc(f))\cap \ind(f) = \emptyset$ for all $n\geq 0$.
\end{defn}

If $f$ is internally stable, then $\ind(f^n) = \bigcup_{j=0}^{n-1} f^{-j}(\ind(f))$ and $\exc(f^n) = \bigcup_{j=0}^{n-1} f^{-j}(\exc(f))$.  In particular, $f^n$ is also internally stable.  We further set $\ind(f^\infty) = \bigcup_{n>0} f^{-n}(\ind(f))$ and $\ind(f^{-\infty}) = \bigcup_{n>0} f^n(\exc(f))$.

\subsection{Tropical Approximation}
An important feature of toric maps is that they are almost semiconjugate, via the logarithm map, to  piecewise linear real maps of $N_\R$.  The following is Theorem 4.8 in \cite{DiRo24} and mostly taken from \cite{DiLi16}.

\begin{thm}
\label{thm:tropicalization}
If $f:\rztO\tto\rztO$ is a toric map, then there exists a finite set $\Sigma_1(f)$ of rational rays $\tau\in N_\R$ and a continuous self-map $A_f:N_\R\to N_\R$ with the following properties.
\begin{enumerate}
\item $A_f$ is `integral', i.e. $A_f(N)\subset N$.
\item If $\sigma\subset N_\R$ is a sector that omits all $\tau\in\Sigma_1(f)$, the restriction $A_f|_\sigma$ is linear with $\det A_f|_\sigma = \pm\rho(f)$.  Hence $\rho(f)$ is an integer.
\item For any pole $C_\tau\subset \rzt\setminus\torus$, we have $f(C_{\tau}) = C_{A_f(\tau)}$.
\item If $\nvec,\nvec'\in N$ are the primitive vectors generating the rational rays $\tau, A_f(\tau) \subset N_\R$, then the ramification of $f$ about the pole $C_\tau$ is given by
$$
\ram(f,C_{\tau}) = \frac{\norm{A_f(\nvec)}}{\norm{\nvec'}},
$$
and $f:C^\circ_\tau\to C^\circ_{A_f(\tau)}$ is a covering of degree $|\rho(f)|/\ram(f,C_\tau)$.  Hence $f$ has local topological degree $\rho(f)$ on a neighborhood of any pole.
\item $A_f:N_\R\setminus\{0\}\to N_\R\setminus\{0\}$ is a covering map with degree $\dtop(f)/|\rho(f)|$.
\end{enumerate}
\end{thm}

We call $\tropf$ the \emph{tropicalization} of $f$. If $g:\rzt\tto\rzt$ is another toric map, then $\tropf \circ\trop{g} = \trop{f\circ g}$.  The tropicalization of a monomial or birational toric map $f$ is always injective (i.e. a homeomorphism), and in fact all other examples of toric maps we know are obtained as finite compositions of these.  So as far as we know, $\tropf$ is always injective.  In any case, when $\tropf$ is a homeomorphism, it induces a homeomorphism on the parameter space $S^1$ of rays in $\R^2$, and we refer to the rotation number of this quotient map as the \emph{rotation number of $\tropf$}.  

To demonstrate just how well $\tropf$ approximates $f$, we make a few further observations.

\begin{prop}
\label{prop:tmapmore}
The following hold for any toric map $f:\rzt\tto\rzt$.
\begin{enumerate}
\item If $\tropf$ is injective, then every component of $\exc(f)$ meets $\ind(f)\setminus\torus$.
\item For each $p\in\ind(f)\cap \torus$, we have $f(p)\setminus\torus\subset f(\exc(f))$.
\item For each $p \in \ind(f)\setminus\torus$, let $C_\tau$ be the unique pole containing $p$ and $q=f|_{C_\tau}(p)\in C_{\tropf(\tau)}$.  Then $f(p)\setminus\torus \subset f(\exc(f)) \cup \{q\}$.
\end{enumerate}
\end{prop}

\begin{proof}
If $C\subset\exc(f)$ is irreducible, then $C$ meets at least two distinct poles $C_\tau,C_{\tau'}\in\rzt$ by \cite[Corollary 3.4]{DiRo24}. If neither intersection is indeterminate for $f$, then both map to the same point $f(C)\in f(\exc(f))$.  This implies, however, that both $C_\tau$ and $C_\tau'$ map onto the pole containing $f(C)$, which contradicts injectivity of $\tropf$.  This proves the first assertion.

To prove the last two assertions, we fix $p\in\ind(f)$ and let $\pi:S\to \rztO$ be a modification that resolves that indeterminacy of $f$ at $p$.  That is, $\pi:S\setminus\pi^{-1}(p)\to \rztO\setminus\{p\}$ is an isomorphism and $\hat f := f\circ\pi$ is holomorphic on a neighborhood of $\pi^{-1}(p)$.  The fact that $f$ is toric implies $\hat f^*\eta = \rho\,\pi^*\eta$, and in particular \cite[Proposition 5.3]{DiLi16} that the poles of $\pi^*\eta$ are precisely the strict transforms of the poles of $\eta$ by both $\pi$ and $\hat f$.  

Now let $q\in \hat f(\pi^{-1}(p))\cap C_\tau$ be a point in some pole $C_\tau\in\rzt$.  We may suppose that $q\notin f(\exc(f))$.  Hence we can choose an open set $U\ni p$ such that $U\setminus\{p\}$ contains no points in $\ind(f)$ and no preimages of $q$.  It follows that $Z := \hat f^{-1}(q)\cap \pi^{-1}(U)$ is a non-empty (possibly reducible) subvariety of $\pi^{-1}(p)$.  Since $q\in C_\tau$, we have that $Z$ meets an irreducible component $C\subset S$ of the strict transform of $C_\tau$ by $\hat f$.  Since $\hat f(C) = C_\tau$, we have $C\not\subset \pi^{-1}(q)$.  Hence $\pi(C)\subset\rztO$ is a curve rather than a point, and $f(\pi(C)) = C_\tau$.  Hence by Theorem \ref{thm:tmapbasics}, $\pi(C)$ is a pole containing $\pi(Z\cap C) = p$, and $f|_{\pi(C)}:\pi(C) \to C_\tau$ maps $p$ to $\hat f(C\cap Z) = q$.
\end{proof}

Our main approximation result is the following.  A version of this was given in \cite{DiRo24}, but the proof given there is incomplete in a couple of places, so we take the opportunity to prove it fully here.

\begin{thm}
\label{thm:tropapprox}
Suppose that $f$ is toric.  Let $U,U',V\subset\rzt$ be neighborhoods of $\exc(f),$ $\ind(f),$ and $f(\exc(f))$ respectively.  Then there exists $C>0$ such that for any $p\in\torus$,
\begin{enumerate}
 \item \label{item:ta1} $\norm{\Log\circ f(p) - \tropf\circ\Log(p)} \leq C$ unless $p\in U$.
 \item \label{item:ta2} $\norm{\Log\circ f(p)} \geq \norm{\tropf\circ\Log(p)} - C$ unless $p\in U'$.
 \item \label{item:ta3} $\norm{\Log\circ f(p)} \leq \norm{\tropf\circ\Log(p)} + C$ unless $p\in U$ and $f(p)\in V$.
\end{enumerate}
\end{thm}

\proof
Recall that if $p\in\torus\setminus\exc(f)$ then $f(p)\in\torus$.  In order to compare the behavior of $f$ with that of $\tropf$, we choose
toric surfaces $X$ and $Y$ whose fans have the following properties.
\begin{itemize}
\item $\Sigma_1(f) \subset \Sigma_1(X)$.
 \item for each sector $\sigma\in\Sigma_2(X)$ there is a sector $\sigma'\in\Sigma_2(Y)$ such that $\tropf(\sigma)\subset \sigma'$.
 \item if $C_\tau\subset\rzt$ is a pole that meets $\ind(f_{XY})$ or $\exc(f)$, then $\tau\in \Sigma_1(X)$; i.e. $C_\tau\subset X$.
 \item if $C_\tau'\subset\rzt$ is a pole that meets $f(\ind(f))$ or $f(\exc(f))$, then $\tau'\in\Sigma_1(Y)$; i.e. $C_{\tau'}\subset Y$.
\end{itemize}
Since $X$ is covered by the finitely many closed polydisks $\overline{\Log^{-1}(\sigma)}$, $\sigma\in\Sigma_2(X)$, it will suffice to work in local coordinates to prove each of conclusions \eqref{item:ta1}-\eqref{item:ta3} of Theorem \ref{thm:tropapprox} on such a polydisk.

Fix $\sigma\in\Sigma_2(X)$ and $\sigma'\in\Sigma_2(Y)$ such that $\tropf(\sigma)\subset \sigma'$, and let $(x_1,x_2)$, $(y_1,y_2)$ be the $\sigma,\sigma'$-coordinates on $\torus$.  Under the above assumptions, we have that $\tropf|\sigma$ is linear, and $f_{XY}$ is well-defined and locally finite near $p_\sigma\subset X$, given by
\begin{equation}
\label{eq:localmap}
(y_1,y_2) = f_{\sigma\sigma'}(x_1,x_2) = (g_1(x)x_1^a x_2^b,g_2(x)x_1^c x_2^d).
\end{equation}
where 
\begin{itemize}
 \item $A=A_{\sigma\sigma'} = \begin{pmatrix} a & b \\ c & d\end{pmatrix}$ is the non-negative $2\times 2$ integer matrix that represents $\tropf$ in the bases for $N_\R$ determined by $\sigma$ and $\sigma'$;
 \item $g_1,g_2:\C^2\tto\C$ are rational functions, well-defined and non-vanishing near $(0,0)$;
 \item the zeroes and poles of $g_1,g_2$ are the components of $\exc(f)$.
\end{itemize}
In the coordinates on $N_\R \cong \R^2$ determined by identifying $\sigma$ with the first quadrant, we have
\begin{eqnarray*}
\err(p) 
&:= & 
\Log\circ f(p) - \tropf\circ\Log(p) 
= 
\Log\circ f_{\sigma\sigma'}(x(p)) - A \Log(x(p)) \\
& = & 
(-\log|g_1(x)|,-\log|g_2(x)|).
\end{eqnarray*}   
Hence if $U\subset X$ is any neighborhood of $\exc(f)$, we have that $\norm{\err}$ is uniformly bounded on $\Log^{-1}(\overline\sigma)\setminus U$.  This is conclusion \eqref{item:ta1} of Theorem \ref{thm:tropapprox}.  

For the conclusion \eqref{item:ta2}, we note first the following special case.

\begin{lem}
\label{lem:c2}
For any point $q\in \exc(f) \setminus\torus$ that is not indeterminate for $f$, there exists a neighborhood $U_q\ni q$ such that conclusion (2) holds for $p\in U_q$.
\end{lem}

Taking Lemma \ref{lem:c2} for granted momentarily, we finish verifying
conclusion \eqref{item:ta2}.  Lemma~\ref{lem:c2} provides an open neighborhood
$U''$ of $(\exc(f) \setminus \torus) \setminus \ind(f)$ on which the desired
estimate holds.  Shrinking the neighborhood $U$ of $\exc(f)$, we may suppose
that $\overline{U\setminus (U' \cup U'')}$ is compact in $\torus$ and therefore
$K:=\Log(\overline{U\setminus (U' \cup U'')})$ is compact in $N_\R$.  So since $p\notin U'$ by hypothesis, if $p
\not \in U''$, then either $p\notin U$ or $\Log(p)\in K$.  In the first case,
Conclusion~(2) follows from Conclusion (1).  In the second case $\norm{A_f\circ
\Log(p)} \leq R$ for some $R = R(K)$ independent of $p$; i.e. conclusion
\eqref{item:ta2} holds for $p\in K$ with constant $C = C(K) = R$.

\begin{proof}[Proof of Lemma \ref{lem:c2}]
By our choice of $X$, $q\in C_\tau$ for some ray $\tau\in\Sigma_1(X)$.  Let $\sigma\subset\Sigma_2(X)$ be the sector with $\tau$ as its `righthand' boundary ray, i.e. so that $q\in\overline{\Log^{-1}(\sigma)}$ has $\sigma$-coordinates $x(q) = (0,\alpha)$ where $0<|\alpha|\leq 1$.  By choice of $Y$, we further have that $\tropf(\tau)\in\Sigma_1(Y)$ and $f_{XY}(q) \in C^\circ_{\tropf(\tau)}$.  Necessarily $\tropf(\tau)$ is one of the rays bounding the sector $\sigma'\in\Sigma_2(Y)$ that contains $\tropf(\sigma)$.  We suppose for argument's sake that this is again the righthand ray so that $f(q)$ has $\sigma'$-coordinates $y(f(q)) = (0,\beta)$ for some $0<|\beta|< 1$.  Since $f_{\sigma\sigma'}$ preserves $(0,0)$ and maps $\{x_1=0\}$ onto $\{y_1=0\}$ implies that the matrix $A$ in \eqref{eq:localmap} has the form $\begin{pmatrix} a & b \\ 0 & d \end{pmatrix}$ where $a,d>0$ and $b\geq 0$.  Finally, since $q\in\exc(f)\setminus\ind(f)$, we must have that $g_1(0,\alpha)=0$ and $(0,\alpha)$ is not a zero or pole of $g_2$.

If, in $\sigma$-coordinates, $U_q\subset X$ is a small ball about $(0,\alpha)$, then $\left|g_1|_{U_q}\right|$ is bounded above by~$1$ and $U_q \cap\exc(f) = \{g_1=0\}$.  For $p\in U_q\setminus\exc(f)$, we have that
$$
A\Log(x_1,x_2) = (-a\log|x_1|-b\log|x_2|,-d\log|x_2|) \approx (-a\log|x_1|-b\log|\alpha|,-d\log|\alpha|)
$$
is uniformly near $(-a\log|x_1|,0)$; moreover,
$$
\Log\circ f_{\sigma\sigma'}(x_1,x_2) - \tropf(x_1,x_2) \approx (-\log|g_1|,-\log|g_2(0,\alpha)|)
$$
is uniformly near $(-\log|g_1|,0)$.  Since $-a\log|x_1|$ and $-\log|g_1|$ are both positive on $U_q$, the lemma follows.
\end{proof}

It remains to establish the third conclusion of Theorem \ref{thm:tropapprox}. Since $f^{-1}(V)$ contains a neighborhood of $\exc(f)\setminus\ind(f)$, Conclusion~(1) reduces the problem to verifying Conclusion~(3) on a small enough neighborhood $U_q$ of each of the finitely many points $q \in\ind(f)$.  When $q\in\torus$, we have from Proposition \ref{prop:tmapmore} that $f(q)\setminus\torus\subset f(\exc(f))$.  Hence Conclusion (3) holds near $q$ if the constant $C$ is large enough that $f(q)\cap\{\norm{\Log}\geq C\} \subset V$.  

If, on the other hand, $q\in\ind(f)\cap C_\tau$ is external, then again by Proposition \ref{prop:tmapmore} the same argument works except that the point $q' := f|_{C_\tau}(q) \in f(q)\setminus\torus$ need not lie in $f(\exc(f))$.  That is, we have reduced to verifying Conclusion (3) at all points $p\in W\cap\torus$, where $W := U_q\cap f^{-1}(V_{q'})$ and $U_q$ and $V_{q'}$ are small coordinate neighborhoods of $q$ and $q'$, respectively.  As in the proof of Lemma \ref{lem:c2} our choices of $X$ and $Y$ allow us to choose sectors $\sigma\in\Sigma_2(X)$, $\sigma'\in\Sigma_2(Y)$ such that $q$ has $\sigma$-coordinate $x(q)=(0,\alpha)$ and $q'$ has $\sigma'$-coordinates $(0,\beta)$ where $0<|\alpha|,|\beta|\leq 1$.  

The matrix $A$ in \eqref{eq:localmap} again has entries $a,d>0$, $b\geq 0$ and $c=0$; therefore $A\Log(x_1,x_2)$ is uniformly close to $(-a\log|x_1|,0)$ for $p=(x_1,x_2)\in U_q$.  If also, $(y_1,y_2) \in f_{\sigma\sigma'}(x_1,x_2)\in V_{q'}$, then $|\log|y_2||$ is bounded by definition.  Likewise, $\log|y_1|$ is bounded above.  Hence
$
\Log\circ f(p) = \tropf\circ\Log(p) + \err(p),
$
where $\err(p)$ is uniformly close to $(-\log|g_1|,0)$ for $p\in W\cap\torus$, and the proof of Theorem \ref{thm:tropapprox} concludes with the following.

\begin{lem}
\label{lem:c3}
If $U_q, V_{q'}$ are small enough, then $|g_1|$ is uniformly bounded below on $W\cap\torus$.
\end{lem}

\begin{proof}
Since $p\in\ind(f)$, $g_1$ need not be defined at $(0,\alpha)$ or finite on $U_q\cap\torus$.  Let $\pi:\hat X\to X$ be a (necessarily non-toric) blowup of $X$ at $p$ chosen so that neither $\hat f := f\circ\pi^{-1}$ and $\hat g_1 = g\circ \pi^{-1}$ have points of indeterminacy in $\pi^{-1}(U_q)$.  From here the argument resembles the one used to obtain the last two assertions in Proposition \ref{prop:tmapmore}.  Let $\hat W = \pi^{-1}(U_q) \cap \hat f^{-1}(V_{q'}) = \pi^{-1}(W)$.  Note that $\hat f$ maps $\{\hat g_1 = 0\}$ into $\{y_1=0\}$, i.e. outside $\torus$.  

On the other hand $\{\hat g_1 = 0\}$ properly intersects the proper transform of the pole $\{x_1=0\}$.  Hence $\{\hat g_1=0\}$ consists of components of $\pi^{-1}(q)$ together with components of the strict transforms of $\exc(f)$.  The latter map by $\hat f$ outside $V_{q'}$ and so do not meet $\hat W$.  Hence $\{\hat g_1=0\}\cap \hat W \subset \pi^{-1}(q)$ is a compact subset of $\hat W$ and must therefore be empty by the same reasoning used in the third paragraph of the proof of Proposition \ref{prop:tmapmore}.  Shrinking $U_q$ and $U_{q'}$ a bit if necessary, we conclude that $|\hat g_1|$ is bounded away from $0$ on $\hat W$; i.e. $|g_1|$ is bounded away from $0$ on $W$.
\end{proof}

%% file: eqcurrents4.tex
Any toric map $f:\rztO\tto\rztO$ induces linear operators $f^*,f_*:\pcc(\rzt)\to \pcc(\rzt)$.  In \cite[\S 8]{DiRo24} we defined these indirectly, working first on toric surfaces and then passing to inverse limits.  Here we give an equivalent definition using local potentials and working directly on $\rztO$.

Let $T\in\pcc(\rzt)$ be a toric current.  Since, by definition, $T$ is a difference of positive toric currents, we can assume $T\geq 0$.  Given $p\in \rztO\setminus\ind(f)$, let $U,V\subset \rztO$ be neighborhoods of $p,f(p)$ small enough that $f(U)\subset V$ and $T|_V = dd^c u$ for some $u\in\psh(V)$.  Then $f^*T|_U := dd^c (u\circ f)$.  The resulting current on $\rztO\setminus\ind(f)$ extends trivially across the finite set $\ind(f)$ to give a toric current $f^*T\in \pcc^+(\rzt)$.  Similarly, if $p\in \rztO\setminus f(\exc(f))$, we choose neighborhoods $V\supset p$ and $U\supset f^{-1}(V)$ small enough that $T = dd^c u$ for some $u\in\psh(U)$ and then set $f_* T|_V = dd^c f_*u$, where
$$
f_*u(p') := \sum_{f(q) = p'} m(f,q)u(q),
$$
and $m(f,q)$ denotes the local multiplicity of $f$ at $q$.  The resulting current on $\rztO\setminus f(\exc(f))$ extends again to a toric current $f_* T\in\pcc^+(\rzt)$.

\begin{prop}[\cite{DiRo24}, Proposition 8.3]
\label{prop:intext}
If $T\in\pcc(\rzt)$ is internal, so are $f^*T$ and $f_*T$.  If $T$ is an external divisor, then $f^*T = D + E$, where $D$ is an external divisor and $E$ is an internal divisor supported on $\exc(f)$.  Likewise, $f_* T = D' + E'$, where $D'$ is an external divisor and $E'$ is an internal divisor supported on $f(\ind(f))$.
\end{prop}

With these definitions it is readily apparent that $f^*,f_*$ preserve cohomological equivalence and so descend to operators $f^*,f_*:\hoo(\rzt)\to\hoo(\rzt)$.  These preserve effective and nef classes and are adjoint for the intersection product $\isect{f^*\alpha}{\beta} = \isect{\alpha}{f_*\beta}$ between classes $\alpha,\beta\in\eltwo(\rzt)$.

\begin{prop}[\cite{DiRo24}, Proposition 8.6] The following are equivalent for a toric map $f$
\begin{itemize}
 \item $f$ is internally stable;
 \item $(f^*)^n = (f^n)^*$ for all $n > 0$;
 \item $(f_*)^n = (f^n)_*$ for all $n > 0$.
\end{itemize}
\end{prop}

\subsection{Pullbacks and pushforwards of homogeneous currents}

The next result, which ensures good control of potentials for pullbacks and pushforwards of homogeneous currents, will play a key role in allowing us to construct measures of maximal entropy for toric maps.

\begin{thm}
\label{thm:pullback}
Let $f:\rztO\tto\rztO$ be a toric map, $T \in\pcc(\rzt)$ be an internal current and $\bar T$ be its homogenization.
\begin{enumerate}
 \item The potential $\rpot_{f^*\bar T}$ for $f^*\bar T - \overline{f^* T}$ is continuous and bounded outside any neighborhood of $\ind(f)$.
 \item If $T$ is positive, then $\rpot_{f^*\bar T}$ can be chosen to be negative. 
 \item Any support function $\sfn_{f^* \bar T}$ for $f^*\bar T$ is nearly homogeneous.
\end{enumerate}
\end{thm}

Note that since each class in $\hoo(\rzt)$ has a unique homogeneous representative, we have
for any internal $T\in\pcc(\rzt)$ that $\overline{f^* T} = \overline{f^* \bar T}$ and $\overline{f_* T} = \overline{f_*\bar T}$.

\begin{proof}
Any internal toric current is a difference of two positive internal currents, so we can assume that $T$ is positive.  Thus $\ch{T}$ is nef, and $\sfn_{\bar T}$ is  convex.  

We have on $\rztO$ that
$
dd^c (\sfn_{\bar T}\circ\Log) = \bar T - D_T$ and
$
dd^c (\sfn_{\overline{f^* T}}\circ\Log) = \overline{f^* T} - D_{f^* T}
$,
where $D_T, D_{f^*T}\in\pcc(\rzt)$ are nef external divisors. Hence
\begin{equation}
\label{eqn:id1}
dd^c \rpot_{f^*\bar T} = f^*\bar T - \overline{f^* T}
=
dd^c (\sfn_{\bar T}\circ\Log\circ f - \sfn_{\overline{f^*T}}\circ\Log) + f^* D_T - D_{f^*T}.
\end{equation}
In particular $f^*D_T$ is cohomologous to $D_{f^*T}$. So Proposition \ref{prop:intext} tells us that
$
f^* D_T - D_{f^*T} = E - D_E,
$
where $E$ is an internal divisor supported on $\exc(f)$, and $D_E$ is an external divisor cohomologous to $E$.

\begin{lem}
\label{lem:bdedepot}
Let $\bar E$ be the homogeneous internal current cohomologous to $E$.  Then the potential $\rpot_E$ for $E-\bar E$ is bounded outside any neighborhood $U\subset\rztO$ of $\supp E$, in particular outside a compact subset of $\rztO$.
\end{lem}

\begin{proof}
Since $\supp E$ is compact in $\rztO$, we can choose a toric surface $X$ be a such that $E$ is compactly contained in $X^\circ$.  It follows that the $\torus_\R$-average $E_{ave}$ of $E$ is also compactly supported in $X^\circ$. Formula (\ref{eqn:homogenization}) implies that $\supp\bar E \setminus \torus \subset \supp E_{ave}\setminus\torus $, so it further follows that $\bar E$ is compactly supported in $X^\circ$.  Thus, $\rpot_E|_{X^\circ}$ extends to $X$ as a function that is pluriharmonic in a neighborhood of each $\torus$-invariant point of $X$.  On the other hand, local potentials for $\bar E$ are continuous near any point in $X^\circ$, so it follows that $\rpot_E$ is continuous on $X\setminus\supp E$ and therefore bounded outside any open $U\subset X$ that contains $\supp E$.  Shrinking so that $U\subset X^\circ \subset \rztO$ concludes the proof.
\end{proof}

We return to the proof of Theorem \ref{thm:pullback}.  Potentials on $\rztO$ are unique up to additive constants, so adding and subtracting $\bar E$ on the right side of \eqref{eqn:id1}, applying Lemma \ref{lem:bdedepot} and dropping $dd^c$'s, leads to the identity
$$
\rpot_{f^*\bar T}
=
\sfn_{\bar T}\circ\Log\circ f - \sfn_{\overline{f^*T}}\circ\Log + \rpot_E - \sfn_{\bar E}\circ\Log + C_1
$$
for some constant $C_1$ and support function $\sfn_{\bar E}$ for $\bar E$.  Rearranging gives
\begin{equation}
\label{eqn:bothsideszero}
\rpot_{f^*\bar T} - v - \rpot_E
=
(\sfn_{\bar T}\circ \tropf - \sfn_{\overline{f^*T}} + \sfn_{\bar E})\circ \Log,
\end{equation}
where the function in parentheses on the right is positively homogeneous; and on the left side
$$
v := \sfn_{\bar T}\circ \Log\circ f - \sfn_{\bar T}\circ \tropf\circ \Log - C_1.
$$
We claim that both sides of \eqref{eqn:bothsideszero} actually vanish.

To see this, note that the left side of \eqref{eqn:bothsideszero} must also be homogeneous.  Hence it does not change if we average with the action of $\torus_\R$ and then homogenize (as in \eqref{eqn:homogenization}) each of the three terms on the left side.  On the other hand, $dd^c \rpot_{f^*T}$ and $dd^c \rpot_E$ are internal currents representing the trivial class in $\hoo(\rzt)$.  Since $0$ is the unique homogeneous current representing this class, it follows that the homogenizations of $\rpot_{f^*T}$ and $\rpot_E$ both vanish.

Concerning $v$, we note that since $\sfn_{\bar T}$ is convex and positively homogeneous, it is uniformly Lipschitz on $N_\R$.  Let $C_2<\infty$ be the Lipschitz norm.  Since $\exc(f)$ is internal and therefore compactly supported in $\rztO$, Corollary \ref{cor:star} tells us that it lies in the interior of a star $Q$ directed by finitely many rays in $N_\R$.  We invoke Theorem \ref{thm:tropapprox} to obtain a constant $C_3(Q)$ such that
$$
|v| \leq C_1 + C_2\norm{\Log\circ f - \tropf\circ \Log} \leq C_1 + C_2C_3.
$$
on $\torus\setminus Q$.  It follows that the homogenization of $v$ vanishes outside
$\Log(Q\cap\torus)$.  In particular, it vanishes asymptotically along all but finitely many rays of $N_\R$.  By homogeneity and continuity it vanishes altogether, proving our claim about the left side of \eqref{eqn:bothsideszero}.  Thus
$$
\rpot_{f^*\bar T} = \rpot_E + v.
$$
Proposition \ref{prop:bartonpole} tells us that local potentials for $\bar T$ are continuous, so local potentials for $f^*\bar T$ are also continuous except at points in $\ind(f)$.  Thus $\rpot_{f^*\bar T}$ is continuous on $\rztO\setminus\ind(f).$  On the other hand, we have already seen that $\rpot_E$ and then $v$ are bounded outside the compact set $Q\subset\rztO$.  Hence $\rpot_{f^*\bar T}$ is bounded outside $Q$.  Since $Q$ is compact, and $\rpot_{f^*\bar T}$ is continuous away from $\ind(f)$, we conclude that $\rpot_{f^*\bar T}$ is bounded outside any neighborhood of $\ind(f)$, which is the first conclusion of Theorem \ref{thm:pullback}.

Assume now, to get the second conclusion that $T$ is positive.  Then $\bar T$ and $f^*\bar T\geq 0$ are positive, too, so that local potentials for each are bounded above on compact sets.  On the other hand, we have from Proposition \ref{prop:bartonpole}, that local potentials for $\bar T^*$ are also uniformly bounded below on compact sets.  It follows that we can subtract a constant from $\rpot_{f^*\bar T^*}$ to arrange $\rpot_{f^*\bar T^*} \leq 0$ on a neighborhood of each of the finitely many points of $\ind(f)$.  By the first conclusion we can increase the constant so that $\rpot_{f^*\bar T^*}\leq 0$ everywhere on $\rztO$.

To obtain the third conclusion, we recall from the discussion after Theorem \ref{thm:homogenization} that if $\rpot_{f^*\bar T,ave}$ is the $\torus_\R$-average of $\rpot_{f^*\bar T}$, then (up to an additive constant)
$$
\sfn_{f^*\bar T,ave}\circ \Log = \sfn_{\overline{f^* T}}\circ \Log + \rpot_{f^* \bar T,ave}.
$$
Hence $\sfn_{f^*\bar T}$ differs from a positively homogeneous support function by a function that is bounded outside $\Log(Q \cap \torus)$ for the compact star $Q$.  So Proposition \ref{prop:nearlyhomogenough} implies that $\sfn_{f^*\bar T}$ is nearly homogeneous.
\end{proof}

The analog of Theorem \ref{thm:pullback} for pushforwards also holds.

\begin{thm}
\label{thm:pushforward}
Suppose that $f:\rztO\to\rztO$ is a toric map whose tropicalization $\tropf:N_\R\to N_\R$ is injective (hence a homeomorphism).  Let $T \in\pcc(\rzt)$ be an internal current and $\bar T$ be its homogenization.
\begin{enumerate}
 \item The potential $\rpot_{f_*\bar T}$ for $f_*\bar T - \overline{f_* T}$ is continuous and bounded outside any neighborhood of $f(\exc(f))$.
 \item If $T$ is positive, then $\rpot_{f_*\bar T}$ can be chosen to be negative everywhere on $\rztO$.
 \item The support function $\sfn_{f_* \bar T}$ for $f_*\bar T$ is nearly homogeneous.
\end{enumerate}
\end{thm}

\begin{proof}
Mostly the argument is the same as for Theorem \ref{thm:pullback}, with $f^*$ replaced by $f_*$ throughout.  The extra hypothesis gives us the necessary control of the difference 
\begin{align*}
v = f_*(\sfn_{\bar T}\circ\Log) - \dtop(\tropf^{-1} \sfn_{\bar T}) \circ\Log
\end{align*} 
as follows.   
If $U\subset \rztO$ is any neighborhood of $f(\ind(f))$, then by Theorem \ref{thm:tmapbasics}(1) there is a neighborhood $V\subset\rztO$ of $\exc(f)$ such that $f(V)\subset U$.  Hence any $p\in \torus \setminus U$ has exactly $\dtop$ distinct preimages $q\in \rztO$, and all such $q$ lies in $\torus\setminus V$.  Hence by Theorem \ref{thm:tropapprox}
$$
\norm{\tropf\circ\Log(q) - \Log(p)} \leq C
$$
for some constant $C = C(V)$.  Since $\tropf$ is a piecewise linear homeomorphism, it has an inverse that is piecewise linear and therefore uniformly Lipschitz.  Hence, after scaling $C$ by the Lipschitz constant of $\tropf^{-1}$, we obtain
$$
\norm{\Log(q) - \tropf^{-1}\circ \Log(q)} \leq C.
$$
Thus for $p \in \torus \setminus V$ the definition of pushfoward of a function gives us
\begin{align*}
v(p) = &\left\|\sum_{q \in f^{-1}(p)} \sfn_{\bar T} \circ \Log(q) - \sum_{q \in f^{-1}(p)} (\sfn_{\bar T} \circ \tropf^{-1}) \circ \Log p\right\| \\
&\leq
{\rm Lip}(\sfn_{\bar T}) \norm{\sum_{q \in f^{-1}(p)} (\Log(q) - \tropf^{-1}\circ\Log(p))}
\leq
C\dtop,
\end{align*}
where ${\rm Lip}(\sfn_{\bar T})$ denotes the Lipschitz constant of $\sfn_{\bar T}$.
\end{proof}


\subsection{Equilibrium currents.}
\label{subsec:eq}

%

Following ideas of \cite{BFJ08} we showed in \cite{DiRo24} that when $f$ is internally stable, the first dynamical degree $\ddeg(f)$ described in \S\ref{sec:intro} is the leading eigenvalue of the operators $f^*$ and $f_*$ on $\hoo(\rztO)$.  With more restrictions, we showed that this can be seen in a particularly strong way on the level of currents.

\begin{thm}
\label{thm:invcurrentsexist}
Let $f:\rztO\tto\rztO$ be an internally stable toric rational map  with small topological degree whose tropicalization $\tropf$ is a homeomorphism with irrational rotation number. Then there are internal currents $T^*,T_*\in\pcc^+(\rzt)$ uniquely determined by the following conditions
\begin{itemize}
 \item $\isect{T^*}{T^*} = \isect{T^*}{T_*} = 1$.
 \item for any other internal current $T\in\pcc(\rzt)$, we have
 $$
  \lim_{n\to\infty} \frac{f^{n*} T}{\ddeg^n} = \isect{T}{T_*}T^*
 \qquad\text{and}\qquad
 \lim_{n\to\infty}\frac{f^n_* T}{\ddeg^n} = \isect{T}{T^*} T_*
 $$
\end{itemize}
\end{thm}

We will call $T_*$ and $T^*$ the forward and backward \emph{equilibrium currents} for $f$.  The hypotheses imposed on $f$ by Theorem \ref{thm:invcurrentsexist} will be default assumptions for the remainder of this section and throughout the rest of this paper.  The condition $\ddeg(f) > \dtop(f)$ implies among other things that $f$ is not a monomial map.  The condition that $\tropf$ has irrational rotation number implies that $f$ is not invertible, i.e. that $\dtop(f)>1$.

Our goal in this subsection is to prove several results about the equilibrium currents $T^*$ and $T_*$ that will be used below but were not established in \cite{DiRo24}.

\begin{prop}
\label{prop:iskahler}
Both currents $T^*$ and $T_*$ in Theorem \ref{thm:invcurrentsexist} represent K\"ahler classes in $\hoo(\rzt)$.
\end{prop}

\begin{proof}
We give the argument for $\ch{T^*}$.  The argument for $\ch{T_*}$ is identical.  Fix a toric surface $X$.  Since $T^*$ is positive, non-trivial and internal, we have $\isect{T^*}{C_\tau} > 0$ for \emph{some} pole $C_\tau\subset\rztO$.  So for every $n\geq 0$, we have
$$
\isect{T^*}{C_{\tropf^n(\tau))}}
=
\isect{T^*}{f^n(C_{\tau})}
=
c_n \isect{T^*}{f^n_* C_\tau}
=
c_n\isect{f^{n*} T^*}{C_\tau}
=
c_n\ddeg^n \isect{T^*}{C_\tau} > 0,
$$
where $c_n = \deg(f^n|_{C_\tau})^{-1}$.  Since $\tropf$ has irrational rotation number, we conclude that $T^*$ has positive intersection number with poles indexed by a dense set of rational rays in $N_\R$.

Now fix a K\"ahler surface $X$, a sector $\sigma\in \Sigma_2(X)$, and $n\geq 0$ such that $\tropf^n(\tau)$ lies in the interior of $\sigma$.  Let $Y\succ X$ be a toric surface such that $f^n(\tau)\in\Sigma_1(Y)$.  If $\tau'\in\Sigma_1(X)$ is one of the rays bounding $\sigma$, then $\pi_{YX}^* C_{\tau'}\in \div(Y)$ is an effective $\Z$-divisor that includes $C_{f^n(\tau)}$ in its support.  Hence
$$
\isect{T^*_X}{C_{\tau'}}_X = \isect{\pi_{YX*} T^*_Y}{C_{\tau'}}_X = \isect{T^*_Y}{\pi_{YX}^* C_{\tau'}}_Y \geq \isect{T^*_Y}{C_{f^n(\tau)}}_Y
\geq 
\isect{T^*}{C_{f^n(\tau)}} > 0.
$$
Thus $\ch{T^*}$ has positive intersection with every pole of $X$, which implies that $\ch[X]{T^*_X}$ is K\"ahler.  We conclude that, by our definition, $\ch{T^*}$ is K\"ahler in $\hoo(\rzt)$.
\end{proof}

To prove existence of the equilibrium currents in \cite{DiRo24}, we first established the existence of invariant \emph{classes} $\ch{\bar T^*}, \ch{T_*} \in \hoo(\rztO)$ and then employed a telescoping series argument, which goes as follows for $T^*$.  By uniqueness of the homogeneous current $\bar T^*$ representing $\ch{T^*}$ and invariance, we have that $\overline{f^*\bar T^*} = \ddeg \bar T^*$.  Let, therefore, $\rpot := \rpot_{\ddeg^{-1} f^*T^*}$ be a potential for $\ddeg^{-1} f^*T^*-\bar T^*$.  We showed that the sum
\begin{equation}
\label{eqn:series}
\rpot_{T^*} := \sum_{n=0}^\infty \frac{\rpot\circ f^n}{\ddeg^n}
\end{equation}
converges in $L^1_{loc}(\rztO)$, and that the backward equilibrium current is then $T^* := \bar T^* + dd^c\rpot_{T^*}$.  

Theorem \ref{thm:pullback} allows us to assume that $\rpot$ is negative.  Since all partial sums $S_n$ of the series~\eqref{eqn:series} satisfy $dd^c S_n \geq - \bar T^*$, we have the following strengthening of the results in \cite{DiRo24}.

\begin{cor} The partial sums of the series \eqref{eqn:series} defining $\rpot_{T^*}$ are monotone decreasing and converge pointwise to $\rpot_{T^*}$.  The same is true of the analogous series defining $T_*$.
\end{cor}

Recall the `extended indeterminacy sets' $\ind(f^\infty)$ and $\ind(f^{-\infty})$ defined for internally stable toric maps after Definition \ref{defn:intstable}.

\begin{cor}
\label{cor:cvgceonpoles}
Under the hypotheses of Theorem \ref{thm:invcurrentsexist}, we have the following for any pole $C_\tau\subset\rztO$.
\begin{enumerate}
\item The intersections $C_\tau\cap \ind(f^\infty)$ and $C_\tau\cap \ind(f^{-\infty})$ are finite.
\item $\rpot_{T^*}|_{C_\tau}$ is continuous off $\ind(f^\infty)$ and $\rpot_{T_*}|_{C_\tau}$ is continuous off $\ind(f^{-\infty})$.
\item There is a uniform constant $C$, independent of $\tau$, such that if $C_\tau$ does not meet $\ind(f^\infty)$, then $|\rpot_{T^*}|\leq C$ on $C_\tau$, and if $C_\tau$ does not meet $\ind(f^{-\infty})$ then $|\rpot_{T_*}|\leq C$ on $\tau$. 
\end{enumerate}
\end{cor}

We will prove a much stronger version of this result for $\rpot_{T^*}$ below (see Theorem \ref{THM:CONTINUITY}).

\begin{proof}
The hypothesis that $\tropf$ is a homeomorphism with irrational rotation number implies that every ray in the bi-infinite sequence $(\tropf^n(\tau))_{n\in\Z}$ is distinct.  Since there are only finitely many points in $\ind(f)$ and $f(\exc(f))$, and $f^n(C_\tau) = C_{A^n_f(\tau)}$, we see that $f^n(C_\tau)$ contains points in $\ind(f)$ and $f(\exc(f))$ for only finitely many $n$.  This proves the first assertion. 

Given $p\in C_\tau\setminus \ind(f^\infty)$, it further follows that there are
neighborhoods $U,V\subset \rztO$ of $p$ and $\ind(f)$ and such that $f^n(U\cap
C_\tau)\cap V = \emptyset$ for all $n\in\N$.  Theorem \ref{thm:pullback} tells
us that $\rpot$ is continuous off $V$, so $\rpot\circ f^n$ is continuous on
$U\cap C_\tau$ for all $n\geq 0$.  Moreover, the values $|(\rpot\circ f^n)(q)|$
are uniformly bounded in both $n$ and $q$.  The series \eqref{eqn:series}
therefore converges uniformly to a continuous function $\rpot_{T^*}$ on $U$.
In the particular case that $f^n(C_\tau)\cap \ind(f) = \emptyset$ for all
$n\geq 0$, we have that $f^n(C_\tau) = C_{\tropf^n(\tau)}$ can only intersect a pole $C_{\tau'}$ containing some point of $\ind(f)$ at one of the two 
$\torus$-invariant points of $C_{\tau'}$.
We therefore can choose $U$ to contain all of $C_\tau$ and $V$ to be any neighborhood of $\ind(f)$ that meets only poles containing points in $\ind(f)$.  So $V$ is independent of $\tau$ in this case, which implies that the series \eqref{eqn:series} is uniformly bounded by a constant that does not depend on $\tau$.  This proves the second and third assertions for $\rpot_{T^*}$.  The arguments for $\rpot_{T_*}$ are identical.
\end{proof}

\begin{rem}
Theorem \ref{thm:pullback} and Corollary \ref{cor:cvgceonpoles} together give an alternative proof of the existence of the equilibrium current $T^*$, i.e. of the convergence of the series \eqref{eqn:series}.  As noted above, Theorem \ref{thm:pullback} allows one to assume that the partial sums of \eqref{eqn:series} are decreasing.  Corollary \ref{cor:cvgceonpoles} shows that the limit is not $-\infty$ everywhere on $\rztO$.  The fact that the partial sums are all $\bar T^*$-psh and Hartog's Compactness Theorem \cite[Theorem 1.26]{GZ_book} for psh functions therefore imply that the series \eqref{eqn:series} converges both pointwise and in $L^1$.  While this is shorter than the argument in \cite{DiRo24}, it is not sufficient for our purposes.  We need that pullbacks of more general positive internal currents, in particular those not cohomologous to multiples of $\bar T^*$, converge to multiples of $T^*$.  Showing this still seems to require the methods of our earlier paper.
\end{rem}

\begin{prop}
\label{prop:comesdetached}
Suppose that $f$ is a toric map satisfying the hypotheses of Theorem \ref{thm:invcurrentsexist} and $C\subset\rztO$ is a curve.  Then for all $n\geq 0$ large enough, we have
$$
(f^n(C)\setminus\torus)\cap \ind(f^\infty) = \emptyset = (f^{-n}(C)\setminus\torus)\cap \ind(f^{-\infty}).
$$ 
\end{prop}

\begin{proof}
The restriction $\tilde f:\rztO\setminus\torus \to \rztO\setminus\torus$ of $f$ to the countably many poles of $\rztO$ is a well-defined (even at points in $\ind(f)$) finite holomorphic map satisfying $\tilde f(C_\tau) = C_{\tropf(\tau)}$ for each pole $C_\tau$.  And for any $n\geq 0$, we have
$$
f^n(C)\setminus\torus \subset \tilde f^n(C\setminus\torus) \cup f^n(\exc(f^n)).
$$
By internal stability $f^n(\exc(f^n))\cap \ind(f^\infty) = \emptyset$.  So $\ind(f^n)\cap(f^n(C)\setminus\torus) \subset \tilde f^n(C\setminus\torus)$.  On the other hand, only finitely many poles meet $\ind(f)$, and since $\tropf$ has irrational rotation number no pole is periodic.  So for any pole $C_\tau$ that there exists $N\geq 0$ such that $C_{\tropf^n(\tau)}$ is disjoint from $\ind(f)$ for all $n\geq N$.  Since $C$ itself meets only finitely poles of $\rztO$, it follows that $\tilde f^n(C\setminus\torus)\cap \ind(f) = \emptyset$ for all $n$ large enough.  Hence
$
(f^n(C)\setminus\torus)\cap \ind(f^\infty) = \emptyset
$
for all such $n$.  The proof that $f^{-n}(C)\setminus\torus$ is disjoint from $\ind(f^{-\infty})$ for large $n$ is similar.
\end{proof}

\begin{cor} 
\label{cor:nidinfinity}
Suppose that $f$ is a toric map as in Theorem \ref{thm:invcurrentsexist} and $C\subset\rztO$ is a curve.  Then the restrictions $\rpot_{T^*}|_C$, $\rpot_{T_*}|_C$ are not identically $-\infty$.  In particular neither $T^*$ nor $T_*$ charge curves in $\rztO$.
\end{cor}

\begin{proof}  This time we focus on $T_*$ instead of $T^*$.
Suppose to get a contradiction that $\rpot_{T_*}|_C \equiv -\infty$.  By invariance of $T_*$, we have that (up to additive constants)
$$
\ddeg\rpot_{T_*} = \rpot_{f_* T_*} = \ddeg f_*\rpot_{T_*} + \rpot_{f_* \bar T_*}.
$$
The function $\rpot_{f_*\bar T_*}$ is continuous everywhere on $\rztO\setminus f(\exc(f))$ by Theorem \ref{thm:pushforward}.  Hence $f_*\rpot_{T_*}|_C \equiv -\infty$.  From the definition of pushforward, we infer that $\rpot_{T_*}|_{C_1} \equiv -\infty$, for some curve $C_1$ (an irreducible component of $f^{-1}(C)$) such that $f(C_1) = C_0$.  Repeated application of the same argument gives a sequence of curves $(C_j)_{j\geq 0}$ with $C_0 = C$, such that $f(C_{j+1}) = C_j$ and $\rpot_{T_*}|_{C_j} \equiv -\infty$ for all $j$.  But Proposition \ref{prop:comesdetached} implies that $C_n\setminus \torus$ is disjoint from $\ind(f^{-\infty})$ when $n$ is large enough.  Hence Corollary \ref{cor:cvgceonpoles} tells us that $\rpot_{T_*}$ cannot be infinite at any point of $C_n\setminus\torus$.  Since every curve contains \emph{some} point outside $\torus$, we have our contradiction.  It follows immediately that $T_*$ does not charge $C$.
\end{proof}

In the case of $T^*$ we can improve a bit on Corollary \ref{cor:nidinfinity}.

\begin{cor}
\label{cor:nodisk}
The Lelong numbers $\nu(T^*,p)$ of the equilibrium current $T^*$ vanish except when $p\in\ind(f^\infty)$.  Hence $T^*$ does not charge analytic disks in $\rztO$.
\end{cor}

\begin{proof}
The statement about Lelong numbers is \cite[Theorem 10.2]{DiRo24}.  Hence the Lelong numbers of $T^*$ are positive at most countably many points in $\rztO$.  This is inconsistent with the fact that a positive closed $(1,1)$ current that charges a non-trivial analytic disk if and only if its Lelong number is positive at all (uncountably many) points in the disk. 
\end{proof}

\begin{rem}
It seems likely to us that, since $f$ has small topological degree, Lelong numbers of $T_*$ vanish except at points in $\ind(f^{-\infty})$.  One can at least show without difficulty that these numbers vanish except at points in the forward orbit of $\ind(f)$ (a countable union of curves).  This weaker fact plus Corollary \ref{cor:nidinfinity} is sufficient to verify that $T_*$ doesn't charge analytic disks either.  But we do not use this in the sequel.
\end{rem}

\begin{thm} 
\label{thm:eqnearhom}
Support functions $\sfn_{T^*}$ and $\sfn_{T_*}$ for the currents $T^*$ and $T_*$ in Theorem \ref{thm:invcurrentsexist} are nearly homogeneous.
\end{thm}

\begin{proof} 
Again, we prove the result only for $T^*$.  Let $\sfn_{T^*}$ be a support function for $T^*$ and (then) $\sfn_{\bar T^*}$ be the support function for $\bar T^*$ obtained by homogenizing $\sfn_{T^*}$.  Let $T^*_{ave}$ be the $\torus_\R$-average of $T^*$.  Recall that $\rpot_{T^*_{ave}} := (\sfn_{T^*} - \sfn_{\bar T^*})\circ\Log$ extends from $\torus$ to $\rztO$ as a potential for $T^*_{ave} - \bar T^*$.  Recall from Theorem \ref{thm:homogenization} and Proposition \ref{prop:average} that $\sfn_{T^*} - \sfn_{\bar T^*}$ is continuous, non-positive and non-increasing along every ray $\tau\subset N_\R$.  Our goal is to prove that it is uniformly bounded below.  By Proposition~\ref{prop:nearlyhomogenough}, it suffices to prove this only on a dense set of rays $\tau\subset N_\R$.

Let $C_\tau$ be (as in e.g. Corollary \ref{cor:nidinfinity}) a pole such that $C_{\tropf^n(\tau)}\cap\ind(f) = \emptyset$ for all $n\geq 0$.  Since moreover, $\tropf$ has irrational rotation number, the rays $\tropf^n(\tau)$ are dense in $N_\R$.  Corollary \ref{cor:cvgceonpoles} tells us that the potential $\rpot_{T^*}$ for $T^*-\bar T^*$ given by \eqref{eqn:series} is uniformly bounded on $\bigcup_{n\geq 0} C_{\tropf^n(\tau)}$.  Hence the the $\torus_\R$-average $\rpot_{T^*,ave}$ of $\rpot_{T^*}$ is uniformly bounded on the same set.  Since $\rpot_{T^*,ave}$ is another potential for $T^*_{ave}-\bar T^*$ it differs from the function $\rpot_{T^*_{ave}}$ in the previous paragraph by a constant.  On the other hand, we have for any $v\in N$ that
$$
\lim_{t\to\infty} (\sfn_{T^*} - \sfn_{\bar T^*})(tv) = \rpot_{T^*,ave}(p),
$$
where $p = \lim_{z\to 0} \gamma_v(z) \in C_{\tropf^n(\tau)}$ for $\gamma_v:\C^*\to \torus$ the one parameter subgroup associated to $v$.
It follows that $\sfn_{T^*} - \sfn_{\bar T^*}$ is uniformly bounded along $\tropf^n(\tau)$ by a constant that does not depend on $n$. Proposition \ref{prop:nearlyhomogenough} now yields that the support function $\sfn_{T^*}$ is nearly homogeneous.
\end{proof}

%% file: eqcurrents_continued2.tex
From now until the end of the paper, we take $f:\rztO\tto\rztO$ to be a toric rational map satisfying the hypotheses of Theorem \ref{thm:mainthm}.  That is,
\begin{enumerate}
 \item\label{stability} $f$ is internally stable.
 \item\label{smalltop} $\dtop(f) < \ddeg(f)$.
 \item\label{rotation} $\tropf$ is a homeomorphism with irrational rotation number.
 \item\label{indout} $\ind(f)\cap\torus = \emptyset$.
\end{enumerate}
Assumption \eqref{stability} implies that the forward and backward indeterminacy orbits $\ind(f^\infty)$ and $\ind(f^{-\infty})$ are disjoint from each other.  Assumption \eqref{indout} implies they are also disjoint from $\torus$.  Assumption \ref{rotation} implies that they meet each pole of $\rztO$ in at most finitely many points.  Hence $\ind(f^\infty)$ and $\ind(f^{-\infty})$ are each closed and discrete in $\rztO$.

The results about equilibrium currents in \S\ref{sec:currents} all apply more or less equally to $T^*$ and $T_*$.  The main result of this section, which directly implies Theorem \ref{thm:ctyresult} above, is specific to $T^*$.  We do not know whether  the analogue for $T_*$ is also true.  

\begin{thm}\label{THM:CONTINUITY}
If $f:\rztO\tto \rztO$ is a toric map satisfying the hypotheses of Theorem \ref{thm:mainthm}, then the series \eqref{eqn:series} defining a potential $\rpot_{T^*}$ for $T^*-\bar T^*$ converges uniformly on any compact subset of $\rztO\setminus\ind(f^\infty)$ to a continuous function.
\end{thm}

For the remainder of this section, we fix an $\R$-linear identification $N_\R\cong \C$.  We let $\norm{\cdot}$ denote the norm and $\sphericalangle(\cdot,\cdot)$ the angle measure induced by the standard Euclidean metric on $\C$.  As the next result makes precise, the  condition on the rotation number of $\tropf$ guarantees that the tropicalization $\trop{f^n} = \tropf^n$ is nearly conformal for large $n$.  On the other hand, all of the hypotheses of Theorem \ref{THM:CONTINUITY} remain true if we replace the map $f$ by an iterate, so we will be able to assume henceforth that $\tropf$ itself is almost conformal. 

\begin{lem}\label{LEM:ALMOST_CONFORMAL}
For any fixed $\delta>0$, we may replace $f$ in Theorem \ref{THM:CONTINUITY} with an iterate $f^n$ in order to arrange that
\begin{enumerate}
\item 
\label{EQUATION:HYPOTHESIS_ALMOST_CONFORMAL}
for any $v\in N_\R$
$$
{\rm e}^{-\delta}\dtop^{1/2}  \leq \frac{\|\tropf \nvec \|}{\|\nvec\|} \leq {\rm e}^{\delta}\dtop^{1/2}; 
$$
\item 
\label{EQN:ANGLE_CHANGE}
and consequently for any $v,v'\in N_\R$.
$$
\sphericalangle (\tropf \nvec,\tropf \nvec') \leq {\rm e}^{\delta} \sphericalangle (\nvec,\nvec').
$$
\end{enumerate}
\end{lem}

\begin{proof}  Conclusion \eqref{EQUATION:HYPOTHESIS_ALMOST_CONFORMAL} follows from \cite[Theorem 5.5]{DiRo24}, so it remains to show how this implies the second conclusion.  Recall from Theorem \ref{thm:tropicalization} that $N_\R$ decomposes into a finite union of closed sectors $\overline{\sigma}$ such that $\tropf$ is linear on each.  Since $\tropf$ is a homeomorphism, the second and last conclusions of Theorem \ref{thm:tropicalization} give us that $\dtop(f)=|\det\tropf|$ on each $\overline{\sigma}$.  

It suffices to prove (\ref{EQN:ANGLE_CHANGE}) when $\sphericalangle(v,v')$ is small.  In particular, we can assume $v,v'$ lie in some sector $\overline{\sigma}$ where $\tropf$ is linear. Using our identification $N_\R\cong \C$, we obtain $a,b\in\C$ such that (without loss of generality) $|a| > |b|$
$$
\tropf(z) = az + b\bar z
$$
on $\overline{\sigma}$.  Conclusion \eqref{EQUATION:HYPOTHESIS_ALMOST_CONFORMAL} then gives
\begin{align*}
1+2 \frac{|b|}{|a|} \leq \frac{|a|+|b|}{|a|-|b|} < e^{2\delta}.
\end{align*}
and hence $1+3 |b|/|a| < e^{3\delta}$.
Let $\theta>\phi$ be the arguments of $v$ and $v'$. Then 
\begin{align*}
\sphericalangle (\tropf \nvec,\tropf \nvec') &= (\theta-\phi)+{\rm arg}\left(\frac{1+e^{-i2\theta} b/a}{1+e^{-i2\phi} b/a}\right) \\
&\leq (\theta-\phi)+{\rm arg}\left(1+(b/a)(e^{-i2\theta}-e^{-i2\phi})+\cdots\right) \\
&\leq (\theta-\phi)+ 3|b/a|(\theta-\phi) = (1+3|b/a|) \sphericalangle(\nvec,\nvec') < e^{3\delta} \sphericalangle(\nvec,\nvec').
\end{align*}
Replacing $\delta$ with $\delta/3$ completes the proof.
\end{proof}

With Lemma \ref{LEM:ALMOST_CONFORMAL} we derive some consequences of Theorem \ref{thm:tropapprox} for `tropical approximation' of forward iterates of $f$.  

\begin{prop} \label{PROP:ITERATIVE_TROP_APPROX1} Suppose that the toric map $f$ in Theorem \ref{THM:CONTINUITY} satisfies the conclusions of Lemma \ref{LEM:ALMOST_CONFORMAL}.  Given open sets $U\supset \ind(f)$ and $V\supset f(\exc(f))$, there exist $C_1, C_2 > 0$ such that the following hold for any $p\in\torus$.  
\begin{itemize}
\item[(i)] If $f^{j}(p) \not \in U$ for all $0 \leq j \leq n-1$, then
\begin{align*}
\norm{\Log\circ f^n(p))} \geq \left(e^{-\delta} \dtop^{1/2} \right)^n \left(\norm{\Log(p)} - C_1 \right).
\end{align*}
\item[(ii)] If $f^{j+1}(p) \not \in V$ for all $0 \leq j \leq n-1$, then
\begin{align*}
\norm{\Log\circ f^n(p))} \leq \left(e^{\delta} \dtop^{1/2} \right)^n \left(\norm{\Log(p)} + C_1 \right).
\end{align*} 
\item[(iii)] If $f^{j}(p) \not \in U$, $f^{j+1}(p) \not \in V$ for all $0 \leq j \leq n-1$ and $\norm{\Log(p)} \geq 2C_1$, then for any $\tau \in N_\R$ we have
\begin{align*}
\sphericalangle(\Log\circ f^n(p)), \tropf^n(\tau)) \leq e^{\delta n}\left( \sphericalangle(\Log(p),\tau) + \frac{C_2}{\|\Log(p)\|}\right).
\end{align*}
\end{itemize}
\end{prop}

\begin{proof}
Concerning (i), let $C>0$ be the constant from \eqref{item:ta2} of Theorem \ref{thm:tropapprox} (with $U' := U$).  Iteratively applying that estimate together with conclusion (\ref{EQUATION:HYPOTHESIS_ALMOST_CONFORMAL}) of Lemma \ref{LEM:ALMOST_CONFORMAL} gives (i) with
\begin{align*}
C_1 = C \sum_{j=0}^\infty \left(e^{-\delta} \dtop^{1/2} \right)^{-j}.
\end{align*}
The proof of (ii) is similar.

Concerning (iii), we have from Conclusion \eqref{EQN:ANGLE_CHANGE} of Lemma \ref{LEM:ALMOST_CONFORMAL} that
\begin{eqnarray*}
\sphericalangle(\Log\circ f(p),\tropf(\tau)) & \leq & \sphericalangle(\Log\circ f(p), \tropf\circ\Log(p)) + \sphericalangle(\tropf\circ\Log(p),\tropf(\tau)) \\
& \leq & \sphericalangle(\Log\circ f(p), \tropf\circ\Log(p)) + e^\delta \sphericalangle(\Log(p),\tau).
\end{eqnarray*}
Meanwhile, the conditions $p \not \in U$, $f(p)\notin V$, and Conclusion \eqref{item:ta1} of Theorem \ref{thm:tropapprox} give
$\|\Log\circ f(p) - \tropf \Log(p)\| \leq C$ and therefore 
\begin{align}\label{EQN:ONE_STEP}
\sphericalangle(\Log\circ f(p), \tropf \Log(p)) \leq {\rm arcsin}\left(\frac{C}{\norm{\Log\circ f(p)}}\right) \leq \frac{\pi}{2} \frac{C}{\norm{\Log\circ f(p)}},
\end{align}
where we have used that ${\rm arcsin}(x) \leq \frac{\pi}{2} x$ for all $x \in [0,1]$.  Thus 
\begin{align*}
\sphericalangle(\Log\circ f(p),\tropf(\tau)) \leq e^\delta \sphericalangle(\Log(p),\tau) + \frac{\pi}{2} \frac{C}{\norm{\Log\circ f(p)}}.
\end{align*}
Repeating this estimate with $f(p),\dots f^{n-1}(p)$ in place of $p$, we obtain
\begin{eqnarray*}
\sphericalangle(\Log\circ f^n(p)), \tropf^n \tau) 
& \leq & 
\sphericalangle(\Log(p),\tau) e^{\delta n} + \frac{C\pi}{2} \sum_{j=1}^n \frac{e^{\delta (n-j)}}{\norm{\Log\circ f^j(p))}} \\
& \leq & 
e^{\delta n}\left(\sphericalangle(\Log(p),\tau) 
+ 
\sum_{j=1}^n \frac{C\pi}{2\dtop^{j/2} \left(\norm{\Log(p)} - C_1\right)} \right) \\
& \leq & 
e^{\delta n} \left(\sphericalangle(\Log(p),\tau) + \frac{C_2}{\norm{\Log(p)}} \right).
\end{eqnarray*}  
The second inequality follows from the assumption $\Log(p)\geq 2C_1$ and the lower bound on $\norm{\Log\circ f^j(p))}$ from (i). 
The third inequality follows from $\norm{\Log(p)}-C_1 \geq \frac12\norm{\Log(p)}$ and taking $C_2 := C\pi\sum_{j=0}^\infty \dtop^{-j/2}$. 
\end{proof}

\begin{proof}[Proof of Theorem \ref{THM:CONTINUITY}]
Since $f$ has small topological degree, we may choose $\delta>0$ so that 
\begin{align}\label{EQN:ASSUMPTION_ON_DELTA}
e^{3\delta} \dtop(f) < \tdeg < \ddeg(f).
\end{align}
We have $\dtop(f^n) = \dtop(f)^n$ and, by internal stability, $\ddeg(f^n) = \ddeg(f)^n$, so \eqref{EQN:ASSUMPTION_ON_DELTA} holds with the same $\delta$ if we replace $f$ with a forward iterate $f^n$.  Likewise, $f^n$ satisfies the hypotheses of Theorem \ref{THM:CONTINUITY} whenever $f$ does.  So since the equilibrium currents $T^*,T_*$ are the same for $f^n$ as they are for $f$, we may begin by replacing $f$ with an iterate in order to assume that the conclusions of Lemma \ref{LEM:ALMOST_CONFORMAL} hold for $\delta$ in \eqref{EQN:ASSUMPTION_ON_DELTA}.

Let $K\subset \rztO\setminus \ind(f^\infty)$ be a given compact set. By Corollary~\ref{cor:star}, $K\subset Q \equiv Q(\tau_1,\dots,\tau_J,R)$ for some star $Q\subset\rztO$.  Increasing $R$ and adding to the finite set $\Sigma_1(Q) := \{\tau_1,\dots,\tau_J\}$ of rational rays $\tau_j\subset N_\R$ that direct $Q$, we may assume that $Q$ contains the finite sets $\ind(f)$ and $f(\exc(f))$.   

Let $\rpot:=\rpot_{\ddeg^{-1} f^*\bar T^*}$ be as in \eqref{eqn:series}. Fix a point $q\in\ind(f)$ and a local coordinate $z$ centered at $q$.  Then (see e.g. \cite[Proposition 1.2]{BeDi05b}) $\rpot(z) \geq -C\log\norm{z}$ for some constant $C>0$.  Theorem \ref{thm:pullback} tells us $|\rpot|$ is bounded on the complement of any neighborhood of $\ind(f)$, so it will suffice to fix a Riemannian distance function on $\rztO$ and show that
\begin{align}\label{EQN:DESIRED_SERIES}
\sum_{j=0}^\infty \frac{\log \dist(f^j(p),\ind{f})}{\ddeg^j}
\end{align}
converges uniformly on $K$.  Each term of this series is continuous on $K$, so it will further suffice to prove this only on the dense subset $K\cap\torus$.

Since $\tropf$ has irrational rotation number and $\Sigma_1(Q)$ is finite, we can choose $n_0 \in \mathbb{N}$ sufficiently large
so that $\tropf^n(\tau) \not \in \Sigma_1(Q)$ for any $\tau\in \Sigma_1(Q)$ and $n\geq n_0$.  We then choose an open set $U\subset Q$ such that
\begin{equation}
\label{EQN:DEF_U1}
\ind(f)\subset U \quad 
\text{and}\quad 
U\cap\bigcup_{j=0}^{n_0} f^j(K) = \emptyset.
\end{equation}
Similarly, since $f$ is internally stable, we can choose an open set $V\subset Q$ such that 
\begin{equation}
\label{EQN:DEF_V}
f(\exc(f))\subset V \quad\text{and}\quad U \cap \bigcup_{j=0}^{n_0} f^j(V) = \emptyset.
\end{equation}
These are the `excluded' neighborhoods we will use when applying Proposition~\ref{PROP:ITERATIVE_TROP_APPROX1}.

Fix a number $\tdeg \in (e^{3\delta}\dtop,\ddeg)$ and suppose to get a contradiction that the series (\ref{EQN:DESIRED_SERIES}) does not converge
uniformly on $K \cap \torus$.   Then for arbitrarily large $n \in \mathbb{N}$ there exists $p = p(n) \in K \cap \torus$ such that 
\begin{align}\label{EQN:HYPOTHETICAL_SITUATION}
|\log \dist(f^n(p),\ind(f))| \geq \tdeg^n.
\end{align}
Hence (for $n$ large enough) $f^n(p) \in U$, and there is a constant $C_3>0$ such that 
\begin{align}\label{EQN:HYPOTHETICAL_SITUATION1}
\|\Log\circ f^n(p) \| \geq C_3 \tdeg^n
\end{align}
for all such $n$.  Finally, by Conclusion (1) of Corollary \ref{cor:cvgceonpoles}, $\ind(f^{-\infty})$ neither intersects nor accumulates on $\ind(f)$, so we may assume that $f^n(p)\notin\ind(f^{-\infty})$.  Hence the orbit segment $p,\dots,f^n(p)$ is completely contained in $\torus$.

From here the general idea is that since $\Log\circ f$ is well-approximated by $\tropf\circ\Log$, the fact that $p$ and $f^n(p)$ both lie in $Q$ suggests that there are arbitrarily large $n$ and rays $\tau,\tilde\tau\in \Sigma_1(Q)$ such that $\tropf^n(\tau) = \tilde\tau$.  But this is impossible, because the set $\Sigma_1(Q)$ of rays directing $Q$ is finite, whereas no ray in $N_\R$ is periodic by $\tropf$.  Proposition \ref{PROP:ITERATIVE_TROP_APPROX1} allows us to make this idea precise, but the details become a little involved since the relationship between $f$ and $\tropf$ breaks down for orbits that pass near $\ind(f)$  and/or $f(\exc(f))$.  In any case, to proceed we fix a positive integer $n$ as in \eqref{EQN:HYPOTHETICAL_SITUATION}.  All statements and inequalities that follow will be understood to hold for `sufficiently large' such $n$.  The meaning of `sufficiently large' will become clear gradually, increasing several (but finitely many!) times as the argument progresses.  All constants $C_j$ that appear will be otherwise independent of $n$ and will remain the same from the moment they are introduced.  

We let $q = f^{n-k}(p)$ be the last point in the orbit segment $p,f(p),\dots,f^n(p)$ that lies in the open set $V$ from \eqref{EQN:DEF_V}.  If no such point exists we take $q=p$, i.e. $k=n$.  Our choice of $q$ allows us to apply Proposition \ref{PROP:ITERATIVE_TROP_APPROX1} to obtain
$$
C_3\tdeg^n \leq \norm{\Log\circ f^n(p)} = \norm{\Log\circ f^k(q)} \leq \left(e^{\delta} \dtop^{1/2}\right)^k (\norm{\Log(q)} + C_1).
$$
Since $\tdeg > e^{3\delta} \dtop > e^\delta\dtop^{1/2}$, this implies for sufficiently large $n$ that 
\begin{align}
\label{EQN:LOWER_BOUND_LOG_P}
\|\Log(q)\|
\geq
\frac{C_3}{2} \frac{\tdeg^n}{(e^{\delta}\dtop^{1/2})^k} \geq {\rm max}(R,2C_1).
\end{align}
Since $q \in V \subset Q$, we also have a ray $\tau \in \Sigma_1(Q)$ such that ${\rm dist}(\Log(q),\tau) \leq R$, where $R$ is the width used in defining $Q$.  From this and \eqref{EQN:LOWER_BOUND_LOG_P}, we infer
\begin{align} \label{EQN:UPPER_BOUND_INITIAL_ANGLE}
\sphericalangle(\Log(q), \tau) \leq {\rm arcsin}\left(\frac{2R}{C_3} \cdot \frac{(e^{\delta} \dtop^{1/2})^k}{\tdeg^n}\right) \leq \frac{\pi R }{C_3} \cdot  \frac{(e^{\delta} \dtop^{1/2})^k}{\tdeg^n}.
\end{align}

Now let $\ell$ be the smallest positive integer such that $f^\ell(q) \in U$.  Since $q\in V$ and $f^k(q) = f^n(p)\in U$, we have $n_0 < \ell \leq k$ by \eqref{EQN:DEF_V}.  Then $U\cap\torus\subset Q$ implies $\dist(f^\ell(q),\tilde\tau) < R$ for some $\tilde\tau\in \Sigma_1(Q)$.   So since $\{f^{\ell+1}(q),\dots,f^n(q)\}\cap V = \emptyset$, we can repeat the argument for \eqref{EQN:UPPER_BOUND_INITIAL_ANGLE} to obtain
\begin{align} 
\label{EQN:UPPER_BOUND_INITIAL_ANGLE2}
\sphericalangle(\Log\circ f^\ell(q), \tilde\tau) \leq {\rm arcsin}\left(\frac{2R}{C_3} \cdot \frac{(e^{\delta} \dtop^{1/2})^{k-\ell}}{\tdeg^n}\right) \leq \frac{\pi R }{C_3} \cdot  \frac{(e^{\delta} \dtop^{1/2})^{k-\ell}}{\tdeg^n}.
\end{align}
Moreover, by construction, the orbit segment $\{q,f(q),\dots,f^\ell(q)\}$ meets $V$ only at $q$ and $U$ only at $f^\ell(q)$. 
So \eqref{EQN:LOWER_BOUND_LOG_P} and the last conclusion of Proposition \ref{PROP:ITERATIVE_TROP_APPROX1} give
\begin{equation}
\label{EQN:UPPER_BOUND_FINAL_ANGLE}
\sphericalangle(\Log\circ f^\ell(q), \tropf^\ell(\tau)) 
\leq 
e^{\delta \ell} \left(\sphericalangle(\Log(q),\tau)  +  \frac{C_2}{\|\Log(q)\|}\right)
\leq  
C_4 e^{\delta\ell}\cdot \frac{(e^\delta \dtop^{1/2})^k}{\tdeg^n}
\leq
C_4 \frac{(e^{2\delta} \dtop^{1/2})^k}{\tdeg^n}.
\end{equation}  
Together \eqref{EQN:UPPER_BOUND_INITIAL_ANGLE2} and \eqref{EQN:UPPER_BOUND_FINAL_ANGLE} imply that
\begin{equation}
\label{eqn:contra1}
\sphericalangle(\tropf^\ell(\tau), \tilde\tau) \leq C_5\frac{(e^{2\delta} \dtop^{1/2})^k}{\tdeg^n}.
\end{equation}
But $\tilde\tau\in \Sigma_1(Q)$, whereas $\ell\geq n_0$ implies that $\tropf^\ell(\tau)\notin\Sigma_1(Q)$, so we claim \eqref{eqn:contra1} is impossible for $n$ large.

To see this, let $v\in N\setminus\{0\}$ be the primitive vector on $\tau$.  Note that since $\tau$ is one of finitely many rays in $\Sigma_1(Q)$, we have $\norm{v}\leq C_6$ for some constant $C_6$ depending on $Q$ but not on the particular ray $\tau$ (i.e. not on $n$).  Finiteness of $\Sigma_1(Q)$ also implies that 
$$
\min\{\dist(\tilde v,\tilde\tau): \tilde v\in N\setminus \tilde \tau\} \geq C_7 > 0
$$
for some constant $C_7$ depending on $Q$ but not on the particular ray $\tilde\tau$.  Hence
$$
\sphericalangle(\tropf^\ell(\tau), \tilde\tau) 
= 
\sphericalangle(\tropf^\ell(v), \tilde\tau) 
\geq 
\frac{C_7}{\norm{\tropf^\ell v}} 
\geq 
\frac{C_7}{C_6 (e^\delta \dtop^{1/2})^\ell}.
$$
Putting this together with \eqref{eqn:contra1}, rearranging and using that $\ell\leq k \leq n$, we obtain that
$$
\tdeg^n \leq C_8 (e^{3\delta} \dtop)^n
$$
for all sufficiently large $n$.  This contradicts $\tdeg > e^{3\delta}\dtop$, justifying our claim, and completing the proof of Theorem \ref{THM:CONTINUITY}. 
\end{proof}

For later reference we note the following consequence of Theorem \ref{THM:CONTINUITY}.

\begin{cor}
\label{cor:curvectrl}
Let $f$ be as in Theorem \ref{THM:CONTINUITY} and $\dist$ denote a Riemannian distance on $\rztO$.  Then for any curve $C\subset\rztO$, the intersection $\ind(f^\infty)\cap C$ is finite, and there exist constants $A,B>0$
$$
\rpot_{T^*}(p) \geq  \min\{-B,A\log\dist(p,\ind(f^\infty)\cap C_\tau)\}.
$$
for all $p\in C$.
\end{cor}

\begin{proof}
Corollary \ref{cor:cvgceonpoles} gives that $\ind(f^\infty)\cap C$ is finite when $C$ is a pole.  Finiteness holds for internal $C$ because under the hypotheses of Theorem \ref{THM:CONTINUITY} $\ind(f^\infty)\cap C \subset C\setminus\torus$, and the latter set is always finite.  If in either case, $C\cap \ind(f^\infty) = \emptyset$, then $\rpot_{T^*}|_C$ is bounded.  When $C$ is a pole, this is the final conclusion of Corollary \ref{cor:cvgceonpoles}.  When $C$ is internal (and therefore compact in $\rztO$), this is an immediate consequence of Theorem \ref{THM:CONTINUITY}.

For a general curve $C$, we apply Proposition \ref{prop:comesdetached} to obtain $n>0$ such that $f^n(C)\cap \ind(f^\infty) = \emptyset$.  By invariance of $T^*$, we have
$$
\ddeg^n\rpot_{T^*} = \rpot_{T^*}\circ f^n + \rpot_{f^{n*}\bar T^*},
$$
where the first term on the right is bounded on $C$ by our choice of $n$, the second term is bounded above everywhere on $\rztO$ by Theorem \ref{thm:pullback}, and
$$
\rpot_{f^{n*}\bar T^*} \geq \max\{- A\,\log\dist(\cdot,\ind(f^n)),-B\}.
$$
for some $A,B>0$ by Conclusion (1) of Theorem \ref{thm:pullback} and e.g. \cite[Proposition 1.3]{BeDi05b}
\end{proof}

%% file: products7.tex
We begin this section by reviewing a general potential theoretic approach for intersecting positive closed $(1,1)$ currents that was initiated by Bedford and Taylor \cite{BeTa76} and further developed by many others.  Our treatment suffices for the purposes of this article, but nearly all of the results below hold in substantially greater generality (see e.g. \cite{DDG11}).
In Subsection \ref{ss:toricwedge} we extend the wedge product to be considered on the toric limit surface $\rztO$.

\subsection{Wedge product on a compact K\"ahler surface.}\label{SUBSEC:WEDGE_COMPACT}
Let $X,\omega$ be a compact K\"ahler surface and $\psh(\omega)$ denote the set of functions $u\in L^1(X)$ such that $\omega +dd^c u\geq 0$.  Note that $\psh(\omega) \subset \psh(C\omega)$ for any $C>1$ and if $\omega'$ is any other K\"ahler form, then there exists $C>1$ such that $\psh(C^{-1}\omega) \subset \psh(\omega') \subset \psh(C\omega')$.  Hence $\qpsh(X) := \bigcup_{C>0} \psh(C\omega)$ is independent of the choice of $\omega$.

\begin{defn} A \emph{regularizing sequence} for $u\in\psh(\omega)$ is a decreasing sequence $(u_j)\subset \psh(\omega)\cap C^\infty(X)$ that converges pointwise to $u$.
\end{defn}

A proof of the following can be found in the book by Guedj and Zeriahi \cite[Proposition 8.16]{GZ_book}

\begin{thm}
\label{thm:regularization}
Any $u\in PSH(\omega)$ admits a regularizing sequence.
\end{thm}

For the remainder of this subsection, we fix a positive closed $(1,1)$ current $S$
on $X$ and let $L^1(S)$ denote the set of Borel measureable functions on $X$
that are integrable with respect to the trace measure $\omega\wedge S$.  Then
$L^1(S)$ is independent of $\omega$ and includes e.g. all bounded 
Borel measurable functions, hence all bounded qpsh functions on $X$.
Less obviously $L^1(S)$ includes qpsh functions that are
unbounded on small enough sets.  The following particular case of \cite[III,
Theorem 4.5]{DemBook} will be useful to us in the sequel.

\begin{prop}
\label{prop:boundedish}
If $I\subset X$ is a finite set and $u\in L^1(X)$ is a qpsh function that is locally bounded on $X\setminus I$, then $u\in L^1(S)$ for any positive closed $(1,1)$ current $S$ on $X$.
\end{prop}

The main idea of the Bedford-Taylor construction (in our context) is that for
any $u\in \psh(\omega)\cap L^1(S)$, the product $uS$ is itself a $(1,1)$
current and we can formally set
\begin{equation}
\label{eqn:wedge}
T \wedge S := \omega\wedge S + dd^c(uS),
\end{equation}
where $T := \omega + dd^c u$.  We say in this case that `the wedge product of
$T$ and $S$ is well-defined'.  If, moreover, $(u_j)$ regularizes $u$, then the
monotone convergence theorem gives us that $u_jS \to uS$ as currents.  Hence
$(\omega+dd^c u_j)\wedge S \to T\wedge S$, and $T\wedge S$ is therefore a
positive Borel measure with total mass equal to that of $\omega\wedge S$.

\begin{lem}\label{LEM:CURRENT_PROJECTION_FORMULA}
Suppose that $X$ is a Kahler surface and $\pi: Y \rightarrow X$ is a birational morphism.   Suppose $T$ is a closed positive $(1,1)$ current on $Y$,
that $u \in QPSH(X)$ and that $u \circ \pi \in L^1(T)$.   Then
\begin{align*}
\pi_* \left( dd^c(u \circ \pi) \wedge_Y T \right) = dd^c u \wedge_X \pi_* T.
\end{align*}
Here, the $\pi_*$ on the left hand side of the equation denotes the pushforward of Borel measures.
\end{lem}

\begin{proof}
Since $u \in  QPSH(X)$ we can suppose that $u \leq 0$.
It suffices to prove for any smooth test function $g$ that 
\begin{align*}
\int_Y (g \circ \pi) \ dd^c (u \circ \pi) \wedge T = \int_X g \  dd^c u \wedge \pi_* T.
\end{align*}
Integration by parts gives that this is equivalent to 
\begin{align*}
\int_Y (u \circ \pi) dd^c (g \circ \pi)  \wedge T = \int_X u \ dd^c g \wedge \pi_* T.
\end{align*}
We can write $dd^c g = \omega_1 - \omega_2$ where $\omega_1$ and $\omega_2$ are Kahler forms on $X$.  Therefore it suffices
to show for any Kahler form $\omega$ on $X$ that
\begin{align}\label{EQN:CURRENT_PROJECTION_FORMULA}
\int (u \circ \pi) \pi^* \omega \wedge T = \int u \  \omega \wedge \pi_* T.
\end{align}
If $u$ is smooth this follows from the definition of the proper pushforward
$\pi_* T$.   Otherwise, one can choose a regularizing sequence $(u_j)$ for $u$.
Equation (\ref{EQN:CURRENT_PROJECTION_FORMULA}) holds with
$u$ replaced by $u_j$ for each $j \in \mathbb{N}$ and the result then follows for $u$ by applying the
monotone convergence theorem to both sides.
\end{proof}

The next result guarantees (among many other things) that when both currents involved represent K\"ahler classes, the Bedford-Taylor definition \eqref{eqn:wedge} of wedge product is symmetric.

\begin{thm}
\label{thm:decapprox1}
Given K\"ahler forms $\omega,\omega'$ on $X$, $u\in\psh(\omega)$ and $v\in \psh(\omega')$, let $T = \omega + dd^c u$ and $T' =\omega'+dd^c v$.  Then $u\in L^1(T')$ if and only if $v\in L^1(T)$.  If either (and therefore both) of the integrability conditions hold and $(u_j),(v_j)$ are regularizing sequences for $u$ and $v$, then we further have
$$
T\wedge T' = \lim_{j\to\infty} T_j\wedge T'_j
$$
weakly as Radon measures.
\end{thm}

We will obtain Theorem \ref{thm:decapprox1} as a special case of a more general result.  In order to state and prove the latter, we associate to the above current $S$ a non-negative and symmetric bilinear form on $C^\infty(X)$:
\begin{align}\label{EQN:DEF_SMOOTH_PAIRING}
\nrg{u,v}_S := \int du\wedge d^c v \wedge S = \int -u \,dd^c v\wedge S = \int -v\,dd^c u\wedge S.
\end{align}
We let $\norm[S]{u} := \nrg{u,u}_S^{1/2}$ denote the associated \emph{$S$-energy} seminorm.  When restricted to $\omega$-psh functions the pairing $\nrg{u,v}_S$ has a useful monotonicity property.

\begin{prop}
\label{prop:nrgbnds}
For any smooth functions $u,\ti u\in \psh(\omega)$ and $v,\ti v\in\psh(\omega')$ such that $u\geq \ti u$ and $v\geq \ti v$, we have
$$
\nrg{u,v}_S \leq \nrg{\ti u,\ti v}_S + \norm[L^1(S\wedge\omega')]{u-\ti u} + \norm[L^1(S\wedge\omega)]{v-\ti v}.
$$
In particular
\begin{align}\label{EQN:PROP_NRGBNDS_INEQ2}
\norm[S]{u}^2 - \norm[L^1(S\wedge\omega)]{u-\tilde u} \leq \nrg{u,\tilde u}_S \leq \norm[S]{\tilde u}^2 + \norm[L^1(S\wedge\omega)]{u-\tilde u}.
\end{align}
\end{prop}

\begin{proof}
Integration by parts gives the following identity.
\begin{equation}
\label{eqn:basicidentity}
\nrg{u,v}_S - \nrg{\ti u,\ti v}_S = \int (\ti u-u) \,dd^c v\wedge S + \int (\ti v - v)\,dd^c \tilde u\wedge S.
\end{equation}
Since $dd^c v \geq -\omega'$, $dd^c \tilde u \geq - \omega$ and $u\geq \ti u$, $v \geq \ti v$, we obtain
$$
\nrg{u,v}_S - \nrg{\ti u,\ti v}_S \leq \int (u - \ti u) \,\omega'\wedge S + \int (v-\ti v)\,\omega\wedge S.
$$
This gives the first inequality.  The second inequality follows from applying the first twice with $\omega = \omega'$:
$$
\nrg{u,u}_S \leq \nrg{u,\ti u}_S + \int (u-\ti u)\,\omega\wedge S \leq \nrg{\ti u,\ti u}_S + 2\int(u-\ti u)\,\omega\wedge S.
$$
\end{proof}

\begin{cor}
\label{cor:decapprox2}
Given $u,v\in \qpsh(X)\cap L^1(S)$ with regularizing sequences $(u_j)$ and $(v_j)$, we have
$$
\lim_{j\to\infty} \nrg{u_j,v_j}_S = \int -u\,dd^c v\wedge S = \int -v\,dd^c u\wedge S \in \R\cup\{+\infty\}. 
$$
In particular, $u\in L^1(dd^c v\wedge S)$ if and only if $v\in L^1(dd^c u\wedge S)$.
\end{cor}

The wedge products in the conclusion should be interpreted in the Bedford-Taylor sense: i.e. technically,
$$
dd^c u\wedge S = dd^c(uS),
$$
which is a signed measure with finite total mass on $X$.

\begin{proof} Since any two K\"ahler forms on $X$ are comparable, we may suppose that $u,v,u_j,v_j$ are all $\omega$-psh for the same K\"ahler form $\omega$.  Given indices $j<k$, we apply the estimate in Proposition \ref{prop:nrgbnds} to obtain
$$
\nrg{u_j,v_j}_S - \norm[L^1(\omega \wedge S)]{v_j-v_k} \leq \nrg{u_j,v_k}_S \leq \nrg{u_k,v_k}_S +
\norm[L^1(\omega \wedge S)]{u_j-u_k}.
$$
Since $u_j\to u$ and $v_j\to v$ in $L^1(S)$, the $L^1$ norms on both sides tend to $0$ as $j\to\infty$.  In particular $\nrg{u_j,v_j}_S$ is asymptotically increasing with $j$ and must converge to its supremum (which might be $\infty$).  Moreover,
$$
\lim_{j\to\infty} \nrg{u_j,v_j}_S \leq \lim_{j\to\infty}\lim_{k\to\infty} \nrg{u_j,v_k}_S \leq \lim_{k\to\infty} \nrg{u_k,v_k}_S.
$$
The middle term may be rewritten as
$$
\lim_{j\to\infty} \lim_{k\to\infty} \int -u_j \,dd^c v_k \wedge S  = \lim_{j\to\infty} \int -u_j \,dd^c v\wedge S = \int -u \,dd^c v\wedge S,
$$
where the first equality is from weak convergence of $v_k \wedge S$ to $v \wedge S$ and the second is by the monotone convergence theorem.
Hence $\nrg{u_j,v_j}_S \to \int -u\,dd^c v\wedge S$.  Switching the roles of $u$ and $v$ gives the other inequality in the corollary.
\end{proof}

\begin{proof}[Proof of Theorem \ref{thm:decapprox1}]
That $u\in L^1(T')$ if and only if $v \in L^1(T)$ follows from taking $S=\omega$ in the final assertion in Corollary \ref{cor:decapprox2}.
If $\psi\in C^\infty(X)$ is a test function, then $dd^c\psi = \omega_1 -\omega_2$ can be written as a difference of K\"ahler forms.   We then have  
$$
\lim_{j\to\infty} \int \psi\,dd^c u_j\wedge dd^c v_j = \lim_{j\to \infty} \int  u_j \,dd^c v_j \wedge dd^c\psi = \int u \,dd^c v\wedge dd^c\psi = \int \psi\,dd^c u\wedge dd^c v,
$$
with the middle equality justified by Corollary \ref{cor:decapprox2} (with $S$ equal to each $\omega_k$ for $k=1,2$) and the other equalities being integration by parts.
Hence
\begin{eqnarray*}
\int \psi \, T_j\wedge T'_j & = & \int \psi\, T_j \wedge \omega' + \int \psi\, \omega \wedge dd^c v_j + \int \psi \, dd^c u_j\wedge dd^c v_j \\
& \to & \int \psi\, T\wedge\omega' + \int \psi\, \omega \wedge dd^c v + \int \psi\,dd^c u\wedge dd^c v = \int\psi T \wedge T',
\end{eqnarray*}
which concludes the argument.
\end{proof}

Theorem \ref{thm:regularization} and Corollary \ref{cor:decapprox2} allow us to extend $\nrg{\cdot,\cdot}_S$ to all of $\qpsh(X)\cap L^1(S)$ via
\begin{equation}
\label{eqn:epairing}
\nrg{u,v}_S := \int -u\,dd^c v\wedge S
\end{equation}
Since $u \in \qpsh(X)$ is bounded above by some constant and $dd^c v$ is bounded below by a negative multiple of the trace measure $\omega\wedge S$ of $S$, the integral is well-defined, though possibly equal to $+\infty$. Approximating $u$ and $v$ with regularizing sequences shows that $\nrg{\cdot,\cdot}_S$ remains non-negative and symmetric.  Hence the Cauchy-Schwarz inequality implies it is finite on the real cone 
$$
\eclass(S) := \{u \in L^1(S) \cap \qpsh(X) \, : \, \norm[S]{u} <\infty\},
$$
which includes all bounded qpsh functions on $X$.  Theorem \ref{thm:regularization} and Corollary \ref{cor:decapprox2} immediately imply the following.

\begin{cor}
\label{cor:nrgbnds}
The identity \eqref{eqn:basicidentity} holds for all $u,v,u',v'\in \eclass(S)$.  Hence so do all conclusions of Proposition \ref{prop:nrgbnds}.  
\end{cor}

With a bit more effort, we can also broaden the scope of Corollary \ref{cor:decapprox2}.

\begin{cor}
\label{cor:decrlims2}
Let $(u_j)\subset \eclass(S)\cap\psh(\omega)$ be a decreasing sequence with limit $u\in L^1(S)$.  Then
$$
\lim_{j\to\infty} \norm[S]{u_j} = \norm[S]{u}. 
$$
Hence $\nrg{\cdot,\cdot}_S$ is similarly continuous with respect to decreasing limits on $\eclass(S)\cap\psh(\omega)$.
\end{cor}

\begin{proof}
Since $u\not\equiv-\infty$, it follows from Hartog's Compactness Lemma (e.g.
\cite[Theorem 1.46]{GZ_book}) that $u\in \psh(\omega)$ and by Monotone
convergence that $u_j\to u$ in $L^1(S\wedge \omega)$.  Hence, in particular,
$(u_j)$ is Cauchy in $L^1(S\wedge\omega)$.  Corollary
\ref{cor:nrgbnds} and (\ref{EQN:PROP_NRGBNDS_INEQ2})  therefore imply that
$\norm[S]{u_j}$ is essentially increasing and in particular that
$\lim_{j\to\infty} \norm[S]{u_j} = \sup\norm[S]{u_j}$ exists.

We first show that $\lim\norm[S]{u_j} \leq \norm[S]{u}$.  This is immediate if $\norm[S]{u} = \infty$.  Otherwise, $u \in \eclass(S)$.  Since each $u_j \in \eclass(S)$ as well, Corollary \ref{cor:nrgbnds} and Equation (\ref{EQN:PROP_NRGBNDS_INEQ2}) give
\begin{align*}
\norm[S]{u_j}^2 \leq \norm[S]{u}^2 + 2 \norm[L^1(S \wedge \omega)]{u_j - u}. 
\end{align*}
The inequality therefore follows from $u_j\to u$ in $L^1(S\wedge\omega)$.

For the reverse inequality $\lim\norm[S]{u_j} \geq \norm[S]{u}$, let $(v_j)\subset C^\infty(X)\cap \psh(\omega)$ be another sequence decreasing to $u$.  By Hartog's Lemma we can replace $v_j$ with $v_j + 2^{-j}$ and pass to subsequences to arrange that $u_j \leq v_j$ for all $j$.  Hence by Corollary \ref{cor:nrgbnds} again
$$
\norm[S]{v_j}^2 \leq \norm[S]{u_j}^2 + 2\norm[L^1(\omega \wedge S)]{v_j-u_j},
$$
for each $j \in \mathbb{N}$.  Letting $j\to\infty$, applying Corollary \ref{cor:decapprox2} on the right side, and using $u_j,v_j\to u$ in $L^1(S\wedge\omega)$ on the left, we obtain $\norm[S]{u} \leq \lim\norm[S]{u_j}$, as desired.

Continuity of $\nrg{\cdot,\cdot}_S$ under decreasing limits in $\eclass(S)\cap \psh(\omega)$ follows from the Cauchy-Schwarz inequality.
\end{proof}

The following is a (less general) version of \cite[Proposition 1.8]{DDG11}, and we refer the reader to the proof given there. 

\begin{prop}
\label{prop:strong} Given $u\in \eclass(S)\cap \psh(\omega)$ and $u_j = \max\{u,-j\}$, let $T = \omega + dd^c u$ and $T_j := \omega + dd^c u_j$ be the associated positive closed $(1,1)$ currents.  Then we have for all $k>j$ that
\begin{equation}
\label{eqn:restriction}
\oone_{\{u>-j\}} \,T_j\wedge S = \oone_{\{u>-j\}}\,T_k\wedge S.
\end{equation}
Consequently, the measures $\oone_{\{u>-j\}} T_j\wedge S$ are increasing with $j$ and converge strongly as Borel measures to $T\wedge S$.
\end{prop}

\begin{cor}
\label{cor:massonpps}
Suppose that $T = \omega + dd^c u$ for some $u\in \eclass(S)\cap\psh(\omega)$ and that $P\subset X$ is a complete pluripolar set not charged by $S$.  Then $T\wedge S$ does not charge $P$ either.
\end{cor}

\begin{proof}
We suppose as usual that $u\leq 0$.  By hypothesis $P=\{v=-\infty\}$ for some negative function $v\in\psh(\omega)$.  Since $\omega\wedge S(P) = 0$, we may replace $v$ with $\chi\circ v$, where $\chi:[-\infty,0]\to[-\infty,0]$ is a (very slowly) increasing convex bijection, in order to suppose that $v \in L^1(S)$.  

Now given a positive integer $j$, let $u_j = \max\{u,-j\}$.  Then $T_j := \omega + dd^c u_j \geq 0$ and
\begin{eqnarray*}
\int -v\, T_j\wedge S & = & \int -v\,\omega\wedge S + \int -v\, dd^c(u_j+j) \wedge S = \int -v\,\omega\wedge S + \int (u_j+j)\,(- dd^c v)\wedge S \\
& \leq & \int -v\,\omega\wedge S + \int (u_j+j)\,\omega\wedge S
< \infty.
\end{eqnarray*}
Corollary \ref{cor:decapprox2} justifies the integration by parts in the second equality and $v \in \psh(\omega)$ justifies the final inequality.  In any case, since $v \equiv -\infty$ on $P$  we infer from finiteness of this integral  that $T_j\wedge S(P) = 0$ for every $j$.  Hence $T\wedge S(P) = \lim_{j\to\infty} T_j\wedge S(P) = 0$ by Proposition \ref{prop:strong}.
\end{proof}

\begin{cor}
\label{cor:lessisgood}
If $u,v\in \psh(\omega)$ satisfy $u\geq v$, then $v\in \eclass(S)$ implies that $u\in \eclass(S)$, too.
\end{cor}

\begin{proof}
Let $u_j,v_j\in C^\infty(X)\cap\psh(\omega)$ be sequences that decrease to $u$ and $v$, respectively.  Since $u\geq v$, we may assume $u_j \geq v_j$ for all $j$.  
Using the hypothesis that $u, v\in \psh(\omega)$ and Proposition  \ref{prop:nrgbnds} we obtain
$$
\norm[S]{u}^2 = \lim_{j\to\infty} \norm[S]{u_j}^2 \leq \lim_{j\to\infty}\left( \norm[S]{v_j}^2 + 2 \norm[L^1(\omega \wedge S)]{u_j-v_j}\right) = \norm[S]{v}^2 + 2 \norm[L^1(\omega \wedge S)]{u-v} < \infty.
$$
Note that $\norm[L^1(S)]{u-v} < \infty$ since $v \in \eclass(S)$ implies $v \in L^1(S)$ and therefore $u-v \in L^1(S)$.
\end{proof}

The following is a particularly useful application of Corollary \ref{cor:lessisgood}

\begin{cor}
\label{cor:logsings}
Suppose that $u\in \psh(\omega)$ is continuous off a finite set $I\subset X$, that local potentials for $S$ are finite at each point in $I$, and that $u(x)\geq -A+B\log\dist(x,I)$ for some constants $A,B\geq 0$  Then $u\in \eclass(S)$.
\end{cor}

\begin{proof}
The bound on $u$ means we can bound $u$ from below by a an $\omega$-psh function $\tilde u$ that is smooth off $I$ and, near each $p\in I$, equal to $\tilde u(z) + B\log\norm{z}$ for some smooth function $\tilde u$ and local coordinate $z$ centered at $p$.  It suffices to show therefore that $\tilde u$ has finite $S$-energy.  This is accomplished in \cite[Theorem 3.6]{BeDi05b}. 
\end{proof}

By linearity, the pairing $\nrg{\cdot,\cdot}_S$ extends to \emph{differences} $u_1-u_2$ of functions $u_1,u_2\in\eclass(S)$.    However, there is no reason to expect that $\eclass(S)$ is any sense complete with respect to $\norm[S]{\cdot}$.  Hence it is necessary for us to allow a vector space larger than the one spanned by $\eclass(S)$.  

\begin{defn}
\label{defn:DPSH_and_ADMIT_WEDGE}
Let $\dpsh(X)$ denote the set of all differences $u = u_1-u_2$, where $u_1,u_2\in\qpsh(X)$.  We will say that $dd^c u$ {\em admits a wedge product with $S$} if the decomposition $u=u_1-u_2$ can be chosen so that $u_1, u_2\in L^1(S)$ separately.  
\end{defn}

This guarantees that $dd^c u\wedge S$ is a signed Borel measure with finite total mass. 
 We introduce the following rather ad hoc terminology.

\begin{defn}
\label{defn:WEAKLY_FINITE_ENERGY}
We say that $u\in\dpsh(X)$ has \emph{weakly finite $S$-energy}, writing $u\in \weclass(S)$, if $dd^c u$ admits a wedge product with $S$, and there is a constant $C>0$ such that 
$$
-\int v\,dd^c u \wedge S \leq C\norm[S]{v}
$$
for any $v\in C^\infty(X)$.  
\end{defn}

\noindent If in the decomposition $u=u_1-u_2$, both $u_1$ and $u_2$ have finite $S$-energy, then $u$ has weakly finite $S$-energy (with e.g. $C = \norm[S]{u_1} + \norm[S]{u_2}$).  However, the point to Definition \ref{defn:WEAKLY_FINITE_ENERGY} is that $u$ might have weakly finite $S$-energy even when both $\norm[S]{u_1}$ and $\norm[S]{u_2}$ are infinite. In any case, if $u$ has weakly finite $S$-energy, then we let $\dnorm[S]{u}$ denote the smallest possible value of $C$ in this definition.  

In several places below, we will convert information about the Lelong numbers of a positive closed current to information about how much mass its wedge product with another current assigns to points.  We rely on the bounds in the following result to do this.

\begin{prop}
\label{prop:lelongtoptmass}
There is an absolute constant $C>0$ such that the following hold for any $u\in\dpsh(X)\cap L^1(S)$ and $p\in X$.
\begin{enumerate}
 \item If $|u|\leq M$ in a neighborhood of $p$, then $(dd^c u\wedge S)(p) \leq CM\nu(S,p)$.
 \item If $u\in\weclass(S)$, then $(dd^c u\wedge S)(p) \leq C\dnorm[S]{u}\nu(S,p)$.
\end{enumerate}
\end{prop}

\begin{proof}
Fix a local coordinate $x:U\to B_1(0)$ on $X$ centered at $p$.  Let $\omega$ be the Euclidean K\"ahler form on $B_1(0)$ and $\chi:B_1(0)\to [0,1]$ be a smooth, compactly supported function equal to $1$ near $0$.  For any $r<1$, let $\chi_r(x) = \chi(x/r)$.  Then
$$
(dd^c u\wedge S)(p) = \lim_{r\to 0} \int \chi_r\,dd^c u\wedge S.
$$
If $|u|\leq M$ near $p$, integration by parts gives
$$
\left|\int \chi_r\,dd^c u\wedge S\right| = \left|\int u\,dd^c\chi_r\wedge S\right|
\leq \frac{M\norm[\mathcal{C}^2]{\chi}}{r^2}\int \omega\wedge S \to CM\nu(S,p)
$$
for some constant $C$.  Biholomorphic invariance of Lelong numbers implies that $C$ does not depend on the choice of local coordinate.  This proves (1).

For (2) we have by definition of $\dnorm[S]{u}$ and $\norm[S]{\chi_r}$ that
$$
\left|\int \chi_r\,dd^c u\wedge S\right| \leq \dnorm[S]{u} \left|\int \chi_r\,dd^c\chi_r\wedge S\right| 
\leq
\frac{C\dnorm[S]{u}}{r^2}\left|\int \omega\wedge S\right| \to C\dnorm[S]{u}\nu(S,p)
$$
as before.
\end{proof}

\subsection{Wedge products of toric currents}
\label{ss:toricwedge}

Let us next extend the above discussion of wedge products to the non-compact
surface $\rztO$.  Rather than aim for maximum generality, we seek only to
establish a framework adequate for the purposes of this paper.  

\begin{defn}\label{DEFN:WEDGE_PRODUCT_TORIC_CURRENTS}
We will say
that toric currents $T,S\in\pcc^+(\rzt)$ \emph{admit a wedge product} if
$\ch{T},\ch{S}$ are K\"ahler classes, and in each toric surface $X$, we have
$T\wedge_X S := T_X\wedge S_X$ is defined (in the Bedford-Taylor sense): i.e.
given $X$, there are K\"ahler forms $\omega_{T,X},\omega_{S,X}$ on $X$ such
that $T_X = \omega_{T,X} + dd^c u_X$, $S_X = \omega_{S,X} + dd^c v_X$, where
$u_X\in L^1(S_X)$ (equivalently $v_X\in L^1(T_X)$).  
\end{defn}

\begin{prop}
\label{prop:shrinks}
If $T,S\in \pcc^+(\rzt)$ admit a wedge product and $Y\succ X$ are toric surfaces, then
\begin{equation}
\label{eqn:wedgeshrinks}
T\wedge_X S = \pi_{YX*}(T\wedge_Y S) + \nu.
\end{equation}
where $\nu$ is a positive measure with $\supp\nu = \pi_{YX}(\exc(\pi_{YX}))$.
\end{prop}

\begin{proof}
It suffices to consider the case where $\pi_{YX}$ is the blowup of $p_\sigma\in X$.  If $E = \pi_{YX}^{-1}(p_\sigma)$, then $\pi_{YX}:Y-E \to X-{p_\sigma}$ is an isomorphism.  Hence
$\pi_{YX*}(T\wedge_Y S) = T\wedge_X S$ on $X-\{p_\sigma\}$.  On the other hand, the total mass of
$T\wedge_Y S$ is given by
$$
\isect{T_Y}{S_Y} = \isect{\pi_{YX}^* T_X + c E)}{(\pi_{YX}^* S_X + c' E} = \isect{T_X}{S_X} - cc'
< \isect{T_X}{S_X}
$$
by Proposition \ref{prop:changes}. Hence $T\wedge_X S$ gives more mass to $p_\sigma$ than $\pi_{YX*} (T\wedge_Y S)$ does.
\end{proof}

Equation \ref{eqn:wedgeshrinks} tells us that when $T$ and $S$ admit a wedge product the mass of the measure $T\wedge_X S$ decreases as the surface $X$ increases, converging to $\isect{T}{S}$ as $X\to\rzt$.  On the other hand, for surfaces $Y\succ X$, we have via the inclusions $X^\circ\subset Y^\circ \subset \rzt^\circ$ that $(T\wedge_Y S)|_{X^\circ} = (T\wedge_X S)|_{X^\circ}$ as measures on $\rzt^\circ$.  Hence 
$$
(T\wedge_X S)|_{X^\circ} \leq (T\wedge_Y S)|_{Y^\circ} \leq T\wedge_Y S.
$$ 
That is, the restricted measures increase with the surface, and their masses
are bounded \emph{above} by $\isect{T}{S}$.  The restrictions therefore
converge to a positive Borel measure on $\rzt^\circ$ that we denote by $T\wedge
S$.  Since $(T\wedge S)(\rzt^\circ)\leq \isect{T}{S}$, we say that $T\wedge S$
\emph{has full mass} (on $\rzt^\circ$) when equality holds.  This is equivalent
to saying that the mass of $T\wedge_X S$ on $\torus$-invariant points of $X$
decreases to $0$ as $X$ increases to $\rzt^\circ$.  In any case, if $T\wedge S$
does not charge curves in $\rzt^\circ$, then
$$
T\wedge S = (T\wedge S)|_{\torus} = (T\wedge_X S)|_\torus = (T\wedge_X S)|_{X^\circ}
$$
for every toric surface $X$.

\begin{thm} 
\label{thm:wedgewhomogeneous}
Let $\bar T, S\in\pcc^+(\rzt)$ be positive internal toric currents such that $\bar T$ is homogeneous and $\ch{\bar T},\ch{S}\in\hoo(\rzt)$ are K\"ahler.  
\begin{enumerate}
 \item $\bar T$ and $S$ admit a wedge product.
 \item If $S$ does not charge a given curve $C\subset\rztO$, then neither does $\bar T\wedge S$.  Hence $\bar T\wedge S = (\bar T\wedge S)|_{\torus}$.
 \item Let $\bar S$ be the homogenization of $S$.  Then $\bar T\wedge \bar S$ has full mass on $\rztO$ and is equal to $\isect{\bar T}{\bar S}$ times Haar measure on $\torus_\R$.
 \item If $S$ has a nearly homogeneous support function, then $\bar T\wedge S$ has full mass on $\rztO$.
\end{enumerate}
\end{thm}

\begin{proof}
Fix a particular toric surface $X$ and write $\bar T_X = \omega + dd^c u$, where $\omega$ is a K\"ahler form and $u\in\psh(\omega)$. By Proposition \ref{prop:bartonpole}, $u$ is continuous on $X^\circ$.  Hence $u\in L^1(S_X)$ by Proposition \ref{prop:boundedish}. This proves that $\bar T\wedge_X S$ is defined in the Bedford-Taylor sense.  Taking $u_j = \max\{u,-j\}$, we have that $u_j \in \eclass(S_X)$ agrees with $u$ on any given compact subset of $X^\circ$ when $j$ is large enough.  So given a curve $C\subset X$ not charged by $S_X$, Corollary \ref{cor:massonpps} gives that $(\omega+dd^c u_j)\wedge S_X$ does not charge $C$ for any $j\geq 0$.  
The interiors of the compact sets $\{u_j=u\}$ exhaust $X^\circ$ as $j\to\infty$, so $\bar T\wedge_X S$ does not charge $C\cap X^\circ$.  In particular, since $S_X$ is internal, $\bar T\wedge S_X$ does not charge $X^\circ\cap C$ for any pole $C\subset X$.  Letting $X$ increase to $\rzt$ gives the first two conclusions.  For Conclusion (3) we first prove

\begin{lem}
$(\bar T \wedge \bar S)|_\torus = (\bar T \wedge \bar S)|_{\torus_\R}$ is a multiple of Haar measure
on $\torus_\R$.
\end{lem}

\begin{proof}
Let $\sfn_{\bar T},\sfn_{\bar S}$ be support functions for $\bar T$ and $\bar S$.  
Let $(\tau_j)_{j\geq 0}$ be a sequence of rational rays with $\bigcup \tau_j$ dense in $N_\R$.  For $n>0$ large, let $\sfn_{\bar T_n}:N_\R\to \R$ be the support function defined by setting $\sfn_{\bar T_n} = \sfn_{\bar T}$ on $\tau_1\cup\dots\cup\tau_n$ and then extending linearly to the $n$ disjoint sectors in the complement.  Then $\sfn_{\bar T_n}$ decreases to $\sfn_{\bar T}$ uniformly on compact subsets of $N_\R$, and the associated positive internal currents $\bar T_n$ converge weakly to $\bar T$.  We similarly approximate $\sfn_{\bar S}$ by support functions $\sfn_{\bar S_n}$ associated to another dense sequence of rays entirely disjoint from $(\tau_j)_{j\geq 0}$. 

Then for any $n$ and each point $p \in \torus\setminus\torus_\R$, we have that either $\sfn_{\bar T_n}\circ\Log$ or $\sfn_{\bar S_n}\circ\Log$ is pluriharmonic on a neighborhood of $p$.  It follows that 
$\bar T_n\wedge \bar S_n|_{\torus}$ is supported on $\torus_\R$.  By $\torus_\R$-invariance, $\bar T_n\wedge \bar S_n|_{\torus}$ must be a non-negative multiple of Haar measure on $\torus_\R$.
On the other hand, continuity of the Monge-Amp\`ere operator with respect to decreasing sequences
of locally bounded psh functions (see e.g. \cite[Theorem 3.18]{GZ_book}) gives us that
\begin{align*}
(\bar T_n \wedge \bar S_n)|_\torus \rightarrow (\bar T \wedge \bar S)_\torus.
\end{align*}
\end{proof}

It remains to show that $(\bar T\wedge \bar S)|_{\torus_\R}$ has (full) mass equal to $\isect{\bar T}{\bar S}$.  To do this we employ the homogeneous approximation $\bar T(X) \in \pcc^+(X)$ (see \cite[Theorem 6.7]{DiRo24}) of $\bar T$ on some toric surface $X$.  That is, $\sfn_{\bar T(X)}$ is linear on every sector in $\Sigma(X)$ and agrees with $\sfn_{\bar T}$ on each ray in $\Sigma_1(X)$.  Hence $\sfn_{\bar T(X)}\circ\Log$ is pluriharmonic near the torus invariant points of $X$, so  $\supp\bar T(X)\wedge_X \bar S$ avoids $\torus$-invariant points.  By Conclusion (2), $\bar T(X)\wedge \bar S$ does not charge curves.  So it follows from the lemma that $\bar T(X)\wedge_X \bar S = (\bar T(X)\wedge \bar S)|_\torus$ is exactly equal to $\isect{\bar T(X)}{\bar S}_X$ times Haar measure on $\torus_\R$.  Since $\bar T_X$ is cohomologous in $X$ to $\bar T(X)$, we further have that
$
\isect{\bar T(X)}{\bar S}_X = \isect{\bar T}{\bar S}_X 
$
decreases to 
$
\isect{\bar T}{\bar S}
$
as $X$ increases to $\rzt$, hence $\bar T(X)\wedge_X \bar S$ decreases to $\isect{\bar T}{\bar S}$ times Haar measure on $\torus_\R$.  

On the other hand, as $X$ increases to $\rzt$, the rays in $\Sigma_1(X)$ become dense in $N_\R$.  Hence as in the proof of the lemma, we have $\bar T(X)\wedge_X \bar S \to \bar T\wedge \bar S$ on any relatively compact open subset of $\torus$.  The third conclusion follows immediately.

For the final conclusion of Theorem \ref{thm:wedgewhomogeneous} we observe that
$$
\sum_{\sigma\in\Sigma_2(X)} \bar T\wedge S(p_\sigma) =
\sum_{\sigma\in\Sigma_2(X)} \bar T\wedge \bar S(p_\sigma) + \sum_{\sigma\in\Sigma_2(X)} \bar T\wedge dd^c\rpot_S(p_\sigma),
$$
where $\rpot_S$ is the $\bar{S}$ potential for $S$.
Since $\bar T$ is $\torus_\R$-invariant, as is each $p_\sigma\in X\setminus X^\circ$, we may suppose in the second sum on the right that $S = S_{ave}$ is also $\torus_\R$-invariant.  So when the support function $\sfn_S$ for $S$ is nearly homogeneous, $\rpot_S = (\sfn_S - \sfn_{\bar S})\circ\Log$ is bounded.  From the first conclusion of Proposition \ref{prop:lelongtoptmass} we infer that
$$
\sum_{\sigma\in\Sigma_2(X)} \bar T\wedge S(p_\sigma) \leq
\sum_{\sigma\in\Sigma_2(X)} \bar T\wedge \bar S(p_\sigma) + C\sum_{\sigma\in\Sigma_2(X)}\nu(\bar T, p_\sigma).
$$
But as $X$ increases to $\rzt$, the first term on the right tends to $0$
because $\bar T\wedge \bar S$ has full mass on $\torus_\R \subset \rztO$, and the second term
tends to $0$ by Proposition \ref{prop:smalllelong}.  Hence the left side tends
to $0$, too, which is equivalent to the statement that $\bar T\wedge S$ has
full mass on $\rztO$.  \end{proof}

Let $C^\infty(\rzt)$ denote the union of $C^\infty(X)$ over all toric surfaces $X$ modulo the identification $u\circ \pi_{XY} \sim u$ for any $u\in C^\infty(Y)$ and any surface $X\succ Y$.  Then we have
$$
C_0^\infty(\rztO) \subset C^\infty(\rzt) \subset C^\infty(\rztO).
$$
That is on the one hand, if $u\in C_0^\infty(\rztO)$, then $\supp u\subset X^\circ$ for some surface $X$.  Hence $u\in C^\infty(X)\subset C^\infty(\rzt)$.  And on the other hand, $v\in C^\infty(\rzt)$ implies that $v|_{X^\circ}$ is well-defined on all sufficiently dominant $X$.  Since the open sets $X^\circ$ exhaust $\rztO$ and the restrictions agree with each other on overlaps, it follows that $v$ is well-defined and smooth on all of $\rztO$.

Any toric current $S\in\pcc^+(\rzt)$ defines a pairing on $C^\infty(\rzt)$ modeled on the one in \eqref{EQN:DEF_SMOOTH_PAIRING}.  Given $u,v\in C^\infty(\rzt)$, we choose a toric surface $X$ in which both $u$ and $v$ are well-defined, and set 
$$
\nrg{u,v}_S := \nrg{u,v}_{S_X} = \int_X du\wedge d^c v\wedge S_X.
$$
Since $\pi_{XY*} S_X = S_Y$ when $X\succ Y$, this definition is independent of the choice of $X$.  If at least one of the two functions belongs to $\in C^\infty_0(\rztO)$, then $du\wedge d^cv$ is a test form on $\rztO$, giving us the equivalent formulation
\begin{equation}
\label{eqn:toricenergy}
\nrg{u,v}_S = \int_{\rztO} du\wedge d^c v\wedge S.
\end{equation}
In any case, we let $\norm[S]{\cdot}$ denote the associated seminorm.

\begin{defn}
\label{DEF:WEAKLY_FINITE_ENERGY_ON_RZTO}
Let $T_1,T_2,S\in\pcc^+(\rztO)$ be currents such that $T_1$ is cohomologous to $T_2$ and both admit wedge products with $S$ (in particular, $\ch{T_j}$ and $\ch{S}$ are K\"ahler classes).  We say that a potential $\rpot\in L^1_{loc}(\rztO)$ for $T_1-T_2$ has weakly finite $S$-energy if there is a constant $C$ such that
for all toric surfaces $X$ and all $u\in C^\infty(X)$, we have
$$
\int_X u\,dd^c\rpot\wedge_X S \leq C\norm[S]{u}.
$$
\end{defn}

\noindent 
We let $\dnorm[S]{\rpot}$ denote the minimum possible value of $C$ in this
definition.  Recall that if $X$ is a toric surface, then $\rpot$ (or strictly
speaking, the restriction $\rpot|_{X^\circ}$) is a also potential for
$T_{1,X}-T_{2,X}$.

\begin{prop}
\label{prop:dnormincreases}
Let $\rpot$ be the potential in Definition \ref{DEF:WEAKLY_FINITE_ENERGY_ON_RZTO} and $X\succ Y$ be toric surfaces.  Then $\dnorm[S_X]{\rpot} \geq \dnorm[S_Y]{\rpot}$.
\end{prop}

\begin{proof}
Suppose that $u\in C^\infty(Y)\subset C^\infty(\rzt)$ and that $X\succ Y$.  Let $\pi = \pi_{XY}$ denote the transition.  Then 
\begin{eqnarray*}
\int_X (u\circ \pi)\,dd^c \rpot \wedge S_X
& = &
\int_X dd^c(u\circ \pi)\wedge (\rpot S_X) 
=
\int_Y dd^c u \wedge \pi_*(\rpot S_X) \\
& = &
\int dd^c u \wedge (\rpot S_Y)
=
\int_Y u\, dd^c\rpot \wedge S_Y.
\end{eqnarray*}
The first and last equalities hold by definition since $u$ is smooth.  Note that $dd^c(u\circ \pi) \wedge S_X$ and $dd^c u\wedge S_Y$ are signed measures dominated above and below by constant multiples of the trace measures of $S_Y$ and $S_X$.  Hence the second inequality follows from the definition of pushforward and integrability of $\rpot$ with respect to the trace measures of $S_X$ and $S_Y$.  Thus
$$
\left| \int_Y u\, dd^c\rpot \wedge S_Y\right| 
= 
\left| \int_X (u\circ\pi)\, dd^c\rpot \wedge S_X\right| 
\leq 
\dnorm[S_X]{\rpot}\norm[S_X]{u\circ\pi} = \dnorm[S_X]{\rpot}\norm[S_Y]{u}.
$$
So $\dnorm[S_X]{\rpot} \geq \dnorm[S_Y]{\rpot}$.
\end{proof}

\begin{cor} The potential $\rpot$ in Definition \ref{DEF:WEAKLY_FINITE_ENERGY_ON_RZTO} has weakly finite $S$-energy if and only if $\dnorm[S_X]{\rpot}$ is finite and uniformly bounded as $X$ ranges over all toric surfaces.  In this case,
$$
\dnorm[S]{\rpot} = \sup_X \dnorm[S_X]{\rpot}.
$$
\end{cor}

\begin{thm}
\label{thm:fullmass}
Suppose that $S,T\in\pcc^+(\rzt)$ are internal and admit a wedge product. If the support function $\sfn_S$ for $S$ is nearly homogeneous and the potential $\rpot_T$ for $T-\bar T$ has weakly finite $S$-energy, then $T\wedge S$ has full mass on $\rztO$.
\end{thm}

\begin{proof}
By hypothesis and Theorem \ref{thm:wedgewhomogeneous}, $\bar T\wedge S$ has full mass on $\rztO$.  Moreover, for each toric surface $X$, we have 
from Propositions \ref{prop:lelongtoptmass} and \ref{prop:dnormincreases} that
$$
\sum_{p\in X\setminus X^\circ} |(dd^c \rpot_T\wedge S)(p)| \leq C\dnorm[S_X]{\rpot_T}\sum_{p\in X\setminus X^\circ} \nu(S_X,p) 
\leq C\dnorm[S]{\rpot_T} \sum_{p\in X\setminus X^\circ}\nu(S_X,p).
$$
By Proposition \ref{prop:smalllelong}, the last sum decreases to $0$ as the surface $X$ increases, so we conclude that $T\wedge S = \bar T\wedge S + dd^c \rpot_T\wedge S$ has full mass on $\rztO$.
\end{proof}

%% file: measure4.tex
Recall our standing assumption that $f:\rztO\tto\rztO$ is a toric map satisfying the hypotheses of Theorem \ref{thm:mainthm}.  Our goal in this section is to apply the machinery from \S \ref{sec:products} to define and investigate the Bedford-Taylor wedge product of the equilibrium currents $T^*$ and $T_*$ associated to $f$ in \S\ref{subsec:eq}.  More precisely, we prove the following result, which amounts to all of Theorem \ref{thm:mainthm} except Conclusions (2) and (3) which are proven in \S\ref{sec:mixing}.

\begin{thm}
\label{thm:eqmeasureexists}
The internal currents $T^*$ and $T_*$ admit a wedge product in the sense of Definition \ref{DEFN:WEDGE_PRODUCT_TORIC_CURRENTS}.  The resulting Borel measure $T^*\wedge T_*$ on $\rztO$ has (full) mass equal to $1$ and does not charge curves.  
\end{thm}

\begin{defn} We call $\mu := T^*\wedge T_*$ the \emph{equilibrium measure} of $f$.
\end{defn}

\noindent Since $\mu$ does not charge curves in $\rztO$, one can alternatively define $\mu$ to be the Borel probability measure obtained by fixing some/any toric surface $X$ and restricting the product $\mu = T^*\wedge_X T_*$ to $\torus$.

The main ingredient in the proof of Theorem \ref{thm:eqmeasureexists} is Theorem \ref{thm:weakenergy1}, which we take for granted momentarily.  

\begin{proof}[Proof of Theorem \ref{thm:eqmeasureexists}] 
Let $X$ be a toric surface.  By Proposition \ref{prop:iskahler}, there is a K\"ahler form $\omega_X$ and $u\in \psh(\omega_X)$ such that $T^*_X = \omega_X + dd^c u$.  Likewise, $\bar T^*_X = \omega_X + dd^c v$ for some $v\in \psh(\omega_X)$.  Hence $\rpot_{T^*} = u-v$ is a potential for $T^*-\bar T^*$.  Proposition \ref{prop:bartonpole} tells us that $v$ is continuous on $X^\circ$, and Theorem \ref{THM:CONTINUITY} tells us that $\rpot_{T^*}$ is continuous on $X^\circ\setminus \ind(f^\infty)$.  Hence $u = \rpot - v$ is continuous off the finite subset consisting of points in $X$ that are $\torus$-invariant and/or indeterminate for some iterate of $f$.  Proposition \ref{prop:boundedish} therefore implies $u\in L^1(T_{*,X})$.  Hence the Bedford-Taylor wedge product $T^* \wedge_X T_*$ is well-defined in $X$.  Since $X$ is arbitrary, we conclude that $T^*$ and $T_*$ admit a wedge product on $\rztO$.

By Theorem \ref{thm:eqnearhom} any support function $\sfn_{T_*}$ for $T_*$ is nearly homogeneous.  By Theorem \ref{thm:weakenergy1}, the potential $\rpot_{T^*}$ has weakly finite $T_*$-energy.  So Theorem \ref{thm:fullmass} tells us that $T^*\wedge T_*$ has full mass on $\rztO$.  Corollary \ref{cor:nidinfinity} tells us that $T_*$ does not charge curves, so Theorem \ref{thm:wedgewhomogeneous} tells us that $\bar T^*\wedge T_*$ does not charge curves either.   It will suffice, therefore, to show that $dd^c\rpot_{T^*}\wedge T_* = dd^c(\rpot_{T^*} T_*)$ does not charge any curve $C\subset \rztO$.

Let $n$ be the largest integer such that $f^{-n}(\ind(f))\cap C$ is non-empty and $X$ be a toric surface sufficiently dominant that $C\subset X^\circ$.  Then we may apply Corollary \ref{cor:massonpps} as we did in the proof of Conclusion (2) in Theorem \ref{thm:wedgewhomogeneous} to see that $dd^c\rpot_{T^*}\wedge_X T_*$ does not charge $C\setminus \ind(f^n)$.  On the other hand, Corollary \ref{cor:cvgceonpoles} gives that local potentials for $T_*$ are finite at points $p \in \ind(f^n)\cap X^\circ$, so Conclusion (2) of Proposition \ref{prop:lelongtoptmass} shows that $dd^c\rpot_{T^*}\wedge T_*$ assigns no mass to $\ind(f^n)$ either.
\end{proof}

The rest of this section is devoted to proving Theorem \ref{thm:weakenergy1}.  We caution from the outset that if $X$ is a given toric surface, both currents $T^*_X,T_{*,X}\in\pcc^+(X)$ have positive Lelong numbers at each $\torus$-invariant point of $X$.  Hence $\varphi_{T^*}$ is \emph{not} a difference of qpsh functions each with finite $T_{*,X}$-energy.  This forces us to accept the weaker conclusion that $\varphi_{T^*}$ has only \emph{weakly} finite $T_*$-energy and leads to arguments that are longer and more delicate than one might hope.

For any integer $n\geq 0$, we set $T_n := \ddeg^{-n} f^{n*}\bar T^* \in
\pcc^+(\rzt)$ which is internal by Proposition~\ref{prop:intext} and
cohomologous to $T_0 = \bar T^*$.  We let $\sfn_n := \sfn_{T_n}$ denote a
support function for $T_n$ and $\rpot_n = \rpot_{T_n}$ denote a potential for
$T_n - T_0$.  Fixing a toric surface $X$ and a K\"ahler form $\omega_X$
cohomologous on $X$ to $T^*_X$, we also have $T_{n,X} = \omega_X + dd^c
u_{n,X}$ for some $u_{n,X}\in\psh(\omega_X)$.  We normalize so that $\rpot_n =
u_{n,X}-u_{0,X}$.  While $u_{n,X}$ depends on the surface, $\rpot_n$ does not.
In any case, we will suppress the subscript $X$ when the surface is understood.

Theorem \ref{thm:pullback} Part~(1) tells us that $\rpot_n$ is not only continuous on $X^\circ\setminus\ind(f^n)$ but also bounded in a neighborhood of each point in $X\setminus X^\circ$.  Since $u_0$ is continuous on $X^\circ$ by Proposition~\ref{prop:bartonpole} we therefore also have that $u_n=\rpot_n + u_0$ is continuous on $X^\circ\setminus\ind(f^n)$.  Hence the Bedford-Taylor product $T_n\wedge_X T_*$ is well-defined for all $n\geq 0$.  As we noted in \S\ref{ss:toricwedge}, the measures $T_n\wedge_X T_*$, and even their masses, vary with $X$.  Our next result shows, however, that for different $n,m\geq 0$ and sufficiently dominant toric surfaces $X$, the measures $T_n\wedge_X T_*$ and $T_m\wedge_X T_*$ vary with $X$ in exactly the same way.

\begin{thm}
\label{thm:cancelation} For any toric surface $X$ such that $\ind(f^n)\subset X^\circ$, the signed measure
$
\mu_n := (T_n - T_0) \wedge_X T_*
$
has no mass on any curve $C\subset X$.  In particular, $\mu_n$ has full mass on $\torus\subset\rztO$ and is independent of (sufficiently dominant) $X$.
\end{thm}

\begin{proof}
Let $C\subset X$ be a curve.  Fix a neighborhood $U$ of the $\torus$-invariant points $X\setminus X^\circ$, chosen so that $\overline{U}\cap \ind(f^n) = \emptyset$.  \
Proposition \ref{prop:smalllelong} tells us that since $\ch{T^*_n} = \ch{T^*}$ is K\"ahler, the potentials $u_n$ have non-trivial logarithmic singularities near each $\torus$-invariant point of $X$.  For $j\geq 0$ large, we let $u_{n,j}\in\psh(\omega)$ be equal to $u_n$ off $U$ and to $\max\{u,-j\}$ on $U$.  Then $u_{n,j}$ is continuous except for logarithmic singularities at points of $\ind(f^n)$.  Corollary \ref{cor:logsings} then tells us that $\norm[T_*]{u_{n,j}}$ is finite.  Since $T_*$ does not charge curves, Corollary \ref{cor:massonpps} further implies that $(\omega+dd^c u_{n,j})\wedge_X T_*$ does not charge curves either.  The open set $\{u_{n,j} \neq u_n\}$ decreases to $X\setminus X^\circ$ as $j\to\infty$, so we conclude that $\mu_n(C\cap X^\circ) = 0$.

It remains only to show that $\mu_n(p_\sigma) = 0$ for each $\torus$-invariant $p_\sigma \in X$.  Since $\rpot_n$ is bounded away from $\ind(f_n)$, our choice of surface $X$ guarantees that $|\rpot_n|\leq M$ on a neighborhood of $X\setminus X^\circ$.  The constant $M$ does not change if we replace $X$ by $Y\succ X$.
So given $\epsilon>0$, choose $Y\succ X$ such that $\sum_{\sigma\in\Sigma_2(Y)} \nu(T_{*,Y},p_\sigma) < \epsilon$.  Observe that
\begin{eqnarray*}
\sum_{\sigma\in\Sigma_2(X)} |\mu_n(p_\sigma)| 
& = & \sum_{\sigma\in\Sigma_2(X)} |(dd^c\rpot_n\wedge_X T_*)(p_\sigma)| 
= \sum_{\sigma\in\Sigma_2(X)} |(dd^c\rpot_n\wedge_Y T_*)(\pi_{YX}^{-1}(p_\sigma))| \\
& = & \sum_{\sigma\in\Sigma_2(Y)} |(dd^c\rpot_n\wedge_Y T_*)(p_\sigma)| \leq CM\sum_{\sigma\in\Sigma_2(Y)} \nu(T_*,p_\sigma) \leq CM\epsilon.
\end{eqnarray*}
The second equality comes from Lemma \ref{LEM:CURRENT_PROJECTION_FORMULA} and the identity $\pi_{YX*} (dd^c\rpot_n\wedge_Y T_*) = dd^c\rpot_n\wedge_X T_*$.  The third equality follows from applying the first paragraph of this proof to curves in $Y$ instead of curves in $X$.  The final inequality is obtained from item (1) in Proposition~\ref{prop:lelongtoptmass}.  The constant $C$ is independent of $Y$, so by letting $\epsilon$ decrease to $0$, we find that $\sum_{\sigma\in\Sigma_2(X)} |\mu_n(p_\sigma)| = 0$ as hoped.
\end{proof}

Now fix two non-negative integers $m,n\geq 0$ and let $N = \max\{m,n\}$.  Continuity of $\rpot_m$ on $\rztO \setminus \ind(f^N)$ implies that $\rpot_m$ is measurable (in any toric surface $X$) with respect to $\mu_n$.  Since $\mu_n$ does not charge curves the product $\rpot_m\,\mu_n$ has no mass outside $\torus$ and is therefore independent of $X$.  That is, we can regard $\rpot_n \mu_m$ as a measure on $\torus$, on $\rztO,$ or on a given toric surface $X$ without distinction.  The value $\int \rpot_n \mu_m$ is the same in any case.

\begin{prop}
\label{prop:integrable}
$\rpot_m$ is integrable with respect to $\mu_n$.
\end{prop}

\begin{proof} Suppose that $X$ is a toric surface such that $\ind(f^N)\subset X^\circ$.  Let $U\subset X^\circ$ be a union of coordinate balls, centered at the points of $\ind(f^N)$. Let $\chi:X\to[0,1]$ be a smooth cutoff function supported in $U$ and equal to $1$ on a neighborhood of $\ind(f^N)$.  By construction and Theorem \ref{thm:pullback}, $(1-\chi)\rpot_m$ is bounded, so it is $\mu_n$-integrable simply because $\mu_n$ has finite total mass.

To show that $\chi\rpot_m$ is also $\mu_n$-integrable, recall that $\rpot_m = u_m-u_0$, where as above $u_m = u_{m,X}$ denotes the potential for $T_m-\omega_X$ on $X$.  Hence it will suffice to show for any $j,k\leq N$ that $\chi u_j$ is integrable with respect to $T_k\wedge_X T_*$.  Since $C \geq u_j \geq A\log\dist(\cdot,\ind(f^N)) + B$, and since $\chi$ is supported on coordinate balls about points in $\ind(f^N)$, it suffices to show that
$$
\int_{\norm{z}<r} \chi \log\norm{z} \,T_k\wedge_X T_* >-\infty,
$$
where $z$ is one of the local coordinates and $r>0$ is small.  To do this, we choose a regularizing sequence $(\log_\ell)\subset \psh(\omega)$ for $\log\norm{z}$ such that $\log_\ell(z) = \log\norm{z}$ on $\supp d\chi$. Then by monotone convergence
$$
\int \chi\log\norm{z}\,T_k\wedge_X T_* = \lim_{\ell\to\infty} \int \chi\log_\ell(z)\,T_k\wedge_X T_* = \lim_{\ell\to\infty} \int u_k \,dd^c(\chi\log_\ell\norm{z})\wedge_X T_*
+ O(1)
$$
since $\chi\log\norm{z}$ is integrable with respect to $\omega\wedge_X T_*$ by Proposition \ref{prop:boundedish}.  Since $\log_\ell(z) = \log\norm{z}$ is independent of $\ell$ on $\supp d\chi$, we have that
$$
dd^c(\chi\log_\ell\norm{z}) - \chi\,dd^c\log_\ell\norm{z}
$$ 
is smooth and independent of $\ell$, bounded above and below by fixed multiples of $\omega$.  Hence there exists $C > 0$ such that up to additive constants
\begin{eqnarray*}
\int \chi\log\norm{z}\,T_k\wedge_X T_* & = & \lim_{\ell\to\infty}\int \chi u_k \,dd^c\log_\ell\wedge_X T_*
   \geq C\lim_{\ell\to\infty}\int\chi\log\norm{z}\,dd^c\log_\ell\wedge_X T_* \\
   & = & C\lim_{\ell\to\infty}\int\chi\log_\ell\,dd^c\log\norm{z}\wedge_X T_*
   = C\int \chi\log\norm{z}\,dd^c\log\norm{z}\wedge_X T_*. 
\end{eqnarray*}
By Corollary \ref{cor:cvgceonpoles} and the assumption \eqref{indout} at the beginning of \S\ref{sec:currents_II}, local potentials for $T_*$ are finite at $z=0$, so we conclude that the last integral is finite by \cite[Theorem 3.6]{BeDi05b}.
\end{proof}

\begin{cor}
\label{cor:integrability}
For any $m\geq n\geq 0$ we have that $\rpot_m - \rpot_n$ is integrable with respect to $\mu_m-\mu_n = dd^c(\rpot_m-\rpot_n)\wedge T_*$, and
$$
\int (\rpot_m-\rpot_n)\,dd^c(\rpot_m-\rpot_n)\wedge T_* = \ddeg^{-n} \int \rpot_{m-n}\,dd^c\rpot_{m-n}\wedge T_*.
$$
\end{cor}

\begin{proof}
The first assertion is immediate from Proposition \ref{prop:integrable}.  For the second assertion, observe that
$$
dd^c(\rpot_m-\rpot_n) = T_m-T_n = \ddeg^{-n}f^{n*} (T_{m-n} - T_0) = \ddeg^{-n}dd^c (\rpot_{m-n}\circ f^n).
$$
Hence $\rpot_m-\rpot_n$ and $\ddeg^{-n} \rpot_{m-n}\circ f^n$ differ by a constant.
Since $dd^c(\rpot_m-\rpot_n)\wedge T_*$ is a signed measure
with net mass $0$ we have
$$
\int_{\rztO} (\rpot_m-\rpot_n)\,dd^c(\rpot_m-\rpot_n)\wedge T_* = \ddeg^{-2n} \int_{\rztO} (\rpot_{m-n}\circ f^n) \,dd^c (\rpot_{m-n}\circ f^n) \wedge T_*.
$$
The proof will therefore be complete once we show that
$$
\int_{\rztO} (\rpot_{m-n}\circ f^n) \,dd^c (\rpot_{m-n}\circ f^n) \wedge T_*  = \ddeg^n\int_{\rztO} \rpot_{m-n}\,dd^c\rpot_{m-n}\wedge T_*.
$$
Theorem \ref{thm:cancelation} tells us that the integrands in this equation have no mass on any curve in $\rztO$.  So given $\epsilon >0$ we can choose a relatively compact open set $U\subset \torus \setminus f^n(\ind(f^n))$ such that $\int_{\rztO\setminus U} \rpot_{m-n}\,dd^c\rpot_{m-n}\wedge T_* < \epsilon$.  By Corollary \ref{COR:PROPER_AND_COVER} $f^n:\torus \setminus f^{-n}(f^n(\ind(f^n)) \to \torus \setminus f^n(\ind(f^n))$ is a finite holomorphic covering map, so we have that $f^{-n}(U)$ is a relatively compact open subset of $\torus \setminus f^{-n}(f^n(\ind(f))$, and by increasing $U$ if necessary that
$$
\int_{\torus \setminus f^{-n}(U)} (\rpot_{m-n}\circ f^n) \,dd^c (\rpot_{m-n}\circ f^n) \wedge T_* < \epsilon.
$$
Let $\chi:\torus \to[0,1]$ be a smooth compactly supported function such that
$\chi\equiv 1$ on $U$.  Since $\rpot_{m-n}$ is continuous on $\torus$,
we can choose a smooth function $\tilde\rpot:U\to \R$ such
that $|\tilde \rpot - \rpot_{m-n}|<\epsilon$ on $U$.  Thus
\begin{eqnarray*}
\int_{\torus} (\rpot_{m-n}\circ f^n) \,dd^c (\rpot_{m-n}\circ f^n) \wedge T_*
& \approx_{3\epsilon} &
\int_{\torus} ((\chi\tilde\rpot)\circ f^n) \,dd^c (\rpot_{m-n}\circ f^n) \wedge T_* \\
& = &
\int_{\torus} (\rpot_{m-n}\circ f^n) \,dd^c((\chi\tilde\rpot)\circ f^n) \wedge T_* \\
& \approx_{2\epsilon} &
\int_{\torus} (\tilde\rpot \circ f^n) \,dd^c((\chi\tilde\rpot) \circ f^n) \wedge T_* 
=
\ddeg^n\int_{\torus} \tilde\rpot\,dd^c(\chi \tilde\rpot)\wedge T_* \\
& \approx_{2\epsilon} &
\ddeg^n \int_{\torus} \rpot_{m-n}\,dd^c(\chi \tilde\rpot)\wedge T_* 
=
\ddeg^n \int \chi\tilde\rpot\,dd^c\rpot_{m-n}\wedge T_* \\
& \approx_{3\epsilon} &
\ddeg^n \int_{\torus} \rpot_{m-n}\,dd^c\rpot_{m-n}\wedge T_*,
\end{eqnarray*}
where $\approx_\delta$ means that the two sides differ by at most $\pm\delta$.  We have used that the signed measure $(T_m-T_n)\wedge T_*$ has total mass $2$ in each occurrence of `$\approx$'.  We have also used the invariance $f^n_*T_* = \ddeg^n T_*$ on the right side of the second `$=$'. the proof concludes on letting $\epsilon\to 0$.
\end{proof}

Even though Corollary \ref{cor:integrability} guarantees that $\int -(\rpot_m-\rpot_n)\,dd^c(\rpot_m-\rpot_n)\wedge T_*$ is well-defined and finite, it does not mean that $\rpot_m-\rpot_n$ has finite $T_*$-energy (which would then be given by the same integral).  That is, because $f^{m*}\bar T^*$, $f^{n*}\bar T^*$ and $\bar T^*$ all represent K\"ahler classes in $\hoo(\rztO)$, their restrictions all have positive Lelong numbers at the $\torus$-invariant points of any given toric surface $X$.  Hence it is not clear whether we can express $\rpot_m-\rpot_n$ in $X$ as a difference of qpsh functions in $\eclass(T_{*,X})$.  Here we prove something a little more modest.

\begin{thm}
\label{thm:finiteenergy1}  For any $m,n\geq 0$, the function $\rpot_m-\rpot_n$ has weakly finite $T_*$-energy
\begin{eqnarray}
\label{eqn:wknrgbd}
\dnorm[T_*]{\rpot_m-\rpot_n} = \left(\int_\torus - (\rpot_m-\rpot_n)\,dd^c(\rpot_m-\rpot_n)\wedge T_*\right)^{1/2}.
\end{eqnarray}
\end{thm}

\proof
We write $\Delta = \rpot_m-\rpot_n$ throughout the proof.  Then
$$
dd^c\Delta\wedge T_* = (T_m-T_n)\wedge T_* =\mu_m-\mu_n,
$$
and Proposition \ref{prop:integrable} guarantees that $\Delta$ is integrable with respect to $dd^c\Delta\wedge T_*$.  Fix a relatively compact neighborhood $U$ of $\ind(f^m)$ in $\rztO$.  Assume $\overline{U}$ meets only those poles that contains points of $\ind(f^m)$.  Let $M := \sup_{\rztO\setminus U}|\Delta|$ which is finite by Theorem \ref{thm:pullback} Part~(1).

\begin{lem}
\label{lem:dnormbndonx}
Let $X$ be a toric surface such that $\ind(f^m)\subset X^\circ$.  Then $\Delta$ has weakly finite $T_{*,X}$-energy, and
$$
\left|(\dnorm[T_{*,X}]{\Delta})^2 - \int -\Delta\,dd^c\Delta\wedge T_*\right| \leq 2M\sum_{p\in X\setminus X^\circ}  (\bar T^*\wedge_X T_*)(p).
$$
\end{lem}

Taking this for granted momentarily, we note that since $T_*$ has a nearly homogeneous support function, Theorem \ref{thm:wedgewhomogeneous}, Part~(4) tells us that the sum on the right decreases to $0$ as $X$ increases to $\rztO$.  Theorem \ref{thm:finiteenergy1} follows.
\qed

\begin{cor}
\label{cor:approxmu}
For any $n\in\N$, the currents $T_n$ and $T_*$ admit a wedge product in the sense of Definition \ref{DEFN:WEDGE_PRODUCT_TORIC_CURRENTS}.  The resulting positive Borel measure $T_n\wedge T_*$ on $\rztO$ has (full) mass equal to $1$ and does not charge curves.
\end{cor}

\begin{proof}
The argument is identical to the above proof of Theorem~\ref{thm:eqmeasureexists}, except that Theorem~\ref{thm:finiteenergy1} takes the place of Theorem \ref{thm:weakenergy1}.
\end{proof}

\begin{proof}[Proof of Lemma \ref{lem:dnormbndonx}]
The hypothesis on $X$ implies that $U$ is a relatively compact open subset of $X^\circ$.  Hence we may choose coordinate neighborhoods $U_p$ of the $\torus$-invariant points $p\in X\setminus X^\circ$ whose closures are disjoint from each other and from $\overline{U}$.  For $r>0$ small, we let $B_p(r)\subset U_p$ denote the coordinate ball of radius $r$ about $p$ and set $B(r) := \bigcup_{p\in X\setminus X^\circ} B_p(r)$.

As above, we write $T_j = \omega_X + dd^c u_j$, on $X$.  For each (large) positive integer $\ell$, we let $u_{j,\ell}\geq u_j$ be equal to $\max\{u_j,-\ell\}$ on each $U_p$ and equal to $u_j$ elsewhere.  Again, since $T_j$ represents a K\"ahler class in $\hoo(\rztO)$, Proposition \ref{prop:smalllelong} tells us that $u_j(p)$ has a non-trivial logarithmic singularity at $p$, so that $u_{j,\ell}\in \psh(\omega_X)$ is continuous off $\ind(f^m)$.  Corollary \ref{cor:logsings} and the fact that local potentials for $T_*$ are finite on $\ind(f^n)$ further imply that $u_{j,\ell}$ has finite $T_{*,X}$-energy
$$
(\norm[T_{*,X}]{u_{j,\ell}})^2 = \int - u_{j,\ell}\,dd^c u_{j,\ell}\wedge_X T_*.
$$
Set $\Delta_\ell := u_{m,\ell}-u_{n,\ell}$.  Then for any fixed $r>0$, we have that $\Delta = \Delta_\ell$ outside $B(r)$ when $\ell$ is large enough.  Hence
$$
\left|\int \Delta_\ell\,dd^c\Delta_\ell\wedge_X T_* - \int\Delta\,dd^c\Delta\wedge_X T_*\right|
\leq
\left|\int_{B(r)} \Delta\,dd^c\Delta \wedge_X T_*\right|
    + \sum_{p\in X\setminus X^\circ} \left|\int_{B_p(r)} \Delta_\ell\,dd^c\Delta_\ell\wedge_X T_*\right|.
$$
To estimate the integrals in the sum, we set $T_{j,\ell} := \omega_X + dd^c u_{j,\ell} \in \pcc^+(X)$ and proceed.
\begin{eqnarray*}
\left|\int_{B_p(r)} \Delta_\ell\, dd^c\Delta_\ell\wedge_X T_* \right|
& = &
\left|\int_{B_p(r)} \Delta_\ell\, (T_{m,\ell}-T_{n,\ell})\wedge_X T_* \right|
\leq \int_{B_p(r)} |\Delta_\ell| \,(T_{m,\ell}+T_{n,\ell})\wedge_X T_*  \\
& \leq & M \int_{B_p(r)} (T_{m,\ell}+T_{n,\ell})\wedge_X T_*  
= M \int_{B_p(r)} (T_m+T_n)\wedge_X T_* .
\end{eqnarray*}
The final equality follows from Stokes' Theorem and the fact that $T_{j,\ell} = T_j$ outside a compact subset of $B_p(r)$ for $j=m,n$.  So for $\ell=\ell(r)$ large 
enough,
$$
\left|\int \Delta_\ell\,dd^c\Delta_\ell\wedge_X T_* - \int\Delta\,dd^c\Delta\wedge_X T_*\right|
\leq
\left|\int_{B(r)} \Delta \,dd^c\Delta \wedge_X T_*\right|
+
M\sum_{p\in X\setminus X^\circ}\left|\int_{B_p(r)} (T_m+T_n)\wedge_X T_* \right|.
$$
As $\ell\to \infty$, we may let $r\to 0$ and apply Theorem \ref{thm:cancelation} and Proposition \ref{prop:integrable} to show that the first term goes to zero.  Thus
\begin{eqnarray*}
\limsup_{\ell\to \infty}
\left|\int \Delta_\ell\,dd^c\Delta_\ell\wedge_X T_* - \int\Delta\,dd^c\Delta\wedge_X T_*\right|
& \leq &
M\sum_{p\in X\setminus X^\circ}
(T_m+T_n)\wedge_X T_*(p) \\
& = & 2M\sum_{p\in X\setminus X^\circ}
\bar T^*\wedge_X T_*(p).
\end{eqnarray*}
The final equality comes from Theorem \ref{thm:cancelation}; i.e. for each $\torus$-invariant $p\in X$ and $j\geq 0$, we have $T_j\wedge_X T_*(p) = T_0\wedge_X T_*(p) = \bar T^* \wedge_X T_*(p)$.

Now if $\kappa\in C^\infty(X)$, we integrate by parts and apply Monotone convergence to $u_{m,\ell}$ and $u_{n,\ell}$ separately to obtain
$$
\left|\int \kappa\,dd^c\Delta\wedge_X T_*\right|^2
=
\lim_{\ell\to\infty} \left|\int \Delta_\ell \,dd^c\kappa\wedge_X T_*\right|^2
\leq
\norm[T_*]{\kappa}^2 \liminf_{\ell\to\infty} \norm[T^*]{\Delta_\ell}^2.
$$
Hence $\Delta$ has weakly finite $T_{*,X}$ energy, and
$$
\dnorm[T_{*,X}]{\Delta}^2  \leq \liminf_{\ell\to\infty} \norm[T^*]{\Delta_\ell}^2 \leq \int -\Delta\, dd^c\Delta \wedge_X T_* + 2M\sum_{p\in X\setminus X^\circ}  (\bar T^*\wedge_X T_*)(p).
$$
For the complementary bound, we observe that Theorem \ref{thm:regularization} and Corollary \ref{cor:decapprox2} allow us to trade our approximations $u_{j,\ell}$ for \emph{smooth} approximations $\tilde u_{j,\ell}\in\psh(\omega_X)$ such that 
$
\lim_{\ell\to\infty} \norm[T_{X,*}]{u_{j,\ell} - \tilde u_{j,\ell}} = 0. 
$
By Hartog's Lemma (see also Corollary \ref{cor:decrlims2}) we can further assume that $\tilde u_{j,\ell}$ decreases pointwise to $u_j$.
Hence if $\tilde\Delta_\ell = \tilde u_{m,\ell}-\tilde u_{n,\ell}$, we have that
\begin{align}\label{EQN:TILDEDELTA_L_APPROX_DELTA_L}
\lim_{\ell\to\infty} \norm[T_{X,*}]{\Delta_\ell - \tilde\Delta_\ell} = 0.
\end{align}
So by Monotone convergence again,
$$
\int -\Delta\,dd^c\Delta\wedge_X T_* = \lim_{\ell\to\infty} \int -\tilde\Delta_\ell\,dd^c\Delta\wedge_X T_* \leq \dnorm[T_{*,X}]{\Delta} \liminf_{\ell\to\infty} \norm[T_{*,X}]{\tilde\Delta_\ell}
\leq
\dnorm[T_{*,X}]{\Delta} \liminf_{\ell\to\infty} \norm[T_{*,X}]{\Delta_\ell}.
$$
Squaring both sides and substituting our above upper bound for the limit gives
$$
\left(\int -\Delta\,dd^c\Delta\wedge_X T_*\right)^2 \leq (\dnorm[T_{*,X}]{\Delta})^2\left(\int -\Delta\,dd^c\Delta\wedge_X T_* + 2M\sum_{p\in X\setminus X^\circ}  (\bar T^*\wedge_X T_*)(p)\right).
$$
This implies 
\begin{align}\label{EQN:DELTA_WEAK_ENERGY_LOWER_BND}
(\dnorm[T_{*,X}]{\Delta})^2 \geq \int -\Delta\,dd^c\Delta\wedge_X T_* - 2M\sum_{p\in X\setminus X^\circ}  (\bar T^*\wedge_X T_*)(p).
\end{align}
Since $\int -\Delta\,dd^c\Delta \wedge_X T_* = \int_\torus -\Delta \,dd^c\Delta\wedge T_*$ does not depend on $X$, the proof is complete.
\end{proof}

Before continuing we record a related technical result needed to prove Corollary \ref{COR:WEAK_ENERGY_BND_TAIL} and then in turn Theorem \ref{thm:geometricIntersection} below.  Here $\Delta$ is as in the proof of Theorem \ref{thm:finiteenergy1}.

\begin{prop}
\label{prop:technical}
Let $S\leq T_*$ be a positive closed $(1,1)$ current on an open subset $U\subset \torus$ and $\chi:U\to \R$ be a smooth compactly supported  function.  Then
$$
\int \chi\,dd^c\Delta \wedge (T_*-S) 
\leq 
\dnorm[T_*]{\Delta} \left( \int d\chi\wedge d^c\chi\wedge (T_*-S)\right)^{1/2}.
$$
\end{prop}

\begin{proof} 
Given $\epsilon>0$, Corollary \ref{cor:approxmu} yields a toric surface $X$ with $2M\sum_{p\in X\setminus X^\circ}  (\bar T^*\wedge_X T_*)(p) < \epsilon$.  Let $\tilde\Delta_\ell$ denote the smooth approximation of $\Delta$ on $X$ used to prove Lemma \ref{lem:dnormbndonx}.  Then integration by parts, monotone convergence and Cauchy-Schwarz
give
\begin{eqnarray*}
\left|\int \chi\,dd^c\Delta \wedge (T_*-S)\right|
& = &  
\lim_{\ell\to\infty} \left|\int \tilde \Delta_\ell \, dd^c\chi\wedge(T_*-S)\right|
=
\lim_{\ell\to\infty} \left|\int -d\tilde\Delta_\ell \wedge d^c\chi\wedge_X (T_*-S)\right| \\
& \leq &
\liminf_{\ell\to\infty}  \left(\int d\tilde\Delta_\ell\wedge d^c\tilde\Delta_\ell \wedge_X (T_*-S)\right)^{1/2}\left(\int d\chi\wedge d^c\chi \wedge_X (T_*-S)\right)^{1/2} \\
& \leq & 
\liminf_{\ell\to\infty} \norm[T_{*,X}]{\tilde\Delta_\ell} \left(\int d\chi\wedge d^c\chi \wedge_X (T_*-S)\right)^{1/2}.
\end{eqnarray*}
Using Equations (\ref{EQN:TILDEDELTA_L_APPROX_DELTA_L}) and (\ref{EQN:DELTA_WEAK_ENERGY_LOWER_BND}), we have 
$
\dnorm[T_{*,X}]{\Delta}^2\geq \liminf \norm[T_{*,X}]{\tilde\Delta_\ell}^2 - \epsilon.
$ 
Hence
\begin{eqnarray*}
\left|\int \chi\,dd^c\Delta \wedge (T_*-S) \right|^2 
& \leq & 
\left(\int d\chi\wedge d^c\chi \wedge (T_*-S)\right) \left( \dnorm[T_{*,X}]{\Delta}^2 + \epsilon\right) \\
& \leq & 
\left(\int d\chi\wedge d^c\chi \wedge (T_*-S)\right) \left( \dnorm[T_*]{\Delta}^2 + \epsilon\right). 
\end{eqnarray*}
Since $\epsilon>0$ is arbitrary, the proof is complete.
\end{proof}

We finally show that the series \eqref{eqn:series} defining $\rpot_{T^*}$ converges in the seminorm $\dnorm[T_*]{\cdot}$.

\begin{proof}[Proof of Theorem \ref{thm:weakenergy1}]
Fix a toric surface $X$ and a test function $\kappa\in C^\infty(X)$.  Then $dd^c\kappa\wedge T_*$ is a signed Borel measure with finite total mass.  Hence
\begin{eqnarray*}
\left|\int \kappa\,dd^c\rpot_{T^*}\wedge T_*\right| & = & \left|\int \rpot_{T^*} \, dd^c \kappa\wedge T_*\right| \leq
\sum_{j\geq 0} \left|\int (\rpot_{j+1}-\rpot_j)\,dd^c\kappa \wedge T_*\right|\\
& \leq & \norm[T_*]{\kappa}\sum_{j\geq 0} \dnorm[T_*]{\rpot_{j+1}-\rpot_j} \\
& = & \norm[T_*]{\kappa}\sum_{j\geq 0} \left(\int -(\rpot_{j+1}-\rpot_j)\,dd^c(\rpot_{j+1}-\rpot_j)\wedge T_*\right)^{1/2} \\
& = & \norm[T_*]{\kappa}\sum_{j\geq 0} \ddeg^{-j/2}\left(\int -\rpot_1\,dd^c\rpot_1\wedge T_*\right)^{1/2} \\
& = & \norm[T_*]{\kappa}\frac{\left(\int -\rpot_1\,dd^c\rpot_1\wedge T_*\right)^{1/2}}{1-\ddeg^{-1/2}}.
\end{eqnarray*}
The equality in the third line follows from Theorem \ref{thm:finiteenergy1} and the next 
equality follows from Corollary~\ref{cor:integrability}.  In any case, we conclude that 
\begin{equation}
\label{eqn:weaknrgbd}
\dnorm[T_*]{\rpot_{T^*}} \leq E_{T^*} := \frac{\dnorm[T_*]{\rpot_{\ddeg^{-1} f^*\bar T^*}}}{1-\ddeg^{-1/2}} < \infty.
\end{equation}
\end{proof}

Adapting the proof of Theorem \ref{thm:weakenergy1} gives further helpful bounds involving the same constant.

\begin{cor}\label{COR:WEAK_ENERGY_BND_TAIL} The following hold for any $n\geq 0$. 
\begin{enumerate}
 \item $\norm[T_*]{\rpot_{T^*} - \rpot_{\ddeg^{-n} f^{n*} \bar T^*} }' \leq E_{T^*} \ddeg^{-n/2}$;
 \item if $U$, $S$, and $\chi$ are as in Proposition \ref{prop:technical}, then
$$
\left|\int \chi\,(T^*-\ddeg^{-n} f^{n*}\bar T^*)\wedge (T_*-S)\right| 
\leq 
E_{T^*}\ddeg^{-n/2}\left(\int d\chi\wedge d^c\chi\wedge (T_*-S)\right)^{1/2}.
$$
\end{enumerate}
\end{cor}

\begin{proof}
The first assertion follows from writing $\rpot_{T^*} - \rpot_{\ddeg^{-n} \bar T^*} = \sum_{j=n}^\infty \rpot_{j+1}-\rpot_j$ and repeating the argument used to bound $\dnorm[T_*]{\rpot_{T^*}}$.  The second assertion follows from similar estimation and Proposition \ref{prop:technical}: 
\begin{eqnarray*}
\left|\int \chi\,(T^* - \ddeg^{-n} f^{n*}\bar T^*)\wedge (T_*-S)\right| 
& = & 
\left|\int (\rpot_{T^*} - \rpot_{\ddeg^{-n} f^{n*}\bar T^*})\,dd^c\chi\wedge (T_*-S)\right| \\
& \leq & 
\sum_{j=n}^\infty \left|\int(\rpot_{j+1}-\rpot_j)\,dd^c\chi\wedge (T_*-S)\right| \\
& \leq & 
\left(\int d\chi\wedge d^c\chi\wedge (T_*-S)\right)^{1/2} \sum_{j=n}^\infty \dnorm[T_*]{\rpot_{j+1}-\rpot_j} \\
& = &
E_{T^*} \ddeg^{-n/2} \left(\int d\chi\wedge d^c\chi\wedge (T_*-S)\right)^{1/2}.
\end{eqnarray*}
as in the proof of Theorem \ref{thm:weakenergy1}
\end{proof}

\subsection{Symmetry between the equilibrium currents}

We can reverse the roles of $T^*$ and $T_*$ in nearly all of the above.  The only ingredient peculiar to $T^*$ is the continuity result Theorem \ref{THM:CONTINUITY}, but we used this only to establish admissibility of the product $T^*\wedge T_*$.  Since all the other results we used are symmetric in $T^*$ and $T_*$, we obtain among other things analogs of Theorems \ref{thm:weakenergy1} and \ref{thm:finiteenergy1} as well as Corollaries \ref{cor:approxmu} and \ref{COR:WEAK_ENERGY_BND_TAIL}.

\begin{thm} 
\label{thm:weakenergy2}
The function $\rpot_{T_*}$ has weakly finite $T^*$-energy 
$$
\dnorm[T^*]{\rpot_{T_*}} \leq E_{T_*} := \frac{\dnorm[T^*]{\rpot_{\ddeg^{-1} f_* \bar T_*}}}{1-\ddeg^{-1/2}}.
$$
\end{thm}

The proof is the same as the one given for Theorem \ref{thm:weakenergy1} except that one appeals to Theorem \ref{thm:decapprox1} (in addition to Theorem \ref{THM:CONTINUITY}) to argue that $dd^c\rpot_{T_*,X}$ admits a wedge product with $T^*_X$ (see Definitions \ref{defn:DPSH_and_ADMIT_WEDGE} and \ref{defn:WEAKLY_FINITE_ENERGY}).

\begin{thm}
\label{thm:finiteenergy2}
Let $\rpot_n$ denote the potential for $\ddeg^{-n} f^n_* \bar T_* - \bar T_*$.  Then for any $n,m\geq 0$, the function $\rpot_n-\rpot_m$ has weakly finite $T^*$-energy satisfying
$$
\dnorm[T^*]{\rpot_n-\rpot_m} = \left(\int -(\rpot_n-\rpot_m)\,dd^c(\rpot_n-\rpot_m)\wedge T^*\right)^{1/2},
$$
where the integral on the right is well-defined and finite.
\end{thm}

\begin{cor}
\label{cor:approxmu2}
For any $n\in\N$ the toric currents $f^n_* \bar T_*$ and $T^*$ admit a wedge product.  The positive Borel measure $f^n_* \bar T_*\wedge T^*$ has (full) mass equal to $1$ on $\rztO$ and does not charge curves.  
\end{cor}

The proofs of Theorem \ref{thm:finiteenergy2} and Corollary \ref{cor:approxmu2} are essentially identical to those of Theorem \ref{thm:finiteenergy1} and Corollary \ref{cor:approxmu}.

%% file: moremeasure3.tex
In this section we investigate the dynamical and geometric properties of the measure $\mu = T^*\wedge T_*$ from Theorem \ref{thm:eqmeasureexists}, proving Conclusions (2) and (3) in Theorem \ref{thm:mainthm} among other things.  As before, $f$ denotes a toric map satisfying the hypotheses of Theorem \ref{thm:mainthm}.

\begin{prop}
$\mu$ is $f$-invariant.
\end{prop}

\begin{proof}
It will suffice to show that $\int (\kappa \circ f) \,d\mu = \int \kappa\,d\mu$ for
any smooth, compactly supported function $\kappa:\rzt^\circ\to \R$.  Since
$\mu$ does not charge curves, we can further assume that $\supp\kappa \subset
\torus \setminus f(\ind(f))$.  Hence $\kappa\circ f$ is also smooth and
compactly supported in $\torus\setminus f^{-1}(f(\ind(f))$.  Let $K\subset
\torus$ be a compact set containing both $\supp\kappa$ and $\supp\kappa\circ f$
in its interior.  Recall that $T^*|_{\torus} = dd^c g$ on $\torus$ where $g =
\sfn_{\bar T^*}\circ \Log + \rpot_{T^*}$ is continuous by Theorem \ref{THM:CONTINUITY}.  Hence
\begin{eqnarray*}
\int_{\rztO} (\kappa\circ f)\,d\mu 
& = & 
\int_{\torus\setminus f^{-1}(f(\ind(f)))} (\kappa \circ f)\, \frac{f^*T^*}{\ddeg} \wedge T_* 
= 
\frac{1}{\ddeg} \int_{\torus\setminus f^{-1}(f(\ind(f)))} (\kappa\circ f)\, dd^c (g\circ f)\wedge T_* \\
& = & 
\frac{1}{\ddeg} \int_{\torus\setminus f^{-1}(f(\ind(f)))} (g\circ f)\,dd^c(\kappa\circ f)\wedge T_*
=
\int_{\torus\setminus f(\ind(f))} g\,dd^c\kappa\wedge \frac{f_*T_*}{\ddeg} \\
& = & \int_{\torus\setminus f(\ind(f))} \kappa \, dd^c g \wedge T_* = \int_{\rztO} \kappa\,d\mu.
\end{eqnarray*}
The fourth equality holds because $T_*$ is a current of order zero; hence pushing forward $T_*$ by the holomorphic covering $f:\torus\setminus f^{-1}(f(\ind(f)))\to \torus\setminus f(\ind(f))$ is adjoint to pulling back continuous forms with compact support.
\end{proof}

\subsection{Mixing}
Next we set about proving that $\mu$ is mixing, closely following a strategy originating in \cite{BeSm91} and refined in \cite{Sib99} (see particularly the proof and application of Corollary 2.2.13) and elsewhere. The particular context here leads to some technical modifications of that argument, so we nevertheless give the full proof.

\begin{thm}
\label{thm:ergodiccurrent}
Let $\alpha$ be a smooth function with compact support on $\rzt^\circ$.  Then we have weak convergence
$$
\lim_{n\to\infty} (\alpha\circ f^n)\,T^* = T^*\cdot\int\alpha\,\mu.
$$
\end{thm}

Before beginning the proof let us observe that since $T^*$ does not charge points, $S_n:=(\alpha\circ f^n) T^*$ is a well-defined (non-closed) $(1,1)$ current of order $0$, acting on continuous $(1,1)$ forms $\beta$ via
$$
\pair{\beta}{S_n} := \int_{\rzt^\circ} (\alpha\circ f^n) \,\beta\wedge T^*.
$$  
By Corollary \ref{COR:PROPER_AND_COVER}, $f:\rzt^\circ\tto\rzt^\circ$ is proper, so $S_n$ is compactly supported.  We may assume that $0\leq \alpha \leq 1$ so that $0\leq S_n \leq T^*$ is positive.  Since $T^*$ has finite total mass in any toric surface, the sequence $(S_n)_{n\in\N}$ has uniformly bounded mass on any compact subset of $\rzt^\circ$.  Hence $(S_n)$ is pre-compact in the weak topology.  We let $\mathcal{S}$ denote the set of all its limit points.

\begin{lem}
\label{lem:limitclosed}
Every element of $\mathcal{S}$ is closed.
\end{lem}

\begin{proof}
Any current $S\in \mathcal{S}$ is positive and bounded above by $T^*$.  In particular $S$ does not charge curves in $\rzt^\circ$.  So in order to show $S$ is closed, it suffices by the Skoda-El Mir Theorem and the fact that $T^*$ has locally finite mass on $\rzt^\circ$ to show only that $S$ is closed in $\torus$.

Fix a smooth real and compactly supported $1$-form $\gamma$ on $\torus$.  Since $\alpha\circ f^n$ is smooth on $\torus$, we can estimate using Schwarz's inequality:
\begin{eqnarray*}
\left|\int \gamma\,d((\alpha\circ f^n) T^*)\right|^2 
\leq  
\left(\int \gamma\wedge J\gamma\wedge T^*\right)
\left(\int d(\alpha\circ f^n)\wedge d^c(\alpha\circ f^n) \wedge T^*\right),
\end{eqnarray*}
where $J$ denotes the complex structure operator on the real cotangent bundle of $\rzt^\circ$.  Additionally, since $T^*$ does not charge $f^n(\ind(f^n))$ and 
$f^n:\torus\setminus f^{-n}(f^n(\ind(f^n))) \to \torus\setminus f^n(\ind(f^n))$ 
is a finite degree holomorphic covering (Corollary \ref{COR:PROPER_AND_COVER}), we have
\begin{eqnarray*}
\int d(\alpha\circ f^n)\wedge d^c(\alpha\circ f^n) \wedge T^* 
& = &
\frac{1}{\ddeg^n}\int_{\torus\setminus f^{-n}(f^n(\ind(f^n)))} f^{n*} (d\alpha\wedge d^c\alpha\wedge T^*) \\
& = &
\frac{\dtop^n}{\ddeg^n} \int_{\torus\setminus f^n(\ind(f^n))} d\alpha\wedge d^c\alpha\wedge T^*
= 
\frac{\dtop^n}{\ddeg^n} \norm[T^*]{\alpha},
\end{eqnarray*}
where $\norm[T^*]{\alpha}$ is defined by \eqref{eqn:toricenergy}.  All told, we see that 
$$
\int \gamma\wedge d((\alpha\circ f^n)\wedge T^*) \leq C \left(\frac{\dtop}{\ddeg}\right)^{n/2} 
\underset{n\to\infty}\longrightarrow 0.
$$
It follows that any current $S\in\mathcal{S}$ is closed in $\torus$, hence as explained above, also closed in $\rzt^\circ$.
\end{proof}

\begin{lem}
\label{lem:limitinvariance}
$\ddeg^{-1}f^*\mathcal{S} = \mathcal{S}$.
\end{lem}

\begin{proof} 
Let $(S_{n_j})\subset (S_n)$ be a subsequence such that $S_{n_j}\to S$ and $S_{n_j+1}\to S'$.  It suffices to show that $f^* S = \ddeg S'$, i.e. that $\int \beta \wedge f^* S = \ddeg \int \beta\wedge S'$ for any smooth compactly supported $(1,1)$ form $\beta$ on $\rzt^\circ$.  Since $S$, $S'$ and $f^*S$ are all dominated by $\ddeg T^*$, none of them charges curves.  So we can further assume that $\beta$ is compactly supported in $\torus\setminus f^{-1}(f(\ind(f))$.  Corollary \ref{COR:PROPER_AND_COVER} then implies that $f_*\beta$ is smooth and compactly supported in $\torus\setminus f(\ind(f))$.  Hence
\begin{eqnarray*}
\int_{\rzt^\circ} \beta \wedge f^*S 
& = &
\int_{\torus \setminus f^{-1}(f(\ind(f)))} \beta \wedge f^*S
=
\int_{\torus \setminus f(\ind(f))} f_*\beta \wedge S \\
& = &
\lim_{j\to\infty} \int_{\torus \setminus f(\ind(f))} f_*\beta \wedge S_{n_j}
=
\ddeg \lim_{j\to\infty}\int_{\torus \setminus f^{-1}(f(\ind(f)))} \beta \wedge S_{n_j+1} \\
& = &
\ddeg \int_{\torus \setminus f^{-1}(f(\ind(f)))} \beta \wedge S' = \ddeg \int_{\rztO} \beta\wedge S'.
\end{eqnarray*}
\end{proof}

\begin{lem}
\label{lem:limitmass}
Every $S\in\mathcal{S}$ satisfies $\isect{S}{T_*} = \int \alpha\,\mu$.  
\end{lem}
 
\begin{proof}  Fix $\epsilon>0$.  By Corollary \ref{cor:approxmu2} we can choose a smooth function $\chi: \rztO \to [0,1]$ that is supported in a compact subset of $\torus \setminus f^{-n}(f^n(\ind(f^n)))$ such that
\begin{align*}
\left| \int_{\rzt^\circ} (1-\chi\circ f^n) T^* \wedge \bar{T_*} \right| < \epsilon  \quad \mbox{and} \quad  \left|\int_{\rzt^\circ}  (1-\chi) T^* \wedge \left(\frac{f_*^n \bar{T_*}}{\ddeg^n}\right)\right| < \epsilon.
\end{align*}
Recall that $\bar T^* = dd^c u$, where $u = \sfn_{\bar T^*}\circ\Log$ is continuous on $\torus$ (and therefore bounded on $\supp\chi$).  So since $S_n \leq T^*$,
\begin{eqnarray*}
\int_{\rzt^\circ} S_n\wedge \bar T_*
& \approx_\epsilon &
\int_{\torus} (\chi\circ f^n) S_n\wedge \bar T_* 
= 
\int_{\torus} u \, \frac{f^{n*} T^*}{\ddeg^n}\wedge f^{n*} dd^c(\chi\alpha) 
= 
\int_{\torus} \frac{f^n_* u}{\ddeg^n}\, dd^c(\chi\alpha)\wedge T^* \\
& = & 
\int_{\torus} \chi\alpha\,T^*\wedge \frac{f^n_*\bar T_*}{\ddeg^n}
\approx_{\epsilon}
\int_{\rzt^\circ} \alpha\,T^*\wedge \frac{f^n_*\bar T_*}{\ddeg^n}.
\end{eqnarray*}
As in the proof of Corollary \ref{cor:integrability}, the symbol $\approx_\epsilon$ means that the two sides differ by at most $\pm\epsilon$.
Letting $\epsilon\to 0$, we obtain that
$$
\int_{\rzt^\circ} S_n\wedge \bar T_*
=
\int_{\rzt^\circ} \alpha\,T^*\wedge \frac{f^n_*\bar T_*}{\ddeg^n}.
$$
Letting $\rpot = \rpot_{\ddeg^{-1} f_* \bar T_*}$, we apply this formula and the analog of Corollary \ref{COR:WEAK_ENERGY_BND_TAIL} (1) with the roles of $T^*$ and  $T_*$ reversed to obtain
\begin{eqnarray*}
\left| \int_{\rzt^\circ} \alpha\,\mu - \int_{\rzt^\circ} S_n\wedge \bar T_* \right|
& = & 
\left|\int_{\rzt^\circ} \alpha\,\left(T_* - \frac{f^n_*\bar T_*}{\ddeg^n}\right)\wedge T^* \right| 
\leq  
\norm[T^*]{\rpot_{T_*} - \rpot_{\ddeg^{-n} f^{n}_* \bar T_*} }'  \norm[T^*]{\alpha}  \\  
& \leq & E_{T_*}\ddeg^{-n/2}\norm[T^*]{\alpha}
\underset{n\to\infty} \longrightarrow 0.
\end{eqnarray*}
The lemma follows immediately.  
\end{proof}

\begin{proof}[Proof of Theorem \ref{thm:ergodiccurrent}]
It suffices to show that $cT^*$ is the only limit point of $S_n$.  The central convergence result \cite[Theorem 10.1]{DiRo24} in our previous paper, together with Lemmas \ref{lem:limitclosed} and \ref{lem:limitmass}, tells us that for every $T\in\mathcal{S}$ we have
$$
\lim_{n\to\infty} \ddeg^{-n} f^{n*} T =  \isect{T}{T_*} T^* = cT^*.
$$
Though we did not note it there, the proof we gave for \cite[Theorem 10.1]{DiRo24} is actually uniform on compact subsets of $\pcc^+(\rzt)$.  That is, if $(T_n)\subset\pcc^+(\rzt)$ is a relatively compact sequence with $\tilde c = \isect{T_n}{T_*}$ independent of $n$, then we have the more general convergence
$$
\lim_{n\to\infty} \ddeg^{-n} f^{n*} T_n = \tilde c T^*.
$$
This amplification of \cite[Theorem 10.1]{DiRo24} holds because, first of all, when $(T_n)$ is relatively compact, so is the associated sequence $(\ddeg^{-n} f^{n*} T_n)$; and secondly, the constants in the key volume estimate \cite[Theorem 6.11]{DiRo24} used to prove \cite[Theorem 10.1]{DiRo24} are uniform on compact sets.  Indeed, even the proof we gave there depended on this uniformity. 

So given $S\in \mathcal{S}$, we apply Lemma \ref{lem:limitinvariance} to obtain a sequence $(T_n)\subset\mathcal{S}$ such that $S = \lim \ddeg^{-n}f^{n*} T_n$.  Since $T_n\leq T^*$ for all $n$, this sequence is relatively compact, and we obtain that
$$
S = \lim \ddeg^{-n} f^{n*} T_n = cT^*
$$
That is, $cT^*$ is the only element of $\mathcal{S}$.
\end{proof}

\begin{thm}
\label{thm:mixing}
The measure $\mu$ is mixing for $f$.
\end{thm}

\begin{proof}
Since $\mu$ is a Borel measure with full, finite mass on $\torus$, it suffices to show for any smooth compactly supported functions $\alpha,\beta:\torus\to [0,1]$ that
$$
\lim_{n\to\infty} \int_{\torus} (\alpha\circ f^n)\beta \, \mu = \left(\int \alpha\,\mu\right)\left(\int \beta\,\mu\right).
$$
\begin{lem}
\label{lem:unifnrgbd}
For all $n\geq 0$, we have
$
\norm[T^*]{(\alpha\circ f^n)\beta} \leq \norm[T^*]{\beta} + \norm[T^*]{\alpha}.
$
\end{lem}

\begin{proof}
We have
\begin{align*}
\norm[T^*]{(\alpha\circ f^n)\beta}^2 &= \int_\torus (\alpha\circ f^n)^2 d\beta\wedge d^c\beta\wedge T^* +
2\int_\torus (\alpha \circ f^n) \beta \ d(\alpha\circ f^n)  \wedge d^c\beta\wedge T^* \\ &+ \int_\torus \beta^2 d(\alpha\circ f^n)\wedge d^c(\alpha\circ f^n) \wedge T^*.
\end{align*}
The first term is non-negative, bounded above by $\norm[T^*]{\beta}^2$.  We can re-write the last term and estimate it as follows
$$
\int_\torus \beta^2 d(\alpha\circ f^n)\wedge d^c(\alpha\circ f^n) \wedge \frac{f^{*n} T^*}{\ddeg^n}
=
\ddeg^{-n}\int_\torus (f^n_*\beta)^2 d\alpha\wedge d^c\alpha\wedge T^*
\leq
\left(\frac{\dtop}{\ddeg}\right)^n \norm[T^*]{\alpha}^2.
$$
For the second term, we choose a smooth compactly supported function $\chi:\torus\to [0,1]$ such that $\chi \equiv 1$ on $\supp\beta$ and then use the Schwarz inequality
$$
\left|\int_\torus (\alpha \circ f^n) \beta \ d\alpha\wedge d^c\beta\wedge T^*\right|^2
\leq \norm[T^*]{\beta}^2
\int_\torus \chi d(\alpha\circ f^n)\wedge d^c(\alpha\circ f^n)\wedge T^*
\leq \left(\frac{\dtop}{\ddeg}\right)^n \norm[T^*]{\beta}^2\norm[T^*]{\alpha}^2.
$$
Adding up these results and taking square roots we get
$$
\norm[T^*]{(\alpha\circ f^n)\beta} \leq \norm[T^*]{\beta} + \left(\frac{\dtop}{\ddeg}\right)^{n/2} \norm[T^*]{\alpha} \leq \norm[T^*]{\beta} + \norm[T^*]{\alpha}
$$
\end{proof}

\begin{lem}
\label{lem:tminusapprox}
Given $\epsilon>0$ and a relatively compact open set $U\subset \torus$, there exist a function $v\in\dpsh(\rzt^\circ)$, a K\"ahler form $\omega$ on $\torus$ and a continuous function $u\in\psh(\omega)$ such that
\begin{itemize}
\item $\dnorm[T^*]{v} <\epsilon$.
\item $T_*|_\torus = \omega + dd^c (u+v)$.
\item $|u| <\epsilon$ on $U$.
\end{itemize}
\end{lem}

\begin{proof}
Let $T_{*,n} = \ddeg^{-n} f^n_* \bar T_*$ for $n\geq 0$.  Then on the one hand, Corollary~\ref{COR:WEAK_ENERGY_BND_TAIL} (with $T^*$ and  $T_*$ reversed) tells us that for $n$ large enough $T_* - T_{*,n} = dd^c v$ for some $v\in\dpsh(\rzt^\circ)$ with weakly finite $T^*$ energy satisfying $\dnorm[T^*]{v} < \epsilon$.  On the other hand, we have
$$
T_{*,n}|_{\torus} = dd^c g_n,
$$
where $g_n = \ddeg^{-n}f^n_*\sfn_{\bar T_*} \in\psh(\torus)$ is continuous because $f^n(\exc(f))\cap\torus = \emptyset$.  Hence we can regularize $g_n$ on $\torus$ to obtain
$$
g_n = \tilde g + u
$$
where $\omega = dd^c\tilde g$ is K\"ahler on $\torus$ and $u\in\psh(\omega)$ satisfies $|u|<\epsilon$ on $U$.
\end{proof}

Continuing the proof of Theorem \ref{thm:mixing}, we choose $\epsilon>0$ and a relatively compact open $U\subset\torus$ containing $\supp\beta$.  With $u,v,\omega$ as in Lemma \ref{lem:tminusapprox}, we then have
$$
\int (\alpha\circ f^n)\beta \,\mu 
= 
\int (\alpha\circ f^n)\beta\, T^*\wedge (\omega + dd^c u + dd^c v).
$$
We expand the right side into three integrals and deal with each separately.
From Lemmas \ref{lem:unifnrgbd} and \ref{lem:tminusapprox}, we have
$$
\left|\int(\alpha\circ f^n)\beta\, T^*\wedge dd^c v\right| \leq \norm[T^*]{(\alpha\circ f^n)\beta}\cdot\dnorm[T^*]{v} \leq \epsilon(\norm[T^*]{\alpha} + \norm[T^*]{\beta}).
$$
From Lemma \ref{lem:tminusapprox} we further obtain
\begin{eqnarray*}
\left|\int(\alpha\circ f^n)\beta\, T^*\wedge dd^c u\right|
& = &
\left|\int u\, dd^c ((\alpha\circ f^n)\beta)\wedge T^* \right| \\
& \leq &
\left|\int u(\alpha\circ f^n)\, dd^c \beta\wedge T^*\right|
+
\left|\int u\beta\, dd^c (\alpha\circ f^n)\wedge T^*\right| +\\
& &
2\left|\int u \, d\beta\wedge d^c(\alpha\circ f^n)\wedge T^*\right| \\
& \leq &
\epsilon\norm[C^2]{\beta}
+
\ddeg^{-n}\left|\int f^n_*(u \beta)\, dd^c\alpha\wedge T_*\right|
+
2 \epsilon\norm[T^*]{\beta}^{1/2}\norm[T^*]{\alpha \circ f^n}^{1/2}  \\
& \leq &
\epsilon\norm[C^2]{\beta}
+
\epsilon\left(\frac{\dtop}{\ddeg}\right)^n \norm[C^2]{\alpha}
+
2 \epsilon \left(\frac{\dtop}{\ddeg}\right)^n \norm[T^*]{\beta}^{1/2}\norm[T^*]{\alpha}^{1/2} \\
& \leq & C\epsilon,
\end{eqnarray*}
where, since $\dtop<\ddeg$, the constant $C$ depends on $\alpha$ and $\beta$ but not on $n\geq 0$. 

For the remaining integral, Theorem \ref{thm:ergodiccurrent} gives
$$
\lim_{n\to\infty} \int (\alpha\circ f^n)\beta\, T^*\wedge\omega 
= 
\left(\int \beta\,T^*\wedge\omega\right)\left(\int \alpha\,\mu\right).
$$
But
\begin{eqnarray*}
\left|\int\beta\,T^*\wedge (T_*-\omega)\right|
& \leq &
\left|\int\beta\,T^*\wedge dd^c v \right|
+
\left|\int\beta\,T^*\wedge dd^c u \right| \\
\leq
\norm[T^*]{\beta}\cdot \dnorm[T^*]{v} + \sup_U |u|\cdot \norm[C^2]{\beta}
\leq C\epsilon.
\end{eqnarray*}
Letting $\epsilon \to 0$ completes the proof.
\end{proof}

\subsection{Geometric intersection}
\label{ss:geometric}
Now we briefly describe an alternative, more geometric formulation of the wedge product $T^*\wedge T_*$ that gives us the equilibrium measure $\mu$.  Our presentation and arguments are nearly the same as the ones in \cite[\S5]{Duj06} and \cite[\S4]{Duj05}, so we refer readers to those sources for full details.

A positive closed $(1,1)$ current $S$ in an open set $U\subset\C^2$ is \emph{uniformly laminar} if it can be written $S = \int [\Delta_\alpha]\,d\nu(\alpha)$ where $\{\Delta_\alpha\}$ is a family of mutually disjoint smooth analytic disks properly embedded in $U$ and $\nu$ is a positive measure on the underlying parameter space.  One says that a second positive closed $(1,1)$ current $T$ on $U$ is \emph{uniformly woven} if we have the same sort of representation $T = \int [\Delta'_\beta]\nu'(\beta)$ but drop the condition that the disks be disjoint from each other.  If $S\in L^1_{loc}(T)$ so that $S\wedge T$ is defined, then  \cite[Proposition 2.6]{DDG11} gives the following geometric reformulation:
\begin{align}\label{EQN:GEOMETRIC_WEDGE_PRODUCT}
S\wedge T = \int [\Delta_\alpha \cap \Delta'_\beta] \,\nu(\alpha)\otimes\nu'(\beta).
\end{align}
The integrability hypotheses implies that $\Delta_\alpha$ meets $\Delta_\beta$ properly for $\nu\otimes\nu'$ almost all $(\alpha,\beta)$, and $[\Delta_\alpha\cap \Delta'_\beta]$ denotes the discrete measure obtained by placing a point mass at each point of intersection between $\Delta_\alpha$ and $\Delta'_\beta$.  One says in this circumstance that the wedge product $S\wedge S'$ is \emph{geometric}.

The equilibrium currents $T^*$ and $T_*$ are not themselves uniformly laminar or woven, but they can be well approximated from below by such currents: i.e. in any toric surface $X$, we have that $T^*$ is \emph{laminar and strongly approximable} \cite[Theorem 10.5]{DiRo24}, and $T_*$ is \emph{woven and strongly approximable} \cite[Theorem 10.6]{DiRo24}.

To make this more precise, fix the surface $X$ and a K\"ahler form $\omega$ on $X$.
Then there exists a dense open set $U_\epsilon \subset X$, a uniformly laminar current $T^*_\epsilon \leq T^*$ on $U_\epsilon$ and a uniformly woven current $T_{*,\epsilon} \leq T_*$ on $U_\epsilon$ such that
$$
\int \omega\wedge_X (T^*-T^*_\epsilon), \int \omega\wedge_X (T_*-T_{*,\epsilon}) < \epsilon^2.
$$
Since $T^*$ doesn't charge analytic disks (Corollary \ref{cor:nodisk}), neither does $T^*_\epsilon$.  Hence the transverse measure $\nu$ in the laminar presentation of $T^*_\epsilon$ has no atoms.   This observation and the geometric formula for the wedge product (\ref{EQN:GEOMETRIC_WEDGE_PRODUCT}) make clear that the measure $T^*_{\epsilon}\wedge T_{*,\epsilon} \leq T^*\wedge_X T_*$ also has no atoms, even though the right side of the inequality has atoms at each torus invariant point of $X$.  In particular, we have the sharper bound $T^*_{\epsilon} \wedge T_{*,\epsilon} \leq (T^*\wedge_X T_*)|_{X^\circ} = T^*\wedge T_*$, since $X\setminus X^\circ$ is finite.

The open set $U_\epsilon$ is a (somewhat flexible) union of mutually disjoint polydisks `of size $\epsilon$'.  The flexibility in the choice of $U_\epsilon$ allows one to further ensure that for any $\delta>0$, there is a smooth compactly supported function $\chi = \chi_{\delta,\epsilon}:U_\epsilon\to [0,1]$ such that
\begin{enumerate}
 \item $\int_X (1-\chi)\,T^*_\epsilon\wedge T_{*,\epsilon} \leq \int_X (1-\chi)\,d\mu \leq \delta$;
 \item $\norm[\infty]{d\chi}^2, \norm[\infty]{dd^c\chi} \leq C_\delta \epsilon^{-2}$, where $C_\delta$ does not depend on $\epsilon$ or $U_\epsilon$;
\end{enumerate}

\noindent
Since $\mu$ has full mass on $\torus$, we can multiply $\chi$ by an $\epsilon$-independent cutoff to further arrange that $\supp\chi$ is contained in a fixed compact set $K\subset\torus$ without affecting properties (1) (especially) and (2).

\begin{thm}\label{thm:geometricIntersection}
The wedge product $\mu = T^*\wedge T_*$ is geometric on any toric surface $X$; that is, $T^*_\epsilon\wedge T_{*,\epsilon} \nearrow \mu$ as $\epsilon\to 0$.
\end{thm}

\begin{proof}
Fix a constant $\delta>0$ and for any $\epsilon>0$, let $\chi = \chi_{\delta,\epsilon}$ be as above.  Then
\begin{eqnarray*}
\int_X d\mu - T^*_\epsilon\wedge T_{*,\epsilon} & \leq & 2\delta + \int_{\torus} \chi \, (T^*\wedge T_* - T^*_\epsilon\wedge T_{*,\epsilon} )\\
& \leq & 2\delta + \int_{\torus} \chi T^*\wedge (T_* - T_{*,\epsilon}) + \int_\torus \chi\,(T^*-T^*_{\epsilon})\wedge T_*.
\end{eqnarray*}
It therefore suffices to show that each of the last two integrals tends to zero with $\epsilon$.   We show this only for the second integral.  The details for the first are similar.

Let $T_n = \ddeg^{-n} f^n_* \bar T$.  Then
$$
\int_\torus \chi \,T_*\wedge (T^*-T^*_\epsilon) 
=
\int_\torus \chi \,T_n\wedge (T^*-T^*_\epsilon) + \int_\torus \chi (T_*-T_n)\wedge(T^*-T^*_\epsilon).
$$
Corollary \ref{COR:WEAK_ENERGY_BND_TAIL} (with $T^*$ and $T_*$ switched) and properties (1) and (2) above allow us to bound the second term as follows.
\begin{eqnarray*}
\left|\int_\torus \chi (T_*-T_n)\wedge(T^*-T^*_\epsilon)\right|^2
& \leq &
E_{T_*}^2\ddeg^{-n} \int_\torus d\chi\wedge d^c\chi\wedge (T^*-T^*_\epsilon) \\
& \leq &
E_{T_*}^2\ddeg^{-n} \norm[\infty]{d\chi}^2 \int_\torus \omega\wedge (T^*-T^*_\epsilon)
\leq E_{T_*}^2 \cdot C_\delta \ddeg^{-n}.
\end{eqnarray*}
So given $\epsilon' > 0$ and fixed $n$ large enough, we have
for all $\epsilon>0$ that
$$
\int_\torus \chi T_*\wedge (T^*-T^*_\epsilon) \leq \int_\torus \chi T_n\wedge (T^*-T^*_\epsilon) + \epsilon'.
$$
We claim that the integral on the right tends to $0$ with $\epsilon$ (which completes the proof).  To see this recall that $T_n|_\torus = dd^c u$ for some continuous psh function $u$.  We can use e.g. the action of $\torus$ to regularize $u$ on a neighborhood of $\supp\chi$, obtaining a smooth function $v$ such that $\max |u-v|$ is as small as we like on $\supp\chi$.  Thus,
\begin{eqnarray*}
\int_\torus \chi\, T_n \wedge (T^*-T^*_\epsilon) 
& \leq & 
\int_\torus \chi\, dd^c v\wedge (T^*-T^*_\epsilon)
+ 
\int_\torus (u-v)\,dd^c\chi\wedge (T^*-T^*_\epsilon) \\
& \leq & 
\int_\torus \chi\, dd^c v\wedge (T^*-T^*_\epsilon)
+
(\max_{\supp\chi}|u-v|) \cdot \norm[\infty]{dd^c\chi} \int_\torus \omega\wedge (T^*-T^*_\epsilon) \\
& \leq & 
\int_\torus \chi\, dd^c v\wedge (T^*-T^*_\epsilon) + C_\delta \max_{\supp\chi} |u-v|.
\end{eqnarray*}
Since $\chi\,dd^c v$ is smooth and fixed, since $\chi \leq 1$ and since $T^*_\epsilon \nearrow T_*$ on $\supp\chi$, the first term on the right tends to $0$ with $\epsilon$.
Moreover, the second term on the right is independent of $\epsilon$ and can be made arbitrarily small, so our claim is proved.
\end{proof}

By now we have established all but the last conclusion of Theorem \ref{thm:mainthm}.  Conclusions (1) and (2) from that theorem and Theorem \ref{thm:geometricIntersection} are in fact the \emph{hypotheses} of \cite[Theorem B]{DDG11b}.  We can therefore invoke that result to infer the final conclusion of Theorem \ref{thm:mainthm}.

\begin{cor}
The metric entropy of $f$ with respect to $\mu$ is $\log\ddeg$.
\end{cor}